\documentclass[11pt]{amsart}
\usepackage{latexsym, amssymb, amsthm,color}
\usepackage[centertags]{amsmath}
\usepackage{pictexwd}
\usepackage[normalem]{ulem}
\definecolor{GREEN}{rgb}{0,1,0}
\definecolor{green4}{rgb}{.1,.5,.1}
 \newcommand{\ep}{\end{proof}}

 \newcommand{\sm}{\smallskip}

 \newif\ifpctex

  \newtheorem{theorem}{Theorem}
  \newtheorem{definition}{Definition}[section]

  \newtheorem{cond}[definition]{Condition}
  \newtheorem{proposition}[definition]{Proposition}
  \newtheorem{lemma}[definition]{Lemma}
  \newtheorem{cor}[definition]{Corollary}
  
  \newtheorem{corollary}[definition]{Corollary}
  \newcommand{\beCond}[2]{\Rand{\vspace{0,6cm}\tt #1}\begin{cond}[#2]
  \label{#1}} \theoremstyle{definition}
  \newtheorem{remark}[definition]{Remark}
  
  \numberwithin{equation}{section}
  \newtheoremstyle{step}{3pt}{0pt}{\itshape}{}{\bf}{}{.5em}{}

\theoremstyle{step} \newtheorem{step}{Step}

\newcommand{\R}{\mathbb{R}} 
\newcommand{\N}{\mathbb{N}} \newcommand{\M}{\mathbb{M}}

 \newcommand{\CI}{\mathcal{I}}

 \newcommand{\CU}{\mathcal{U}}

\newcommand{\Rand}[1]{\marginpar{#1}} %\renewcommand{\Rand}[1]{}
\newcommand{\be}[1]{\begin{equation}\label{#1}}
\newcommand{\ee}{\end{equation}}
\newcommand{\bew}[1]{\Rand{\vspace{0,6cm}\tt #1}\begin{equation*}\label{#1}}
\newcommand{\eew}{\end{equation*}}
\newcommand{\bea}[1]{\Rand{\vspace{0,6cm}\tt #1}\begin{eqnarray*}\label{#1}}
\newcommand{\eea}[1]{\end{eqnarray*}}

\newcommand{\beL}[2]{\Rand{\vspace{0,6cm}\tt #1}\begin{lemma}[#2]\label{#1}}
\newcommand{\beD}[2]{\Rand{\vspace{0,6cm}\tt #1}\begin{definition}[#2]\label{#1}}
\newcommand{\beT}[2]{\Rand{\vspace{0,6cm}\tt #1}\begin{theorem}[#2]\label{#1}}
\newcommand{\beP}[2]{\Rand{\vspace{0,6cm}\tt #1}\begin{proposition}[#2]\label{#1}}
\newcommand{\beC}[1]{\Rand{\vspace{0,6cm}\tt #1}\begin{corollary}\label{#1}}
\newcommand{\beR}[1]{\Rand{\vspace{0,6cm}\tt #1}\begin{remark}[#1]\label{#1}}
\newcommand{\Tto}{{_{\displaystyle\Longrightarrow\atop t\to\infty}}}
\newcommand{\tto}{{_{\displaystyle\longrightarrow\atop t\to\infty}}}
\newcommand{\Tno}{{_{\displaystyle\Longrightarrow\atop n\to\infty}}}
\newcommand{\tno}{{_{\displaystyle\longrightarrow\atop n\to\infty}}}
\newcommand{\TNo}{{_{\displaystyle\Longrightarrow\atop N\to\infty}}}

\newcommand{\smallk}{\mathpzc{k}}
\newcommand{\smallp}{\mathpzc{p}}
\newcommand{\smallu}{\mathpzc{u}}
\newcommand{\smallx}{\mathpzc{x}}

\DeclareMathAlphabet{\mathpzc}{OT1}{pzc}{m}{it}

\begin{document}

\title[Tree-valued resampling dynamics]
{{\Large Tree-valued resampling dynamics}\\[3mm]
  {\large Martingale Problems and applications}}

\author{Andreas Greven}
\address{Andreas Greven\\ Mathematisches Institut\\ University of Erlangen\\ Bismarckstr.\ 1$\tfrac 12$\\
  D-91054 Erlangen \\ Germany} \email{greven@mi.uni-erlangen.de}

\author{Peter Pfaffelhuber}
%\email{p.p@lmu.de}
\address{Peter Pfaffelhuber\\ Abteilung f{\"u}r mathematische
  Stochastik\\ Albert-Ludwigs University of
  Freiburg\\ Eckerstra\ss e 1\\ D-79104 Freiburg\\
  Germany} \thanks{All authors were supported in part by the
  DFG-Forschergruppe 498 through grant GR 876/13-1,2,3.}
\thanks{Peter Pfaffelhuber was supported in part by the BMBF, Germany,
  through FRISYS (Freiburg Initiative for Systems biology),
  Kennzeichen 0313921.}  \email{p.p@stochastik.uni-freiburg.de}

\author{Anita Winter}
\address{Anita Winter\\
  Universit\"at Duisburg-Essen\\
  Fakult\"at f\"ur Mathematik\\Universit\"atsstr. 2\\ D-45141 Essen\\
  Germany} \email{anita.winter@uni-due.de} \thanks{Anita Winter was
  supported in part at the Technion by a fellowship from the Aly
  Kaufman Foundation}

\thispagestyle{empty} \date{\today, 110615{\_}mp{\_}ptrf.tex}



  \begin{abstract}\sloppy
    The measure-valued Fleming-Viot process is a diffusion which
    models the evolution of allele frequencies in a multi-type
    population. In the neutral setting the Kingman coalescent is known
    to generate the genealogies of the ``individuals'' in the
    population at a fixed time.  The goal of the present paper is to
    replace this static point of view on the genealogies by an
    analysis of the evolution of genealogies.

    We encode the genealogy of
      the population as an (isometry class of an) ultra-metric
      space which is equipped with a probability measure. The space
      of ultra-metric measure spaces together with the Gromov-weak
      topology serves as state space for tree-valued processes.
       We
    use well-posed martingale problems to construct the tree-valued
    resampling dynamics of the evolving genealogies for both the
    finite population Moran model and the infinite population
    Fleming-Viot diffusion.

    We show that sufficient information about any ultra-metric measure
    space is contained in the distribution of the vector of subtree
    lengths obtained by sequentially sampled ``individuals''.  We give
    explicit formulas for the evolution of the Laplace transform of
    the distribution of finite subtrees under the tree-valued
    Fleming-Viot dynamics.
\end{abstract}

\maketitle

\section{Introduction}
The evolution of a population is commonly modeled using branching or
resampling dynamics. In both cases the analysis of the genealogical
relationships of individuals leads to a deeper understanding of the
underlying dynamics and is crucial in applications in population
genetics.  An observation which is fundamental for the present paper
is that genealogical relationships between individuals change as the
population evolves. We here want to construct and study the evolution
of the {\em genealogical structure} for the neutral Fleming-Viot
process which arises as a large population limit of various finite
resampling models
(\cite{FlemingViot1978,FlemingViot1979,Dawson1993,EthierKurtz1993,DGV95,MR1779100}).

  A basic finite resampling model is the Moran model, which can be
  described as follows: Each pair of individuals, taken from a finite
  population of fixed size, \emph{resamples} at constant
  rate. Resampling means that one individual is chosen at random from
  the pair, the pair dies and is replaced by two new individuals which
  are both offspring of the chosen individual.

  In resampling models
  genealogical trees can be generated by coalescent processes. The
  equilibrium genealogy of the Fleming-Viot diffusion, for example, is
  generated by the Kingman coalescent
  (\cite{Kingman1982a,Ald1993,Evans2000,GLW05}).  More general
  resampling dynamics which allow for an infinite offspring variance
  are studied in \cite{BertoinLeGall2003}. Their genealogical trees
  are described by $\Lambda$-coalescents
  (\cite{Pit1999,GrePfaWin2006a}).  Genealogical trees are also
  considered for branching models which allow for a varying population
  size. Prominent examples are the Kallenberg tree (\cite{Kal1977}),
  the Yule tree (\cite{EvansOConnell1994}), the Brownian continuum
  random tree (\cite{MR93f:60010}) and the Brownian snake
  (\cite{MR1714707}). More general branching mechanisms lead e.g.\ to
  L\'evy trees (\cite{MR1954248}), which are the infinite variance
  offspring distribution counterpart of the Brownian continuum random
  trees, and trees arising in catalytic branching systems
  (\cite{GrePopWin2007}).

   Coalescent trees describe the genealogy
  of a population at a fixed time and give therefore a static picture
  only. The main goal of the present paper is to give with the {\em
    tree-valued Fleming-Viot dynamics} a dynamic picture which
  describes the evolution of genealogies.  Evolving genealogies in
  exchangeable population models have already been described by {\em
    look-down processes}
  (\cite{DonKur1996,DonKur1998,DonKur1999,BirknerBlathMoehle2009});
  see also Remark~\ref{Rem:14}. For neutral evolution, look-down
  processes contain -- though in an implicit way -- all information
  about the genealogies. The depth of the tree
  (\cite{PfaffelhuberWakolbinger2006,EvansRalph2010,DelmasEtAl2010})
  and the total tree length
  (\cite{PfaffelhuberWakolbingerWeisshaupt2010}) are examples of
  functionals of a genealogy which are constructed and studied via the
  look-down construction. The crucial point in the construction of
  look-down processes is the use of labels as coordinates. This
  information is often not needed and constraints the construction of
  tree-valued processes in selective (unequal chances of producing
  offspring) and spatial (only pairs in the same location may
  resample) settings.

  A first approach in the direction of a coordinate-free description
  has already been established for spatially structured populations
  via {\em historical processes} (\cite{DawPer91,GLW05}). Here,
  however, the coding of the genealogical relationships requires that
  different ``offspring'' immediately follow different spatial paths,
  almost surely. Only then the genealogy can be read off from the
  spatial paths of the ``individuals''. Historical processes can
  therefore, in particular, not be used to study genealogies in
  non-spatial situations.

  A different and more canonical coding of trees is therefore needed.
  In this paper we rely on the fact that genealogical distances
  between individuals define a metric. To take the individuals'
  contribution to the population into account we equip the resulting
  metric space with the (weak limit of the) empirical distribution of
  the individuals. We then follow the theory of {\em metric measure
    spaces} equipped with the {\em Gromov-weak topology} as developed
  in \cite{GrePfaWin2006a}. We show weak convergence of tree-valued
  Moran models and construct the limiting tree-valued Fleming-Viot
  dynamics. Such weak convergence results are best treated by using
  well-posed martingale problems, which allow -- in contrast to other
  techniques such as Dirichlet forms -- for statements concerning
  convergence of infinitesimal characteristics. In order to define
  these characteristics, we require a suitably large class of
  continuous functions which are easy to manipulate.  For tree-valued
  processes such an approach is novel. We make use of general theory
  in order to establish well-posedness of the limiting martingale
  problem (Theorem~\ref{T:01}), weak convergence of tree-valued Moran
  models (Theorem~\ref{T:02}) and the long-time behavior
  (Theorem~\ref{T:03}).

  Another useful consequence of a well-posed martingale problem is
  that it allows to study the evolution of continuous functionals of
  these processes and to characterize the functionals which are strong
  Markov processes. Of particular importance is the vector of tree
  lengths for subsequently sampled ``individuals''. An important
  result (Theorem~\ref{T:04}) is that the resulting {\em subtree
    length distribution} characterizes the ultra-metric measure tree
  uniquely. From a theoretical point of view this can be considered as
  a generalization of the moment problem for bounded real-valued
  random variables to metric measure spaces. It is also of interest in
  statistical applications since it states that all sufficient
  information about genealogies is contained in the lengths of
  subtrees spanned by a finite sample.  Under the Fleming-Viot
  dynamics we construct the evolution of the tree length distribution
  via a well-posed martingale problem (Theorem~\ref{T:05}). Moreover,
  we derive explicit formulas for the evolution of the Laplace
  transform of finite subtrees.

  Markov dynamics with values in the space of continuum trees have
  been constructed only recently. Examples include excursion
  path-valued Markov processes with continuous sample paths - which
  can therefore be thought of as tree-valued diffusions - as
  investigated in \cite{MR1814427, MR1891060, MR1959795}, and dynamics
  working with real-trees, for example, the so-called {\em root growth
    with re-grafting} (\cite{EvaPitWin2006}), the so-called {\em
    subtree prune and re-graft move} (\cite{EvaWin2006}) and the
  limiting {\em random mapping} (\cite{math.PR/0701657}).  The present
  construction is extended to Fleming-Viot processes with selection in
  \cite{DepperschmidtGrevenPfaffelhuber2011}.

\section{Main results (Theorems~\ref{T:01},\ref{T:02} and~\ref{T:03})}
\label{S:results}
In this section we state our main results. In
Subsection~\ref{Sub:state} we recall concepts and terminology used to
define the state space
which consists of (ultra-)metric measure spaces
equipped with the Gromov-weak topology.  In Subsection~\ref{Sub:mp} we
state the tree-valued Fleming-Viot martingale problem and its
well-posedness (Theorem~\ref{T:01}), and present the approximation by
tree-valued Moran dynamics in Subsection~\ref{Sub:conv}
(Theorem~\ref{T:02}). In Subsection~\ref{Sub:long} we identify a
unique equi\-librium and state that it will be approached as time
tends to infinity (Theorem~\ref{T:03}).

\subsection{State space: metric measure spaces}
\label{Sub:state}
To define the state space we consider trees as metric
spaces. Moreover, to allow for a topology which discards atypical
points in the tree, we will equip these metric spaces with a
probability measure on the leaves. (Compare also with
Remark~\ref{rem:roleII}).  We then equip the space of metric measure
spaces with the Gromov-weak topology which combines the concept of
weak convergence of probability measures in a fixed metric space with
Gromov's idea of comparing different metric spaces.  In
\cite{GrePfaWin2006a} topological aspects of the space of metric
measure spaces equipped with the Gromov-weak topology are
investigated.  In this subsection we recall basic facts and notation.

As usual, given a topological space $(X,{\mathcal O})$ we denote by
${\mathcal M}_1(X)$ the space of all probability measures defined on
the Borel-$\sigma$-algebra of $X$, and by $\Rightarrow$ weak
convergence in $\mathcal M_1(X)$. Recall that the support
$\mathrm{supp}(\mu)$ of $\mu\in\mathcal M_1(X)$ is the smallest closed
set $X_0\subseteq X$ such that $\mu(X_0)=1$.  The push forward of
$\mu$ under a measurable map $\varphi$ from $X$ into another
topological space $Z$ is the probability measure
$\varphi_\ast\mu\in{\mathcal M}_1(Z)$ defined by
\begin{equation}\label{push}
  \varphi_\ast\mu(A) := \mu\big(\varphi^{-1}(A)\big),
\end{equation} for
all Borel subsets $A\subseteq Z$. We denote by $\mathcal B(X)$ and
${\mathcal C}_b(X)$ the bounded real-valued functions on $X$ which are
measurable and continuous, respectively.

A {\em metric measure space} is a triple $(X,r,\mu)$ where $(X,r)$ is
a metric space equipped with $\mu\in\mathcal M_1(X)$ such that
$(\mathrm{supp}(\mu),r)$ is complete and separable. Two metric measure
spaces $(X,r,\mu)$ and $(X^\prime,r^\prime,\mu^\prime)$ are
\emph{measure-preserving isometric} or \emph{equivalent} if there
exists an isometry $\varphi$ between $\mathrm{supp}(\mu)$ and
$\mathrm{supp}(\mu')$ such that $\mu'=\varphi_\ast\mu$. It is clear
that the property of being measure-preserving isometric is an
equivalence relation. We write $\overline{(X,r,\mu)}$ for the
equivalence class of a metric measure space $(X,r,\mu)$. Define the
set of (equivalence classes of) metric measure spaces
\begin{equation}\label{eq:defM}
  \mathbb M
  :=
  \big\{\smallx=\overline{(X,r,\mu)}:\,(X,r,\mu)\text{ metric
    measure space}\big\}.
\end{equation}
If $(X,r,\mu)$ is such that $r$ is only a pseudo-metric on $X$, (i.e.\
$r(x,y)=0$ is possible for $x\neq y$) we can still define its
measure-preserving isometry class.  Since this class contains also
metric measure spaces, there is a bijection between the set of
pseudo-metric measure spaces and the set of metric measure spaces and
we use both notions interchangeably.

For a metric space $(X,r)$ we define by
\begin{equation}\label{mono1a}
  R^{(X,r)}:\,
  \left\{\begin{array}{ll}X^{\mathbb N} &\to \mathbb R_+^{\binom{\mathbb N}{2}} \\
      \big((x_i)_{i\geq 1}\big)
      &\mapsto
      \big(r(x_i, x_j)\big)_{1\leq i<j}\end{array}\right.
\end{equation}
the map which sends a sequence of points in $X$ to its (infinite)
distance matrix, and denote, for a metric measure space $(X,r,\mu)$,
the {\em distance matrix distribution} of $(X,r,\mu)$ by
\label{Rem:02}
\begin{equation}\label{mono1b}
   {\nu}^{(X,r,\mu)} :=
   \big(R^{(X,r)}\big)_\ast\mu^{\otimes\mathbb N}\in{\mathcal M}_1(\mathbb R_+^{\N\choose 2}).
\end{equation}
Obviously, ${\nu}^{(X,r,\mu)}$ depends on $(X,r,\mu)$ only through its
measure-preserving isometry class $\smallx=\overline{(X,r,\mu)}$. We
can therefore define:

\begin{definition}[Distance matrix distribution] The distance matrix
  distribution $\nu^{\smallx}$ of $\smallx\in\mathbb{M}$ is the
  distance matrix distribution $\nu^{(X,r,\mu)}$ of an arbitrary
  representative $(X,r,\mu)\in\smallx$.
\label{Def:17}
\end{definition}\sm

By Gromov's reconstruction theorem metric measure spaces are uniquely
determined by their distance matrix distribution (see
Section~$3\tfrac12.5$ in \cite{Gromov2000} and Proposition 2.6 in
\cite{GrePfaWin2006a}). We therefore base our notion of convergence
in $\mathbb M$ on the convergence of distance matrix distributions.
In \cite{GrePfaWin2006a} we introduced the {\em Gromov-weak topology} in which a sequence $(\smallx_n)_{n\in\N}$
converges to $\smallx$ if and only if
\begin{equation}\label{e:convv2}
   \nu^{\smallx_n} \Tno\nu^\smallx
\end{equation}
in the weak topology on $\mathcal M_1(\mathbb R_+^{\binom{\mathbb
    N}{2}})$ (and, as usual, $\mathbb R_+^{\binom{\mathbb N}{2}}$
equipped with the product topology); compare Theorem 5 of
\cite{GrePfaWin2006a}.  Although $\{\nu^\smallx: \smallx\in\mathbb
M\}\subseteq \mathcal M_1(\mathbb R^{\binom{\N}{2}})$ is not closed,
we could show in Theorem~1 of \cite{GrePfaWin2006a} that $\mathbb M$,
equipped with the Gromov-weak topology, is Polish. \sm

Several sub-spaces of $\mathbb M$ are of special interest throughout
the paper. Above all, these are the \emph{ultra-metric} and
\emph{compact} metric measure spaces.

(The equivalence class of) a metric measure space $(X,r,\mu)$ is
called \emph{ultra-metric} iff
\begin{equation}\label{e:ultra}
   r(u,w)
 \leq
  r(u,v)\vee r(v,w),
\end{equation}
for $\mu$-almost all $u,v,w\in X$. Define
\begin{equation}\label{w:002}
\begin{aligned}
  \mathbb U
 :=
  \big\{\smallu\in\mathbb{M}:\,\smallu\text{ is ultra-metric}\big\}.
\end{aligned}
\end{equation}

\begin{remark}[Ultra-metric spaces are trees]
  Notice that there is a close connection between ultra-metric spaces
  \label{rem:ultraTrees}
  and $\R$-trees, i.e., complete path-connected metric spaces
  $(X,r_X)$ which satisfy the four-point condition
\begin{equation}\label{grev30}
  \begin{aligned}
    &r_X(x_1,x_2)+r_X(x_3,x_4)
    \\
    &\leq
    \max\big\{r_X(x_1,x_3)+r_X(x_2,x_4),r_X(x_1,x_4)+ r_X(x_2,x_3)\big\},
  \end{aligned}
\end{equation}
for all $x_1, x_2, x_3, x_4 \in X$ (see, for example, \cite{Dre84,
  DreMouTer96, Ter97}). \label{Rem:01} On the one hand, every complete
ultra-metric space $(U,r_U)$ {\em spans} a path-connected complete
metric space $(X,r_X)$ which satisfies the {\em four point condition},
such that $(U,r_U)$ is isometric to the set of leaves $X\setminus
X^o$.  On the other hand, given an $\R$-tree $(X,r_X)$ and a
distinguished point $\rho_X\in X$ which is often referred to as the
{\em root} of $(X,r_X)$, the level sets $X^t:=\{x\in
X:\,r(\rho_X,x)=t\}$, for $t\ge 0$, form ultra-metric sub-spaces
of $(X,r_X)$. For more details, see \cite[Theorem~3.38]{Evans2005}.

Because of this connection between ultra-metric spaces and real trees,
ultra-metric spaces are often (especially in phylogenetic analysis)
referred to as {\em ultra-metric trees}.
\hfill$\qed$
\end{remark}\sm

The next lemma implies that $\mathbb{U}$ equipped with the Gromov-weak
topology is again Polish.
\begin{lemma}\label{l:21}
  The sub-space $\mathbb{U}\subset\M$ is closed.
\end{lemma}\sm

\begin{proof} Let $(\smallu_n)_{n\in\N}$ be a sequence in $\mathbb U$
  and $\smallx\in\mathbb M$ such that $\smallu_n\rightarrow\smallx$ in
  the Gromov-weak topology, as $n\to\infty$. Equivalently, by
  \eqref{e:convv2}, $\nu^{\smallu_n}\Rightarrow\nu^\smallx$ in the
  weak topology on $\mathcal M_1(\mathbb R_+^{\binom{\mathbb N}{2}})$,
  as $n\to\infty$.  Consider the open set
\begin{equation}\label{eq:A}
\begin{aligned}
   &A
  \\
 &:=
   \big\{(r_{i,j})_{1\leq i<j}:\,r_{1,2}> r_{23}\vee
   r_{1,3}\mbox{ or }r_{2,3}> r_{1,2}\vee r_{1,3}\mbox{ or }r_{1,3}> r_{1,2}\vee
   r_{2,3}\big\}.
\end{aligned}
\end{equation}
By the Portmanteau Theorem, $\nu^\smallx(A) \leq \liminf_{n\to\infty}
\nu^{\smallu_n}(A) = 0$.  Thus, (\ref{e:ultra}) holds for
$\mu^{\otimes 3}$-all triples $(u,v,w)\in X^3$.  In other words,
$\smallx$ is ultra-metric.
\end{proof}\sm

(The equivalence class of) a metric measure space $(X,r,\mu)$ is
called {\em compact} if and only if the metric space
$(\mathrm{supp}(\mu),r)$ is compact. Define
\begin{equation}\label{grev31}
   \mathbb M_c
 :=
   \big\{\smallx\in\mathbb{M}:\,\smallx\text{ is compact}\big\}.
\end{equation}
Moreover, we set
\begin{equation}\label{e:grev31a}
   \mathbb{U}_c
 :=
   \mathbb{U}\cap\mathbb{M}_c.
\end{equation}

\begin{remark}[$\mathbb{M}_c$ is not a closed subset of $\mathbb M$]
\begin{itemize}
\item[{}]
\item[(i)] If $\smallx=\overline{(X,r,\mu)}$ is a {\em finite} metric
  measure space, i.e, $\#\mathrm{supp}(\mu)<\infty$, then
  $\smallx\in\mathbb{M}_c$.
\item[(ii)] Since elements of $\mathbb M$ can be approximated by a
  sequence of finite metric measure spaces (see the proof of
  Proposition 5.3 in \cite{GrePfaWin2006a}), the sub-space
  $\mathbb{M}_c$ is not closed. \label{Rem:03} A similar argument
  shows that $\mathbb U_c$ is not closed.
\item[(iii)] In order to establish convergence within the space of
  compact metric measure spaces, we provide a {\em relative
    compactness criterion} in $\mathbb M_c$ in Proposition~\ref{noteP:04}.
  \hfill$\qed$
\end{itemize}
\end{remark}\sm

\subsection{The martingale problem (Theorem~\ref{T:01})}
\label{Sub:mp}

In this subsection we define the tree-valued Fleming-Viot dynamics as
the solution of a well-posed martingale problem.  We start by
recalling the terminology. All proofs are given in
Section~\ref{S:proofs}.
\begin{definition}[Martingale problem] \label{D:01} Let $(E,{\mathcal
    O})$ be a Polish space, $\mathbf
  P_0 \in\mathcal M_1(E)$, $\mathcal F$ a subspace of the space
  ${\mathcal B}(E)$ of bounded measurable functions on $E$ and
  $\Omega$ a linear operator on ${\mathcal B}(E)$ with domain
  $\mathcal F$.

  The law $\mathbf{P}$ of an $E$-valued stochastic process
  $X=(X_t)_{t\geq 0}$ is called a solution of the $(\mathbf
  P_0,\Omega,\mathcal F)$-martingale problem if $X_0$ has distribution
  $\mathbf P_0$, $X$ has paths in the space ${\mathcal
    D}_E([0,\infty))$ of $E$-valued c\`adl\`ag functions, almost
  surely (where $\mathcal D_E([0,\infty)$ is equipped with the
  Skorohod topology) and for all $F\in\mathcal F$,
  \begin{equation}\label{13def}
    \Big(F(X_t)-\int_0^t\mathrm{d}s\,
    \Omega F(X_s)\Big)_{t\geq 0}
  \end{equation}
  is a $\mathbf{P}$-martingale with respect to the canonical
  filtration.

  Moreover, the $(\mathbf P_0,\Omega,\mathcal F)$-martingale problem
  is said to be well-posed if there is a unique solution $\mathbf{P}$.
\end{definition}\sm

Recall that the classical measure-valued Fleming-Viot process
$\zeta=(\zeta_t)_{t\geq 0}$ is a probability measure-valued diffusion
process, which describes the evolution of allelic frequencies; see
e.g.\ \cite{Dawson1993,MR1779100}. In particular, for a fixed time
$t$, the state $\zeta_t\in\mathcal M_1(K)$ records the current
distributions of allelic types on some (Polish) type space $K$.  This
process is defined as the unique solution of the martingale problem
corresponding to the following operator $\widehat \Omega^\uparrow$
(see \cite{EthierKurtz1993}): for functions $\widehat \Phi: \mathcal
M_1(K)\to\mathbb R$ of the form
\begin{equation}\label{pp13a}
  \widehat\Phi(\zeta )
  =
  \big\langle \zeta^{\otimes \mathbb N}, \widehat \phi\big\rangle
  :=
  \int_{K^\N} \zeta^{\otimes \mathbb N}({\rm d} \underline u)\, \widehat\phi\big(
  \underline u\big)
\end{equation}
with $\underline u=(u_1,u_2,...)\in K^\N$ and $\widehat\phi
\in{\mathcal C}_b(K^{\mathbb N})$ depending only on finitely many
coordinates, set
\begin{equation}\label{gen1}
  \widehat\Omega^\uparrow\widehat\Phi(\zeta)
  =
  \frac\gamma 2 \sum_{k,l\geq 1}\Big(\langle  \zeta^{\otimes\mathbb N},
  \widehat\phi\circ \widehat\theta_{k,l} \rangle -
  \langle \zeta^{\otimes\mathbb N}, \widehat\phi \rangle \Big)
\end{equation}
where the \emph{replacement operator}
$\widehat\theta_{k,l}:K^\mathbb{N}\to K^{\mathbb{N}}$ is the map which
replaces the $l^{\mathrm{th}}$ component of an infinite sequence of
types by the $k^{\mathrm{th}}$:
\begin{equation}\label{gen2}
  \widehat\theta_{k,l}
  (u_1,u_2,\ldots, u_{l-1}, u_l, u_{l+1}, \ldots)
  := (u_1,u_2,\ldots, u_{l-1}, u_k, u_{l+1}, \ldots).
\end{equation}
Here and in the following $\gamma\in(0,\infty)$ is referred to as the
{\em resampling rate}.\sm

In order to state the martingale problem for the tree-valued
Fleming-Viot dynamics we need the notion of \emph{polynomials} on
$\mathbb M$.

\begin{definition}[Polynomials]
  A function $\Phi:\mathbb{M}\to\R$
  \label{metricD:01}
  is called a {\em polynomial} if there exists a bounded, measurable
  {\em test function} $\phi: \mathbb R_+^{\binom{\mathbb
      N}{2}}\to\mathbb R$, depending only on finitely many variables
  such that
  \begin{equation}\label{metricpp3}
    \Phi\big(\smallx\big)
    =
    \big\langle\nu^\smallx,\phi\big\rangle
    :=
    \int_{\R_+^{\N\choose 2}}\nu^\smallx(\mathrm d\underline{\underline r})\,\phi\big(\underline{\underline r}\big),
  \end{equation}
  where $\underline{\underline r}:=(r_{i,j})_{1\leq i<j}$.
\end{definition}\sm

\medskip

\begin{remark}[Properties of polynomials]\rm
\begin{itemize}
\item[{}]
  \item[(i)]
      Let $\Phi$ and $\phi$ be as in
    Definition~\ref{metricD:01}. If $\smallx = \overline{(X,r,\mu)}$,
    then
    \begin{equation}\label{metricpp3b}
      \begin{aligned}
        \Phi\big(\smallx\big)
        =
        \int_{X^\N}\mu^{\otimes \mathbb N}(\mathrm{d}(x_1, x_2,\ldots))\,
        \phi\big((r(x_{i},x_{j}))_{1\le i<j}\big),
      \end{aligned}
    \end{equation}
    where $\mu^{\otimes \mathbb N}$ is the $\mathbb N$-fold product measure of $\mu$.\label{Rem:13}
  \item[(ii)]
      If $n\in\mathbb N$ is the minimal number such that there
    exists $\phi\in\mathcal B(\mathbb R_+^{\binom{\mathbb N}2})$,
    depending only on $(r_{i,j})_{1\leq i<j\leq n}$ such that
    \eqref{metricpp3} holds, $n$ is referred to as \emph{degree} and
    $\phi$ as a \emph{minimal test function} of $\Phi$. We write
    $\Phi=\Phi^{n,\phi}$.
  \item[(iii)] For $m\in\mathbb{N}$, let $\Sigma_m$ be the set of
    permutations of $\mathbb N$ which leave $m+1$, $m+2$, ... fixed.
    For $\sigma\in\Sigma_\infty := \bigcup_{m\in\mathbb{N}}\Sigma_m$,
    define
    \begin{equation}\label{permut}
      \widetilde\sigma\big( (r_{i,j})_{1\leq i<j}\big)
      :=
      \big(r_{\sigma(i)\wedge\sigma(j),\sigma(i)\vee\sigma(j)}\big)_{1\leq i<j}.
    \end{equation}
    The \emph{symmetrization} of $\phi$ is given by
    \begin{align}\label{eq:symm}
      \bar\phi = \frac{1}{n!} \sum_{\sigma\in\Sigma_n} \phi\circ
      \widetilde\sigma.
    \end{align}
    By symmetry of $\nu^\smallx$, $ \langle \nu^\smallx, \phi\rangle =
    \langle \nu^\smallx, \bar\phi\rangle$, or equivalently,
    $\Phi^{n,\phi} = \Phi^{n,\bar\phi}$.
  \end{itemize}
  \hfill$\qed$
\end{remark}

Recall from Subsection~\ref{Sub:state} the space ${\mathcal
  B}(\R_+^{\N\choose 2})$ of bounded measurable real-valued functions
on $\R_+^{\N\choose 2}$.  An element $\phi\in {\mathcal
  B}(\R_+^{\N\choose 2})$ is said to be differentiable if for all
$1\le i<j$ the partial derivatives $\frac{\partial \phi}{\partial
  r_{i,j}}$ exist and if $\sum_{1\le i<j}{|\frac{\partial
    \phi}{\partial r_{i,j}}}|<\infty$. In this case we put

\begin{equation}
\label{e:divergenz}
   \langle\nabla \phi, \underline{\underline {2}}\rangle
 :=
   2\sum_{1\le i<j}\frac{\partial \phi}{\partial r_{i,j}}
 =
   \sum_{{1\le i,j}\atop{i\neq j}}\frac{\partial \phi}{\partial r_{i\wedge j, i\vee j}}.
\end{equation}

Denote by ${\mathcal C}_b^1(\R_+^{\N\choose 2})$ the space of all
bounded and continuously differentiable real-valued functions $\phi$
on $\R_+^{\N\choose 2}$ with bounded derivatives.  The function spaces
we use in the sequel are the space of polynomials
\begin{equation}\label{eq:po1}
  \Pi :=
  \big\{\Phi^{n,\phi}\mbox{ as in Remark~\ref{Rem:13}(ii)}:\,n\in\N,\phi\in\mathcal B(\mathbb
  R_+^{\binom{\mathbb N}{2}})\big\}, \end{equation}
and its sub-spaces
\begin{equation}
\begin{aligned}\label{eq:po1k}& \Pi^{0} :=
  \big\{\Phi^{n,\phi}\in\Pi:\,n\in\N,\phi\in{\mathcal C}_b(\mathbb
  R_+^{\binom{\mathbb N}{2}})\big\},
\end{aligned}
\end{equation}
and
\begin{equation}
  \begin{aligned}\label{eq:po1l}
    &\Pi^{1} :=
    \big\{\Phi^{n,\phi}\in\Pi:\,n\in\N,\phi\in{\mathcal C}_b^{1}(\mathbb
    R_+^{\binom{\mathbb N}{2}})\big\}.
\end{aligned}
\end{equation}

\medskip
\medskip

\begin{remark}[Polynomials form an algebra that separates
    points]
  \begin{enumerate}
  \item[{}]
  \item[(i)]   Observe that $\Pi, \Pi^0$ and $\Pi^1$ are algebras of
    \label{rem:05}
    functions. Specifically, given $\Phi^{n,\phi}$, and
    $\Phi^{m,\psi}\in \Pi$,
    \begin{align}\label{eq:qvprep}
      \Phi^{n,\phi}\cdot \Psi^{m,\psi}=\Phi^{n+m,(\phi,\psi)_n}=\Phi^{n+m,(\psi,\phi)_m}
    \end{align}
    where for $\phi,\psi\in{\mathcal B}(\R_+^{\binom{\N}{2}})$ and
      $\ell\in\N$,
    \begin{align}\label{eq:qvprep1}
      (\phi, \psi)_\ell\big(\underline{\underline r}\big) :=
      \phi(\underline{\underline{r}}) \cdot \psi(\tau_\ell
      \underline{\underline{r}}),
    \end{align}
    with $\tau_\ell\big( (r_{i,j})_{1\leq i<j}\big) = (r_{\ell+i,
      \ell+j})_{1\leq i<j}$.
  \item[(ii)] By Proposition~2.6 in~\cite{GrePfaWin2006a}, $\Pi$ and
    $\Pi^0$ separate points in $\mathbb M$. Since ${\mathcal
      C}_b^1(\R_+^{\N\choose 2})$ is dense in the topology of
    point-wise convergence in ${\mathcal C}_b(\R_+^{\N\choose 2})$,
    $\Pi^1$ separates points as well.  \hfill $\qed$
  \end{enumerate}
\end{remark}\sm

\begin{remark}[The Gromov-weak topology]
  Let $\smallx$, \label{rem:exch} $\smallx_1$,
  $\smallx_2,...\in\mathbb M$. Recall from \eqref{e:convv2} that
  $\smallx_n \to\smallx$, as $n\to\infty$, in the Gromov-weak topology
  iff $\nu^{\smallx_n}\Longrightarrow\nu^\smallx$, as
  $n\to\infty$. Equivalently, $\Phi(\smallx_n)\to\Phi(\smallx)$, as
  $n\to\infty$, for all $\Phi\in\Pi^0$ (see Theorem~5 in
  \cite{GrePfaWin2006a}). Notice that $\smallx_n \to\smallx$, as
  $n\to\infty$, if we restrict to $\Pi^1$ or to the set
  $\{\Phi^{n,\bar\phi}: \Phi^{n,\phi}\in\Pi\}$ of symmetric test
  functions. (Compare with Remark~\ref{Rem:13}(iii)). \hfill $\qed$
\end{remark}\sm

To lift the measure-valued Fleming-Viot process to the level of trees
and thereby construct the tree-valued Fleming-Viot dynamics, we
consider the martingale problem associated with the operator
$\Omega^{\uparrow}$ on $\Pi$ with domain $\mathcal D(\Omega^\uparrow) = \Pi^1$.
To define $\Omega^{\uparrow}$ we let for $\Phi=\Phi^{n,\phi}\in\Pi^1$,
\begin{equation}\label{Pi1}
   \Omega^{\uparrow}\Phi
 :=
   \Omega^{\uparrow,\mathrm{grow}}\Phi+\Omega^{\uparrow,\mathrm{res}}\Phi.
\end{equation}

The \emph{growth operator} $\Omega^{\uparrow,\mathrm{grow}}$ reflects
the fact that the population gets older and therefore the genealogical
distances grow at speed $2$ as time goes on. We therefore put
\begin{equation}\label{eq:omega1}
  \Omega^{\uparrow,\mathrm{grow}}\Phi\big(\smallu\big)
  :=
  \big\langle\nu^{\smallu},\langle \nabla \phi, \underline{\underline 2}\rangle\big\rangle.
\end{equation}

For the \emph{resampling operator} let
\begin{align}
\label{eq:omega2}
   \Omega^{\uparrow,\mathrm{res}}\Phi\big(\smallu)
 &:=
    \frac{\gamma}{2}\sum_{1\le k,l\leq n}\Big(\big\langle \nu^\smallu,\phi\circ\theta_{k,l}\big\rangle -\big\langle\nu^\smallu,\phi\big\rangle\Big),
\end{align}
where we put $r_{k,k}=0$ for all $k\ge 1$, and
\begin{equation}\label{pp11b}
  \begin{aligned}
    \big(\theta_{k,l}\big( (r_{i',j'})_{1\leq i'<j'}\big)\big)_{i,j}
    & :=
    \begin{cases}r_{i,j}, & \mbox{ if }i,j\neq l
      \\
      r_{i\wedge k, i\vee k} , & \mbox{ if }j=l,
      \\
      r_{j\wedge k, j\vee k} , & \mbox{ if }i=l.\end{cases}
  \end{aligned}
\end{equation}
Note that $\Omega^\uparrow \Phi \in \Pi$ for all $\Phi\in\Pi^1$.

Our first main result states that the martingale problem associated
with $(\Omega^\uparrow,\Pi^1)$ is well-posed.

\begin{theorem}[Well-posed martingale problem]
  For all $\mathbf P_0 \in \mathcal M_1(\mathbb{U})$, the $(\mathbf
  P_0,\Omega^{\uparrow}, {\Pi^1})$-martingale problem is
  well-posed. \label{T:01}
\end{theorem}\sm

This leads to the following definition.

\begin{definition}[The tree-valued Fleming-Viot dynamics]
  Fix $\mathbf P_0 \in \mathcal M_1(\mathbb{U})$. The {\rm tree-valued
    Fleming-Viot dynamics} with initial distribution $\mathbf P_0$ is
  a stochastic process with distribution $\mathbf P$, the unique
  solution of the $(\mathbf P_0,\Omega^{\uparrow}, \Pi^1)$-martingale
  problem. \label{Def:01}
\end{definition}\sm

\begin{proposition}[Sample path properties]
  \label{P:07}
  The tree-valued Fleming-Viot dynamics ${\mathcal U}$ has the
  following properties.
\begin{itemize}
\item[(i)] $\mathcal U$ has sample paths in ${\mathcal
    C}_{\mathbb{U}}([0,\infty))$, $\mathbf P$-almost surely.
\item[(ii)] $\mathcal U_t\in \mathbb U_c$, for all $t>0$, $\mathbf
  P$-almost surely.
\end{itemize}
\end{proposition}\sm

\begin{proposition}[Feller property]
  The \label{P:05} tree-valued Fleming-Viot dynamics $\mathcal U$ is a
  strong Markov process. Moreover, it has the Feller property, i.e.,
  $\smallu \mapsto \mathbf E[f(\mathcal U_t)|\mathcal U_0=\smallu]$ is
  continuous if $f\in\mathcal C_b(\mathbb{U})$.
\end{proposition}\sm

\begin{cor}[Quadratic variation]
  Let $\mathcal U = (\mathcal U_t)_{t\geq 0}$ be the tree-valued
  Fleming-Viot dynamics with initial distribution $\mathbf
  P_0\in\mathcal M_1(\mathbb U)$ and $\Phi = \Phi^{n,\phi} \in \Pi^1$,
  Then $\Phi(\mathcal U):=(\Phi(\mathcal U_t))_{t\geq 0}$ is a
  continuous $\mathbf P$-semi-martingale with quadratic
  variation\label{cor:qv2}
\begin{equation}
\label{eq:qv}
  \langle\Phi(\mathcal U)\rangle_t
 =
  \gamma n^2\int_0^t\mathrm{d}s\,
  \big\langle\nu^{\mathcal U_s},
  % \big(\bar{\phi}\circ\theta_{1,n+1}-\bar{\phi},\bar{\phi}\circ\theta_{1,n+1}-\bar{\phi}\big)_n\big\rangle^2
  (\bar\phi, \bar\phi)_n \circ
  \theta_{1,n+1}-(\bar\phi,
  \bar\phi)_n
  \big\rangle.
\end{equation}
\end{cor}\sm

\begin{remark}[Quadratic variation for a representative] Assume that
  for all $t>0$, $\mathcal U_t = \overline{(U_t,r_t,\mu_t)}$. Then the
  quadratic variation of $\Phi(\mathcal U)$ can be expressed
  as \label{rem:06}
\begin{equation}
  \label{eq:qv2}
  \begin{aligned}
    \langle \Phi(\mathcal U)\rangle_t = \gamma n^2 \int_0^t
    \mathrm{d}s\, \langle \mu_s, \big( \chi_s -
    \langle \mu_s, \chi_s\rangle\big)^2 \rangle,
  \end{aligned}
\end{equation}
where $\chi_s=\chi^\phi_s: U_s\to\R$ is defined as
  \begin{align}\label{name}
    \chi_s(u_1) := \int \mu_s^{\otimes \mathbb N}
    ({\rm d}(u_2,u_3,\ldots))\, \bar\phi ((r_s(u_i, u_j))_{1\leq
      i<j}).
  \end{align}
\hfill$\qed$
\end{remark}\sm

  \begin{remark}[The r\^{o}le of $\mu$]
    Throughout \label{rem:roleII} the paper we encode trees as metric
    measure spaces rather than just metric spaces. In the context of
    resampling, given $\smallu = \overline{(U,r,\mu)}$, the measure
    $\mu$ can be understood as the weak limit of empirical
    distribution of the individuals in the population (which are
    associated with points in $(U,r)$). This observation is in analogy
    to the measure-valued Fleming-Viot processes which arises as the
    large population limit of empirical distributions on type
    space. Moreover, the additional structure of a probability measure
    $\mu$ allows for defining polynomials and is therefore very
    helpful to come up with a suitably large class of generic
    functions on equivalence classes of measure metric spaces.
    \hfill$\qed$
  \end{remark}\sm

\begin{remark}[Extended martingale problem]
  \sloppy The \label{Rem:07} martingale approach characterizes a
  Markov process through a separating
      class
  of martingales.  Here, for example, the operator $(\Omega^\uparrow,
  \Pi^1)$ extends to an operator on the algebra
  \begin{align}
    \label{eq:fpoly0}
    \mathcal F = \{f\circ \Phi: f\in\mathcal B(\mathbb R), \Phi \in
    \Pi\}
  \end{align}
  with domain
  \begin{align}
    \label{eq:fpoly0s}
    {\mathcal F}^{2,1}:=\big\{f\circ \Phi: f\in\mathcal C_b^2(\mathbb R),
    \Phi\in\Pi^1\}
  \end{align}
  as follows (see e.g.\ \cite[Corollary~1.2]{FukushimaStroock1986}):
  \begin{equation}
    \label{e:extend}
    \begin{aligned}
      &\Omega^{\uparrow} (f\circ\Phi)(\smallu)
      \\
      &= f'\big(\Phi(\smallu)\big)\cdot \Omega^{\uparrow}\Phi(\smallu)
      +\tfrac{1}{2}f''\big(\Phi(\smallu)\big) \cdot \gamma n^2 \cdot
      \big\langle\nu^{\smallu}, (\bar\phi, \bar\phi)_n \circ
      \theta_{1,n+1}-(\bar\phi, \bar\phi)_n \big\rangle.
    \end{aligned}
  \end{equation}
  In particular, the tree-valued Fleming-Viot dynamics is the unique
  solution of the $(\Omega^\uparrow,{\mathcal F}^{2,1})$-martingale
  problem.  \hfill$\qed$
\end{remark}\sm

\begin{remark}[Reduced martingale problem]
  In view of Remark~\ref{Rem:07} one is interested in finding a
  preferably minimal class of functions such that the
  martingales~\eqref{13def} uniquely determine the process.  Here, for
  example, we can use the class of {\em prime} polynomials, were we
  want to refer to to $\Phi\in\Pi$ as {\em prime} if $\Phi$ is not of
  the form $\Phi \neq \widehat\Phi\cdot\widetilde\Phi$ for
  non-constant $\widehat\Phi, \widetilde\Phi\in\Pi$. Indeed by
  \eqref{e:extend} together with Corollary~\ref{cor:qv2} it is easy to
  see that an $\mathbb U$-valued process $\mathcal U =(\mathcal
  U_t)_{t\geq 0}$ is the unique solution of the $(\Omega^\uparrow,
  \Pi^1)$-martingale problem iff
  \begin{equation}
    \label{e:martingalevar}
    \begin{aligned}
      &\Big(\Phi({\mathcal U}_t)-\Phi({\mathcal
        U}_0)-\int^t_0\mathrm{d}s\,\Omega^{\uparrow}\Phi({\mathcal
        U}_s)\Big)_{t\ge 0}
    \end{aligned}
  \end{equation}
  is a martingale for all prime $\Phi\in\Pi^1$ with quadratic \label{Rem:09}
  variation given by \eqref{eq:qv}.
  \label{Rem:12}
  \hfill$\qed$
\end{remark}\sm

\subsection{Particle approximation (Theorem~\ref{T:02})}
\label{Sub:conv}
A classical result in population genetics gives the approximation of
the measure-valued Fleming-Viot process by a finite population model
-- the so called Moran model -- in the limit of large population size
(see e.g.\,\cite{Dawson1993, MR1779100}). In this model, ordered pairs
of individuals are replaced by new pairs in a way that the
``children'' choose a parent - which then becomes their common
ancestor - independently at random from the parent pair. In this
subsection we state that also the tree-valued Fleming-Viot dynamics
can be approximated by tree-valued resampling dynamics which
correspond to the Moran model.

We will proceed as follows. For further reference, we provide with
Proposition~\ref{PP:02} a condition for the compact containment
condition for finite population models in a general setting. For
example, the population size in Definition~\ref{def:finite} and
Proposition~\ref{PP:02} is not assumed anymore to be constant, and
$\tau$ denotes the time when the population goes extinct.  We use
Proposition~\ref{PP:02} for the convergence of the tree-valued Moran
dynamics to the tree-valued Fleming-Viot dynamics in the proof of
Theorem \ref{T:02} (where we have a constant population size and
$\tau=\infty$.) Our compact containment condition will be applicable
also in the construction of evolving $\Lambda$-coalescents or
branching trees.

\begin{definition}[Finite population dynamics]
  Let $(\Omega, (\mathcal A_t)_{t\geq 0}, \mathbf P)$ be a
  filtered \label{def:finite} probability space. Let $\mathcal
  I=(\mathcal I_t)_{t\in\R}$ be an adapted process with values in
  $\{\{1,...,n\}:\, n\in\mathbb N_0\}$. For each $t\in\R$, we refer to
  $\mathcal I_t$ as the population at time~$t$. Furthermore, let
  $\preceq =(\preceq_t)_{t\ge 0}$ be an adapted process such that for
  all $t\ge 0$, $\preceq_t$ is a partial order on $\{(i,s):
  s\in(-\infty,t],\,i\in \mathcal I_s\}$ which defines the
  genealogical relationships at all times before $t$ and satisfies the
  following:
  \begin{enumerate}
  \item[(i)] for all $r,s,t \in \mathbb R$ with $0,r\leq s\leq t$,
    $i_r\in \mathcal I_r$, and $i_s\in \mathcal I_s$, $(i_r, r)
    \preceq_s (i_s,s)$ implies that $(i_r,r) \preceq_t (i_s,s)$,
    i.e., order relations from earlier times are
    preserved,
  \item[(ii)] for all $i\in \mathcal I_t$ and $s\leq t$ there is a
    unique $A_s(i,t)\in \mathcal I_s$ such that $(A_s(i,t),s)\preceq_t
    (i,t)$. We say that $A_s(i,t)$ is the ancestor of $i$ at time $s$,
  \item[(iii)] for all $i,j\in \mathcal I_0$ there is an almost surely
    finite time $T^0_{ij}$ such that $A_{T^0_{ij}}(i,t) =
    A_{T^0_{ij}}(j,t)$, i.e., all individuals at time $t=0$ are
    related.
  \end{enumerate}
  Let $\tau:=\inf\{s\ge 0:\,\mathcal I_s=\emptyset\}$ be the lifetime of the
  population. Put then for all $t\le\tau$ and $i,j\in \mathcal I_t$,
  \begin{align}\label{eq:854}
    r_t(i,j)
  :=
    2\big(t - \sup\big\{s\le t:\, A_s(i,t) = A_s(j,t)\big\}\big).
  \end{align}

  The tree-valued population dynamics $(\mathcal U_t)_{t\in[0,\tau)}$
  read off from $(\mathcal I,\preceq)$ and is defined as follows: for
  all $t\in[0,\tau)$,
  \begin{align}\label{eq:853}
    \mathcal U_t:=\overline{(\mathcal I_t, r_t, \tfrac{1}{|\mathcal
        I_t|}\sum_{i\in \mathcal I_t} \delta_i)} \in \mathbb U.
  \end{align}
\end{definition}\sm

\begin{figure}
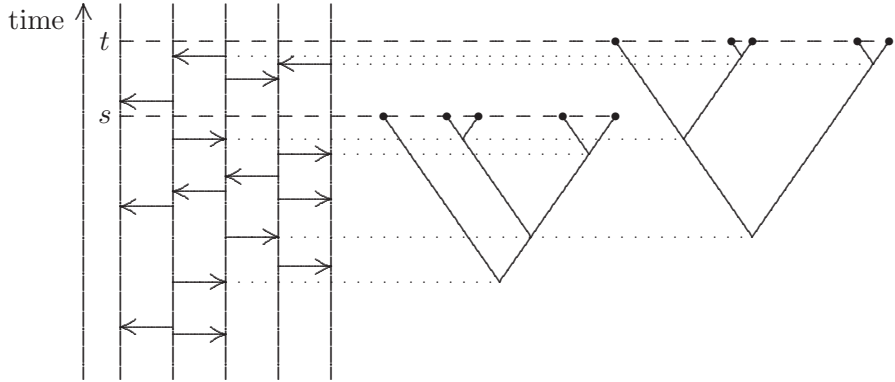

\beginpicture
\setcoordinatesystem units <.7cm,1cm>
\setplotarea x from 0.5 to 12, y from 0 to 5
\plot 1 0 1 5 /
\plot 2 0 2 5 /
\plot 3 0 3 5 /
\plot 4 0 4 5 /
\plot 5 0 5 5 /
\arrow <0.2cm> [0.375,1] from 2 0.6 to 3 0.6
\arrow <0.2cm> [0.375,1] from 2 0.7 to 1 0.7
\arrow <0.2cm> [0.375,1] from 2 1.3 to 3 1.3
\arrow <0.2cm> [0.375,1] from 4 1.5 to 5 1.5
\arrow <0.2cm> [0.375,1] from 3 1.9 to 4 1.9
\arrow <0.2cm> [0.375,1] from 2 2.3 to 1 2.3
\arrow <0.2cm> [0.375,1] from 4 2.4 to 5 2.4
\arrow <0.2cm> [0.375,1] from 3 2.5 to 2 2.5
\arrow <0.2cm> [0.375,1] from 4 2.7 to 3 2.7
\arrow <0.2cm> [0.375,1] from 4 3.0 to 5 3.0
\arrow <0.2cm> [0.375,1] from 2 3.2 to 3 3.2
\arrow <0.2cm> [0.375,1] from 2 3.7 to 1 3.7
\arrow <0.2cm> [0.375,1] from 3 4.0 to 4 4.0
\arrow <0.2cm> [0.375,1] from 5 4.2 to 4 4.2
\arrow <0.2cm> [0.375,1] from 3 4.3 to 2 4.3

\plot 8.2 1.3 6 3.5 /
\plot 8.2 1.3 10.4 3.5 /
\plot 7.2 3.5 8.8 1.9 /
\plot 7.8 3.5 7.5 3.2 /
\plot 9.4 3.5 9.9 3.0 /
\setdots
\plot 2 1.3 8.2 1.3 /
\plot 3 1.9 13 1.9 /
\plot 4 3.0 9.9 3.0 /
\plot 2 3.2 11.7 3.2 /
\plot 3 4.3 12.8 4.3 /
\plot 5 4.2 15.3 4.2 /

\put{\tiny$\bullet$} [cC] at 6 3.5
\put{\tiny$\bullet$} [cC] at 7.2 3.5
\put{\tiny$\bullet$} [cC] at 7.8 3.5
\put{\tiny$\bullet$} [cC] at 10.4 3.5
\put{\tiny$\bullet$} [cC] at 9.4 3.5

\put{\tiny$\bullet$} [cC] at 10.4 4.5
\put{\tiny$\bullet$} [cC] at 15.6 4.5
\put{\tiny$\bullet$} [cC] at 15 4.5
\put{\tiny$\bullet$} [cC] at 13 4.5
\put{\tiny$\bullet$} [cC] at 12.6 4.5

\setsolid
\plot 10.4 4.5 13 1.9 /
\plot 13 1.9 15.6 4.5 /
\plot 15 4.5 15.3 4.2 /
\plot 13 4.5 11.7 3.2 /
\plot 12.6 4.5 12.8 4.3 /

\arrow <0.2cm> [0.375,1] from 0.3 0 to 0.3 5
\put{time} [cC] at -.6 4.8
\setdashes
\plot 1 3.5 10.4 3.5 /
\plot 1 4.5 15.6 4.5 /
\put{$s$} [cC] at 0.7 3.5
\put{$t$} [cC] at 0.7 4.5
\endpicture
\caption{\label{fig:MM}
The graphical representation of a Moran model
  of size $N=5$. By resampling the genealogical relationships between
  individuals change. Arrows between lines indicate resampling events.
  The individual at the tip dies and the other one reproduces. At any
  time, genealogical relationships of individuals $\bullet$, which are
  currently alive, can be read from this graphical representation.}
\end{figure}

For a particular choice of $(\mathcal I,\preceq)$ we obtain the Moran
dynamics. (Compare also with Figure~\ref{fig:MM}).

\begin{definition}[Tree-valued Moran dynamics of population size $N$]
  Fix $N\in\N$.  The tree-valued population Moran dynamics with
  population size $N$ is the tree-valued population dynamics read off
  from $(\mathcal I,\preceq)$ as follows: Put $\mathcal I = (\mathcal
  I_t)_{t\in\mathbb R}$ with ${\mathcal I}_t := \mathcal I^N
  :=\{1,2,...,N\}$ for all $t\in\mathbb R$.
  \label{def:moran}
  Let   $\preceq_0$ be a  random partial order on $(-\infty, 0]\times{\mathcal
    I^N}$ which satisfies~(iii) in Definition~\ref{def:finite}, almost surely.
  Consider an independent family of rate~$\tfrac{\gamma}{2}$-Poisson
  processes $\eta:=\{\eta^{i,j};\,i,j\in \mathcal I^N\}$.  (Note that
  at time $\eta^{i,j}$ an arrows from $i$ to $j$ appears in the
  graphical representation, Figure~\ref{fig:MM}.)

  For any $s,t\in\R$
  with $0\le s\le t$ and $i_s,i_t\in{\mathcal I^N}$, we say that
  $(i_s,s)\preceq_t(i_t,t)$ iff there is a path of descent from
  $(i_s,s)$ to $(i_t,t)$, i.e., if there exist $n\in\N$, $s=:u_0 \le
  u_1<u_2< ...<u_n:= t$ and $j_1:=i_s,
  j_n:=i_t,j_1,...,j_{n-1}\in\{1,2,...,N\}$ such that for all
  $k\in\{1,...,n\}$, $\eta^{j_{k-1},j_k}\{u_k\}=1$ and
  $\eta^{m,j_{k-1}}(u_{k-1},u_k) = 0$ for all $m \in \CI^N$.
\end{definition}\sm

In empirical population genetics models for finite populations rather
than infinite populations are of primary interest.  The next result
states that the known convergence of Moran to Fleming-Viot dynamics
holds also on the level of trees.

\begin{theorem}[Convergence of Moran to Fleming-Viot dynamics]
  For $N\in\N$, let $\mathcal U^N:=(\mathcal U_t)_{t\geq 0}$ be the
  tree-valued Moran dynamics of population size $N$, and
  let \label{T:02} $\mathcal U = (\mathcal U_t)_{t\geq 0}$ be the
  tree-valued Fleming-Viot dynamics.  If $\mathcal
  U^N_0\Rightarrow\mathcal U_0$ weakly with respect to the Gromov-weak
  topology, as $N\to\infty$, then
  \begin{equation}\label{ag2}
    \mathcal U^N \TNo \mathcal U,
  \end{equation}
  weakly with respect to the Skorohod topology on ${\mathcal
    D}_{\mathbb{U}}([0,\infty))$.
\end{theorem}\sm

\begin{remark}[Connection with the look-down process]
  Since all the information about trees seems to be contained in the
  look-down construction of \cite{DonKur1996}, one might wonder
  whether one could read off the tree-valued Fleming-Viot dynamics
  from there.  This works for the well-posedness of the Fleming-Viot
  martingale problem as we want to sketch here shortly.  Recall that
  the look-down construction contains the tree-valued Moran dynamics
  for different population sizes on the same probability space as
  follows: Put $\mathcal I\equiv \mathbb{N}$. Choose a partial order
  $\preceq_0$ on $(-\infty, 0]\times\mathbb{N}$ which satisfies~(iii)
  in Definition~\ref{def:finite}.  Consider an independent family of
  rate~$\gamma$-Poisson processes $\eta:=\{\eta^{i,j};\,1\le i<j\}$.
  As in Definition~\ref{def:moran}, let for any $0\le s\le t$ and
  $i_s,i_t\in\mathbb{N}$, $(i_s,s)\preceq_t(i_t,t)$ iff there is a
  path of descent from $(i_s,s)$ to $(i_t,t)$. As in (\ref{eq:854}) we
  can define a process $(\underline{\underline R}_t)_{t\geq 0}$ with
  $\underline{\underline R}_t := (R_t(i,j))_{1\leq i<j}$ which satisfy
  for all $1\le i<j$,
  \begin{equation} \label{eq:9111}
    \begin{aligned}
      R_t(i,j) & \text{ grows linearly at speed 2},\\
      R_t(i,j) & = 0 \text{ if }t\in \eta^{i,j},\\
      R_t(i,j) & = R_{t-}(k,j) \text{ if }t \in \eta^{k,i} \text{ for
        some $k<i$},\\
      R_{t}(i,j) & = R_{t-}(i\wedge k, i\vee k) \text{ if }t\in
      \eta^{k,j} \text{ for some $i\neq
        k<j$}.
    \end{aligned}
  \end{equation}
  \sloppy If $\preceq_0$ is exchangeable, the tree-valued population
  dynamics $({\mathcal U}^N_t)_{t\ge 0}$ read off from the restricted
  graphical representation $(\{1,2,3,...,N\},\preceq)$ equals the
  tree-valued Moran dynamics, for each $N\in\mathbb{N}$.  Moreover,
  the almost sure limit
  \begin{equation} \label{eq:9113} \mathcal U^{\infty}_t :=
    \lim_{N\to\infty}{{\mathcal U}^N_t}
  \end{equation}
  exists for all $t\geq 0$. (Compare with Theorem~4 in
  \cite{GrePfaWin2006a}). This limit easily extends to finitely many
  time points.  By the Kolmogorov extension theorem, existence of a
  process $(\mathcal U^{\infty}_t)_{t\geq 0}$ with these finite
  dimensional distributions follows, as well as convergence of finite
  Moran models in finite dimensional distributions. In addition, with
  a bit more effort it is possible to show that there is a
  modification of $(\mathcal U^{\infty}_t)_{t\geq 0}$ with continuous
  sample paths, as an estimate of $\mathbb{E}[(\Phi({\mathcal
    U}_t)-\Phi({\mathcal U}_s))^4]$ for $\Phi\in\Pi^1$ reveals.

  The process $(\mathcal U^{\infty}_t)_{t\geq 0}$ solves the
  martingale problem for $\Omega^\uparrow$.  Indeed, if $\mathbf
  P_0\in\mathcal M_1(\mathbb U)$ is independent of the Poisson
  processes and its first moment measure equals the distribution of
  $(R_{0}(i,j))_{1\leq i< j}$, then the process
  $(\underline{\underline R}_t)_{t\geq 0}$ is the unique strong Markov
  process with generator $\widetilde{\Omega}$ acting on functions
  $\phi \in\mathcal C_1^b\big(\mathbb R_+^{\binom {\mathbb N} 2}\big)$
  which depend only on finitely many coordinates given by
  \begin{equation} \label{e:tildeOmega}
    \widetilde{\Omega}\phi:=\langle\nabla \phi, \underline{\underline
      {2}}\rangle+ {\gamma}\sum_{1\le
      k<l}\big(\phi\circ\theta_{k,l}-\phi\big)
  \end{equation}
  with $\theta_{k,l}$ as in \eqref{pp11b}.  That is, for $\Phi =
  \Phi^{\phi,n} \in \Pi^1$,
  \begin{equation} \label{e:test}
    \Omega^\uparrow\Phi^{\phi,n}(\smallu)=\Phi^{\widetilde{\Omega}\phi,n}(\smallu)
    = \langle
    \nu^{\smallu},\widetilde{\Omega}\phi\rangle
  \end{equation}
  and therefore by exchangeability, for all $\phi\in\mathcal
  C_b(\mathbb R^{\binom{\mathbb N}{2}})$,
  \begin{equation} \label{e:jochen}
    \mathbb{E}\big[\langle\nu^{{\mathcal
        U}^\infty_t},\phi\rangle\big]=\mathbb{E}\big[\phi\big(\underline{\underline
      R}_t\big)\big].
  \end{equation}
  Since distance matrix distributions are determined uniquely by their
  first moment measure (this follows since polynomials are separating,
  see Remark~\ref{rem:05}), the process $(\mathcal U_t^\infty)_{t\geq
    0}$ solves the $(\mathbf P_0, \Omega^\uparrow, \Pi^1)$-martingale
  problem. 

  However, the above arguments establish convergence of Moran models
  to the tree-valued Fleming-Viot dynamics only in finite-dimensional
  distributions.  A proof of tightness of Moran models in $\mathcal
  D_{\mathbb U}([0,\infty))$ must be carried out to obtain a full
  convergence result as stated in Theorem~\ref{T:02}.  We therefore
  follow a different route, which also has the advantage of not
  explicitely relying on exchangeability. Hence our approach allows
  also for the construction of tree-valued dynamics coming from
  population models with selection and recombination, or more
  generally, also from tree-valued Markov chains arising outside the
  context of population models.
  \label{Rem:14}
  \hfill$\qed$
\end{remark}\sm

\begin{remark}[Universality]\sloppy
  The measure-valued Fleming-Viot process is universal in the sense
  that it is the limit point of frequency paths of various
  exchangeable population models of constant size. (A precise
  condition is found in \cite{MoehleSagitov2001}.)  We conjecture that
  the same universality holds on the level of trees, i.e., the
  tree-valued Fleming-Viot dynamics is the point of attraction of
  various exchangeable tree-valued dynamics. The crucial step for
  convergence of tree-valued processes is tightness of the finite
  models; see Section~\ref{sub:cpapp} in the case of the tree-valued
  Moran dynamics.\label{Rem:15} \hfill $\qed$
\end{remark}\sm

The proof of Theorem~\ref{T:02} relies on a criterion for the compact
containment condition in $\mathbb U$ to hold. We state it here for the
class of population dynamics given in Definition~\ref{def:finite}. It
is based on the number of ancestors and descendants.

For $t\in[0,\tau)$ and $\varepsilon>0$, denote by
\begin{equation}
    \label{eq:861}
    \begin{aligned}
      S_{2\varepsilon} (\mathcal U_t) &:=
      \#\big\{A_{t-\varepsilon}(i,t): i\in \mathcal I_{t}\big\}
    \end{aligned}
  \end{equation}
the {\em number of ancestors of
  $\mathcal I_t$ at time $t-\varepsilon$}, and by
\begin{equation}
  \label{eq:861b}
  \begin{aligned}
    \widetilde S_{2\varepsilon}(\mathcal U_t) &:= \inf_{\mathcal
      J\subseteq \mathcal I_t:\;\#\mathcal J\le
      2\varepsilon\#{\mathcal I}_t}\#\big\{A_{t-\varepsilon}(i,t):\,
    i\in\mathcal I_t\setminus \mathcal J\big\}
  \end{aligned}
\end{equation}
the {\em minimal number of ancestors at time $t-\varepsilon$ whose
  descendants cover a fraction of at least $1- 2\varepsilon$ of the
  time-$t$-population}. For $t\geq \tau$ and $\varepsilon>0$, set
$\widetilde S_{2\varepsilon}(\mathcal U_t) = S_{2\varepsilon}(\mathcal
U_t)=0$. Moreover, for $\mathcal J\subseteq \mathcal I_s$ and $s\leq
t$, let
\begin{align}
  \label{eq:862}
  D_t(s,\mathcal J) := \#\big\{i\in \mathcal I_t:\, A_s(i,t)\in
  \mathcal J\big\}
\end{align}
denote the {\em number of descendants of the set $\mathcal J$ at
  time $t$}.  \sm

The following criterion for a compact containment condition will be
proved in Section~\ref{sub:51}. It uses the setting of finite
population models from Definition~\ref{def:finite}. Recall that the
population size is in general not constant and $\tau$ refers to the
time the population goes extinct.

\begin{proposition}[Compact containment for population dynamics]
  For each $N\in\mathbb N$, let $(\Omega^N,(\mathcal
  A_t^N)_{t\in\R},\mathbf P^N)$, $(\mathcal I^N,\preceq^N)$, and
  $\tau^N$ be as in Definition~\ref{def:finite}. Let $\mathcal
  U^N=(\mathcal U_t^N)_{t\in[0,\tau^N)}$ be the tree-valued population
  dynamics read off from $(\mathcal I^N, \preceq^N)$.

  Assume that the family
  $\{\mathcal U_0^N;\,N\in\N\}$ is tight in $\mathbb
  U$.  \label{PP:02} Furthermore fix $T>0$, and consider the following assumptions:
  \begin{itemize}
  \item[(i)] For all $0<\varepsilon< T$ there exists a
    $\delta=\delta(\varepsilon)>0$ such that for all $s\in[0,T)$,
    $N\in\N$ and ${\mathcal A}^N_{s}$-measurable random subsets
    $\mathcal J^N\subset \mathcal I_s^N$ with $\#{\mathcal J^N}\le
    \delta\cdot \#\mathcal I_s^N$,
    \begin{equation}
      \label{P:171}
      \limsup_{N\in\mathbb N}\mathbf P^N \big\{\sup_{t\in[s,T\wedge\tau^N)}
      \tfrac{D_t^N(\mathcal J^N,s)}{\#\mathcal I_t^N}
      >\varepsilon\big\} \le
      \varepsilon.
    \end{equation}
  \item[(ii.i)] For all $0 < \varepsilon \leq t < T$, the family
    $\{S_{2\varepsilon}(\mathcal U_t^N):\,N\in\N\}$ is tight.
  \item[(ii.ii)] For all $0 < \varepsilon \leq t < T$, the family
    $\{\widetilde S_{2\varepsilon}(\mathcal U_t^N):\,N\in\N\}$ is
    tight.
  \end{itemize}
  Then, the following compact containment conditions hold:
  \begin{itemize}
  \item[(a)] Under (i) and (ii.i), for all $\varepsilon >0$ there
    exists a set $\Gamma_{\varepsilon,T}\subseteq\mathbb U_c$ which is
    compact in $\mathbb U_c$ such that
    \begin{equation}\label{eq:P02b}
      \inf_{N\in\mathbb N}\mathbf{P}^N\big\{\mathcal
      U^N_t\in\Gamma_{\varepsilon,T} \mbox{ for all } t\in[\varepsilon,T\wedge\tau^N)\big\}
      > 1-\varepsilon.
    \end{equation}
  \item[(b)] Under (i) and (ii.ii), for all $\varepsilon >0$ there
    exist a set $\widetilde\Gamma_{\varepsilon,T}\subseteq\mathbb U$
    which is compact in $\mathbb U$ such that
    \begin{equation}\label{eq:P02c}
      \inf_{N\in\mathbb N}\mathbf P^N \big\{\mathcal
      U^N_t\in\widetilde \Gamma_{\varepsilon, T} \mbox{ for all }
      t\in[0,T\wedge\tau^N)\big\} > 1-\varepsilon.
    \end{equation}
  \end{itemize}
\end{proposition}\sm

\subsection{Long-term behavior (Theorem~\ref{T:03})}
\label{Sub:long}
Genealogical relationships in neutral models are frequently studied
since the introduction of the Kingman coalescent in
\cite{Kingman1982a}. This stochastic process describes the genealogy
of a Moran population in equilibrium and its projective limit as the population size tends to infinity.
In this section we formulate the related
convergence result for  the tree-valued resampling dynamics.

Recall that a partition of $\N$ is a collection $\smallp = \{\pi_1,
\pi_2, \ldots\}$ of pairwise disjoint subsets of $\N$, also called
{\em blocks}, such that $\N=\cup_i \pi_i$. The partition $\smallp$
defines an equivalence relation $\sim_{\smallp}$ on $\N$ by
$i\sim_{\smallp}j$ if and only if there exists a partition element
$\pi\in\smallp$ with $i,j\in\pi$.  We denote by $\mathbb S$ the set of
partitions of $\N$ and define for each $k\in\N$ the restriction
$\rho_k$ on $\mathbb S$ to the set $\mathbb{S}_k$ of partitions of
$\{1,2,...,k\}$ by $\rho_k\circ\smallp := \{\pi_i\cap\{1,...,k\}:
\pi_i\in\smallp\}$. Each $\smallp\in\mathbb{S}$ can be identified with
the sequence
$(\rho_1\circ\smallp,\rho_2\circ\smallp,...)\in\mathbb{S}_1\times\mathbb{S}_2\times
...$.  Give $\mathbb{S}$ the topology it inherits as a subset of
$\mathbb{S}_1\times\mathbb{S}_2\times ...$ with the product of
discrete topologies. So $\mathbb{S}$ is compact and metrizable and
hence Polish.

Starting in $\mathcal P_0 = \smallp\in\mathbb{S}$, the {\em Kingman
  coalescent} is the unique $\mathbb S$-valued strong Markov process
$\mathpzc K=(\mathpzc K_s)_{s\geq 0}$ such that any pair of blocks
merges at rate $\gamma$ (see, for example,
\cite{Kingman1982b,Pit1999}).  Every realization $\smallk =
(\smallk_s)_{s\geq 0}$ of $\mathpzc K$ gives a pseudo-metric
$r^{\smallk}$ on $\N$ defined by
\begin{equation}\label{e:lambda2}
  r^{\smallk}\big(i,j\big) := 2\cdot \inf\big\{s\ge
  0:\,i\sim_{\smallk_s}j\big\},
\end{equation}
i.e., $r^{\smallk}\big(i,j\big)$ is proportional to the time needed
for $i$ and $j$ to coalesce. Note that $(\N, r^{\smallk})$ is
ultra-metric and  that $r^\smallk(i,j)$
can be thought of as a genealogical distance. Denote then by
$(L^{\smallk},r^{\smallk})$ the completion of $(\N,r^{\smallk})$.
Clearly, $(L^{\smallk}, r^{\smallk})$ is also ultra-metric.  Define
$H_N$ to be the map which takes a realization of the
$\mathbb{S}$-valued coalescent and maps it to (an equivalence class
of) a pseudo-metric measure space by
\begin{equation}\label{e:Hn}
  H_N:\,\smallk\mapsto \overline{\Big(L^{\smallk},r^{\smallk},\mu^{\mathpzc K}_N :=
    \tfrac{1}{N}\sum\nolimits_{i=1}^N\delta_i\Big)}.
\end{equation}
Notice that for each $N$, the map $H_N$ is continuous.

By Theorem~4 in \cite{GrePfaWin2006a}, there exists a $\mathbb
U$-valued random variable $\mathcal U_\infty$ such that
\begin{equation}
\label{ag2b}
H_N(\mathpzc K)\TNo \mathcal U_\infty,
\end{equation}
weakly with respect to the Gromov-weak topology. The limit object $\mathcal
U_\infty$ is called the \emph{Kingman measure tree}. Since the Kingman
coalescent comes immediately down from $\infty$, the Kingman measure
tree is compact (see \cite{Evans2000}).

\begin{theorem}[Convergence to the Kingman measure tree]
    Let $\mathcal U=(\CU_t)_{t\geq 0}$ be the tree-valued
  Fleming-Viot
  \label{T:03} dynamics starting in $\mathcal U_0$ and ${\mathcal
    U}_\infty$ the Kingman coalescent measure tree. Then
\begin{equation}\label{ag4}
  \CU_t \Tto \mathcal U_\infty.
\end{equation}
In particular, the distribution of $\mathcal U_\infty$ is the unique
equilibrium distribution of the tree-valued Fleming-Viot dynamics.
\end{theorem}\sm

\begin{remark}[Exchange of limits] Recall from Definition
  \ref{def:moran} the tree-valued Moran dynamics $\{\mathcal
  U^N=({\mathcal U}^N_t)_{t\ge 0};N\in\N\}$. It is straightforward to
  check that for all $N\in\N$ and for all possible initial states,
  $\mathcal U^N_t\Tto H_N(\mathpzc K)$, and therefore the limits
  $N\to\infty$ (see Theorem \ref{T:02}) and $t\to\infty$ (see Theorem
  \ref{T:03}) can be exchanged due to \eqref{ag2b}.  \hfill$\qed$
\end{remark}

\sm

\subsubsection*{Outline}
The rest of the paper is organized as follows. As an application we
study the evolution of subtree length distributions in
Section~\ref{S:kurtz}. A duality relation of the tree-valued
Fleming-Viot dynamics to the tree-valued Kingman coalescent is given
in Section~\ref{S:unique}. In Section~\ref{S:finite} we give a formal
construction of tree-valued Moran dynamics using well-posed martingale
problems. The Moran models build, as shown in Section~\ref{sub:51}, a
tight sequence. Duality and tightness provide the tools necessary for
the proof of Theorems~\ref{T:01} through~\ref{T:03}, which are carried
out in Section~\ref{S:proofs}. In Section~\ref{S.proofapp} we give the
proofs of the applications of Section~\ref{S:kurtz}.\sm

\section[Subtree length distributions]{Application: Subtree length
  distribution (Theorems~\ref{T:04} and~\ref{T:05})}
\label{S:kurtz}
In this section we investigate the distribution of the vector
containing the lengths of the subtrees spanned by subsequently sampled
points, which is referred to as the \emph{subtree length
  distribution}. All proofs are given in Section~\ref{S.proofapp}.

The main result in Subsection~\ref{Sub:subtree} is that the subtree
length distribution uniquely determines ultra-metric measure spaces
(Theorem 4). In Subsection~\ref{Sub:mpLambda} the corresponding
martingale problem and its well-posedness is stated (Theorem 5). In
Subsection~\ref{Sub:explicit} we study with the mean sample Laplace
transform a special functional of the subtree length distribution.
\sm

\subsection{The subtree length distribution (Theorem~\ref{T:04})}
\label{Sub:subtree}
Recall from Remark~\ref{Rem:01} that we can isometrically embed any
ultra-metric space $(U,r_U)$ via a function $\varphi$ into a
path-connected space $(X,r_X)$ which satisfies the four-point
condition \eqref{grev30} such that $X\setminus X^\circ$ is isometric
to $(U,r_U)$. For a sequence $u_1,...,u_n\in U$ with $n\in\N$, let
\begin{equation}\label{e:length}
\begin{aligned}
  L_n^{(U,r_U)}\big(& \{u_1,...,u_n\}\big) := L_n^{(X,r_X)}\big(
  \{\varphi(u_1),...,\varphi(u_n)\}\big)
  \\
  &:= \mbox{length of the subtree of $(X,r)$ spanned by
    $\{\varphi(u_1),...,\varphi(u_n)\}$},
\end{aligned}
\end{equation}
where for an $\R$-tree $(X,r_X)$ with finitely many leaves the length
of the tree is defined as the total mass of the one-dimensional
Hausdorff measure on $(X,{\mathcal B}(X))$.

Note that the length of the tree spanned by a finite sample is a
function of their mutual distances as we state next.

\begin{lemma}[Total length of a sub-tree spanned by a finite subset]
 \label{L:02}
 For a metric space $(X,r_X)$ satisfying the four point condition
 \eqref{grev30} and for $x_1,...,x_n\in X$,
\begin{equation}\label{pp4}
\begin{aligned}
  L_n^{(X,r_X)}\big(\{x_1,...,x_n\}\big) &=
  \frac{1}{2}\inf\big\{\sum_{i=1}^n
  r(x_i,x_{\sigma(i)});\,\sigma\in\Sigma^1_n\big\},
\end{aligned}
\end{equation}
where $\Sigma^1_n:=\{\text{permutations of $\{1,...,n\}$ with one cycle}\}$.
\end{lemma}\sm

To specify the distribution of the length of the subtrees of
subsequently sampled points we consider the map
\begin{equation}
\begin{aligned}
\label{e:lunder}
   \underline\ell:\left\{\begin{array}{cl}\R_+^{\N\choose 2}&\to\R_+^\N\\
   \underline{\underline{r}}&\mapsto (0,\ell_2(\underline{\underline{r}}),\ell_3(\underline{\underline{r}}),...),\end{array}\right.
\end{aligned}
\end{equation}
where for each $n\in\N$,
\begin{equation}\label{qqq8}
\begin{aligned}
   \ell_n(\underline{\underline r})
 :=
   \frac{1}{2}\inf\big\{\sum_{i=1}^n
   r_{i,\sigma(i)};\,\sigma\in\Sigma^1_n\big\}.
  \end{aligned}
\end{equation}
We then define the \emph{subtree length distribution} of
$\smallu\in\mathbb U$ by
\begin{equation}\label{grev33}
  \xi(\smallu)
  :=
  \underline\ell_\ast\nu^{{\smallu}}\in{\mathcal M}_1(\R_+^\N).
\end{equation}

The first key result states that the subtree length distribution
uniquely characterizes ultra-metric measure spaces.

\begin{theorem}[Uniqueness and continuity of tree lengths
  distribution]
  The map $\xi:\,\mathbb U\to{\mathcal M}_1(\R_+^\N)$ from
  \eqref{grev33} is injective. Let $\xi(\mathbb U)\subseteq {\mathcal
    M}_1(\R_+^\N)$ be equipped with the weak topology and $\R_+^\N$
  with the product topology. Then, $\xi$ and $\xi^{-1}: \xi(\mathbb
  U)\to\mathbb U$ are continuous.
  \label{T:04}
\end{theorem}\sm

\begin{remark}[$\xi(\mathbb U)$ is Polish]
  Take a complete metric $d^{\mathbb U}$ on $\mathbb U$.  Using the
  injectivity of $\xi$, we define a metric $d^{\xi(\mathbb U)}$ on
  $\xi(\mathbb U)$ by setting
  \begin{align}\label{dxi}
    d^{\xi(\mathbb U)}(\lambda_1, \lambda_2) := d^{\mathbb
      U}(\xi^{-1}(\lambda_1), \xi^{-1}(\lambda_2)), \qquad \lambda_1,
    \lambda_2\in \xi(\mathbb U).
  \end{align}
  Since both, $\xi$ and $\xi^{-1}$, are continuous (with respect to
  the weak topology on $\mathcal M_1(\mathbb U)$), we see that
  $d^{\xi(\mathbb U)}$ generates the weak topology on $\xi(\mathbb
  U)$. Since $\xi(\mathbb U)$ inherits the separability from $\mathbb
  U$, we conclude that $\xi(\mathbb U)$ is Polish.  \hfill$\qed$
\end{remark}\sm

\begin{remark}[Conjecture about general tree spaces]
  Theorem \ref{T:04} shows uniqueness of the tree length distribution
  on the space of ultra-metric spaces. We conjecture that uniqueness
  still holds on the space of metric measure spaces satisfying the
  four-point condition (\ref{grev30}).  \hfill$\qed$
  \label{rem:conjGen}
\end{remark}\sm

\subsection{Martingale problem of subtree length
  distribution~(Theorem~\ref{T:05})}
\label{Sub:mpLambda}
We investigate the evolution of the subtree length distribution under
the tree-valued Fleming-Viot dynamics.
 That is, given the tree-valued Fleming-Viot dynamics
$\mathcal U=(\mathcal U_t)_{t\geq 0}$, we consider
\begin{equation}\label{Lambd}
  \Xi=(\Xi_t)_{t\ge 0}, \qquad \Xi_t := \xi(\mathcal U_t).
\end{equation}
To describe the process $\Xi$ via a martingale problem, we define the
operator $\Omega^{\uparrow,\Xi}$ on the algebra $\Pi^{\Xi} := \{
\Phi\circ\xi^{-1}: \Phi\in\Pi\}$ with domain $\Pi^{1,\Xi} := \{
\Phi\circ\xi^{-1}: \Phi\in\Pi^1\}$ by
\begin{equation}
\label{e:Omegalength}
  \Omega^{\uparrow,\Xi}(\Phi\circ\xi^{-1})(\lambda) :=
  \Omega^{\uparrow}\Phi(\xi^{-1}(\lambda)),
\end{equation}
for all $\lambda\in\xi(\mathbb{U})$.

In $\Pi^{1,\Xi}$ we find, in particular, functions $\Psi\in\Pi^{1,\Xi}$ which are of the form
\begin{equation}\label{psilam}
  \Psi^\psi(\lambda) = \langle \lambda, \psi\rangle :=
  \int_{\R_+^{\N}}\lambda(\mathrm{d}\underline{{l}})\,\psi\big(\underline {{l}}\big),
\end{equation}
for a test function $\psi\in{\mathcal C}_b^1(\R_+^{\mathbb N})$ which depends on finitely many entries only.
Indeed, if $\psi$ depends only on the first $k$ entries, then
$\Psi^\psi=\Phi^{k,\psi\circ\ell}\circ\xi^{-1}$.% for an $k'\le k$.

The main result of the section is the following.

\begin{theorem}[The subtree lengths distribution process]
  For $\mathbf{P}_0\in{\mathcal M}_1(\mathbb{U})$, let $\mathcal U$ be
  the tree-valued Fleming-Viot dynamics with initial distribution
  $\mathbf P_0$. \label{T:05}
  \begin{itemize}
  \item[(i)] The $(\xi_\ast \mathbf{P}_0,\Omega^{\uparrow,\Xi}
    ,\Pi^{1,\Xi})$-martingale problem is well-posed. Its unique
    solution is given by $\Xi = (\Xi_t)_{t\geq 0}$ with $\Xi_t =
    \xi(\mathcal U_t)$, for $t\ge 0$. The process $\Xi$ has the Feller
    property. In addition, $\mathbf{P}$-almost surely, it has
    continuous sample paths, where $\xi(\mathbb U)\subseteq \mathcal
    M_1(\mathbb R_+^{\mathbb N})$ is equipped with the weak topology.
  \item[(ii)] The action of $\Omega^{\uparrow, \Xi}$ on a function
    $\Psi^\psi$ of the form \eqref{psilam} is given by
    \label{l:genY}
    \begin{equation}\label{eq:genY4} \Omega^{\uparrow,\Xi}\Psi^\psi(\lambda)
      = \sum_{n\ge 2} n \big\langle \lambda,
      \frac{\partial}{\partial{{l}}_n}\psi\big\rangle + \gamma \sum_{n\ge
        1} n\big\langle \lambda, \psi\circ\beta_{n}-\psi\big\rangle
    \end{equation}
    where % Define for each $n\in\N$, the {\em doubling-shift
          % operator}
    $\beta_n:\{0\}\times\R_+^{\N}\to\{0\}\times\R_+^{\N}$ is given by
    \begin{equation}\label{eq:genY2}
      \beta_n:\,({{l}}_1=0,{{l}}_2,{{l}}_3,...)
      \mapsto
      ({{l}}_1=0,{{l}}_2,...,{{l}}_{n-1},{{l}}_n,{{l}}_n,{{l}}_{n+1},...).
    \end{equation}
  \end{itemize}
\end{theorem}\sm

\subsection{Explicit calculations}
\label{Sub:explicit}
We consider in this section the {\em mean sample Laplace transforms},
i.e., functions of the form (\ref{psilam}) with test functions
\begin{align}\label{eq:gpw1}
  \psi(\underline{{l}}) = e^{-\sigma {{l}}_n}
\end{align}
for some $n\in\N$ and $\sigma \in \mathbb R_+$ in \eqref{psilam} for each
$n\in\N$. Using \eqref{eq:genY4} we obtain the following explicit
expressions.

\begin{corollary}[Mean sample Laplace transforms] Let
  $\Xi=(\Xi_t)_{t\geq 0}$ be the solution of the $(\xi_\ast \mathbf
  P_0, \Omega^{\uparrow, \Xi}, \Pi^{1,\Xi})$ martingale problem. For
  all $\sigma\in\mathbb R_+$ and $n\geq 2$, set
  \begin{equation}\label{grev40}
    g^n(t,\sigma) :=
    \mathbf{E}\Big[\int
    \Xi_t(\mathrm{d}\underline{{{l}}})\,e^{-\sigma{{l}}_n}\Big].
  \end{equation}\sm
  Then,
  \begin{equation}\label{e:solu45}
    \begin{aligned}
   &g^{n}(t,\sigma)
  \\
   &= \frac{\Gamma(n)\Gamma\big( \tfrac 2\gamma
        \sigma +1\big)}{\Gamma \big(\tfrac 2\gamma \sigma + n \big)} +
      n! \sum_{k=2}^n\frac{{n-1 \choose
          k-1}(-1)^{k}(\frac{2}{\gamma}\sigma+2k-1)}{\Gamma(\frac{2}{\gamma}\sigma+n+k)}
      \cdot e^{-k(\sigma+\frac{\gamma}{2}(k-1))t}
      \\
      &\qquad \cdot\Big\{\Big( \sum_{m=2}^k\frac{{k-1\choose
          m-1}(-1)^{m}\Gamma(\frac{2}{\gamma}
        \sigma+k+m-1)}{m!}g^{m}(0;\sigma)\Big) \\ & \qquad \qquad
      \qquad \qquad \qquad \qquad \qquad -
      \frac{k-1}{k(\tfrac 2\gamma\sigma + k
        -1)}\Gamma(\tfrac{2}{\gamma}\sigma+k+1) \Big\}.
    \end{aligned}
  \end{equation}
  In particular, if $g^n(\sigma)=\lim_{t\to\infty}g^n(t;\sigma)$ then
  \begin{equation}\label{e:ttoinf}
    g^n(\sigma)
    =
    \mathbf{E}\big[e^{-\sigma\sum_{k=2}^n{\mathcal E}^k}\big],
  \end{equation}
  where $\{{\mathcal E}^{k};\,k=2,...,n\}$ are independent and
  $\mathcal E^k$ is exponentially distributed with mean
  $\frac{2}{\gamma(k-1)}$, $k=2,...,n$.
  \label{Cor:04}
\end{corollary}\sm

\begin{remark}[Length of $n$-Kingman coalescent]
  Consider the Kingman coalescent started with $n$ individuals, and
  let $L_n$ denote the total length of the corresponding genealogical
  tree. Note that \eqref{e:ttoinf}, together with Theorem \ref{T:03}
  implies the well-known fact (implicitly stated already in
  \cite{Watterson1975}) that
  \begin{equation}\label{ag7equ}
    L_n \stackrel{d}{=} \sum\nolimits_{k=2}^n \mathcal E^k. \hfill \qed
  \end{equation}
  \label{Rem:24}
\end{remark}\sm

\section{Duality}
\label{S:unique}
Duality is an extremely useful technique in the study of Markov
processes. It is well-known that the Kingman coalescent is dual to the
neutral measure-valued Fleming-Viot process (see, for example,
\cite{Dawson1993,MR1779100}). In this section this duality is lifted
to the tree-valued Fleming-Viot dynamics. We apply the duality to show
uniqueness of the martingale problem for the tree-valued Fleming-Viot
process and its relaxation to the equilibrium Kingman measure tree in
Section~\ref{S:proofs}.

\subsection*{The dual process}
\label{Sub:dualpro}
Recall from Subsection~\ref{Sub:long} the Kingman coalescent $\mathpzc
K=(\mathpzc K_s)_{s\geq 0}$ and its state space $\mathbb S$ of
partitions of $\N$.  Since we are constructing a dual to the
$\mathbb{U}$-valued dynamics, we add a component which measures
genealogical distances. The state space of the dual tree-valued
Kingman coalescent therefore is
\begin{equation}\label{eq:dualStSp}
  \mathbb K
  :=
  \mathbb{S} \times \mathbb R_+^{\binom{\mathbb N}{2}},
\end{equation}
equipped with the product topology. In particular, since $\mathbb{S}$
and $\mathbb R_+^{\binom{\mathbb N}{2}}$ are Polish, $\mathbb{K}$ is
Polish as well.

In the following we call the $\mathbb{K}$-valued stochastic process
${\mathcal K}=({\mathcal K}_s)_{s\ge 0}$, with
\begin{equation}\label{grevy}
  \mathcal K_s = (\mathpzc K_s,\underline{\underline{r}}'_s)
\end{equation}
the {\em tree-valued Kingman coalescent}, if it follows the dynamics:
\begin{itemize}
\item[{}] {\bf Coalescence. } $\mathpzc K=(\mathpzc K_s)_{s\ge 0}$ is
  the $\mathbb{S}$-valued Kingman coalescent with pair coalescence
  rate $\gamma$.
\item[{}] {\bf Distance growth.} At time $t$, for all $1\le i<j$ with
  $i\not\sim_{\mathpzc K_s} j$, the genealogical distance
  $r'_{\boldsymbol{\cdot}}(i,j)$ grows with constant speed~$2$.
\end{itemize}\sm

To state the duality relation it is necessary to associate a
martingale problem with the tree-valued Kingman coalescent. Consider
for $\smallp\in\mathbb{S}$, the {\em coalescent operator}
$\kappa_{\smallp}:\smallp^2\to\mathbb{S}$ such that for
$\pi,\pi'\in\smallp$,
\begin{equation}\label{ppw} \kappa_{\smallp}(\pi,\pi') :=
  \big(\smallp\setminus\{\pi,\pi'\}\big)\cup \big\{\pi\cup\pi'\big\},
\end{equation}
i.e., $\kappa_{\smallp}$ sends two partition elements of the partition
$\smallp$ to the new partition obtained by coalescence of the two
partition elements into one.

We consider the space (recall $\rho_k$ from Subsection~\ref{Sub:long})
\begin{equation}\label{mpmathpi}
\begin{aligned}
     {\mathcal G}
  &:=
     \big\{G\in{\mathcal B}(\mathbb{K}):\,G(\boldsymbol{\cdot},\underline{\underline{r}}')\in{\mathcal C}(\mathbb{S}),G(\boldsymbol{\cdot},\underline{\underline{r}}')\mbox{ depends on }\smallp\mbox{ only through }\\
     &\hspace{3cm}
     \rho_k\circ\smallp\mbox{ for some }k\in\N;\;\forall \underline{\underline{r}}'\in\mathbb{R}_+^{\N\choose 2}\big\}
\end{aligned}
\end{equation}
and the domain
  \begin{equation}\label{mpmathpi2}
    {\mathcal G}^{1,0}
    :=
    \big\{G\in {\mathcal G}:\, \langle \nabla _\smallp^{r'}G, \underline{\underline 2}\rangle\mbox{
      exists},\forall\smallp\in\mathbb{S}\big\}
\end{equation}
with
\begin{equation}\label{e:parti}
  \langle \nabla_\smallp^{r'}\cdot, \underline{\underline 2}\rangle
  % \mathrm{div}_\smallp^{r'}
  :=2\sum_{i \not\sim_{\smallp} j,i<j} \frac{\partial}{\partial r'_{i,j}}=
  \sum_{i \not\sim_{\smallp} j} \frac{\partial}{\partial r'_{i\wedge j,i\vee j}}.
\end{equation}

We then consider the martingale problem associated with the operator
$\Omega^{\downarrow}$ on ${\mathcal G}$ with domain ${\mathcal
  G}^{1,0}$, where $\Omega^\downarrow :=
\Omega^{\downarrow,\mathrm{grow}}+\Omega^{\downarrow,\mathrm{coal}}$,
with
\begin{equation}\label{eq:Gdown1}
  \Omega^{\downarrow,\mathrm{grow}}G(\smallp,\underline{\underline{r}}')
  :=
  \langle\nabla^{r'}_{\smallp}G, \underline{\underline 2}\rangle
  (\smallp,\underline{\underline{r}}')
\end{equation}
and %for $\smallk=(\smallp,\underline{\underline{r}}')$ with $\#\smallp<\infty$,
\begin{equation}\label{eq:Gdown2}
    \Omega^{\downarrow,\mathrm{coal}}G(\smallp, \underline{\underline{r}}')
 :=
    \gamma \sum_{{\{\pi,\pi'\}\subseteq\smallp}\atop{\pi\neq\pi'}}
    \big(G(\kappa_{\smallp}(\pi,\pi'),\underline{\underline{r}}')-G(\smallp,\underline{\underline{r}}')\big).
\end{equation}

Fix $\mathbf P_0\in\mathcal M_1(\mathbb K)$.  By construction, the
tree-valued Kingman coalescent solves the $(\mathbf
P_0,\Omega^\downarrow,{\mathcal G}^{1,0})$-martingale problem.  \sm

\subsection*{The duality relation}
\label{Sub:dualrel}
We are ready to state a duality relation between the tree-valued
Fleming-Viot dynamics and the tree-valued Kingman coalescent.

To introduce a class ${\mathcal H}$ of \emph{duality functions},
we identify every partition $\smallp\in\mathbb S$ with
the map $\smallp$ which sends $i\in\N$ to the block $\pi\in\smallp$
iff $i\in\pi$, and put for $\smallp\in\mathbb S$,
\begin{equation}\label{rp}
(\underline{\underline{r}})^\smallp
:=
\big(r_{\min\smallp(i),\min\smallp(j)}\big)_{1\leq i<j}.
\end{equation}
Let then for each $n\in\N$ and $\phi\in\mathcal
C_b^1(\mathbb R_+^{\N\choose 2})$ depending on the coordinates
$(r_{i,j})_{1\leq i<j\leq n}$ only, the function $H^{n,
  \phi}:\,\mathbb{U}\times\mathbb{K}\to\R$ be defined as
\begin{equation}\label{eq:treeL1g}
  \begin{aligned}
    H^{n,\phi}\big(\smallu,(\smallp,\underline{\underline{r}}')\big)
    &:= \int\nu^{\smallu}(\mathrm{d}\underline{\underline{r}})\, \phi
    \big((\underline{\underline{r}})^\smallp+\underline{\underline{r}}^\prime\big).
  \end{aligned}
\end{equation}

Notice that then the collection of functions
\begin{equation}
\label{grevz}
    \mathcal H = \big\{H^{n,\phi}(\boldsymbol{\cdot},\smallk):\,n\in\N, \smallk\in\mathbb K,
    \phi\in{\mathcal C}_b^1(\mathbb R^{\N\choose 2})\big\}
\end{equation}
is equal to $\Pi^1$, and thus separates points in ${\mathcal
  M}_1(\mathbb{U})$, see Remark~\ref{rem:05}.

\begin{proposition}[Duality relation]
  \label{P:dual}
  For $\mathbf P_0\in\mathcal M_1(\mathbb U)$ and $\smallk \in \mathbb
  K$, let ${\mathcal U}=(\mathcal U_t)_{t\geq 0}$ and ${\mathcal
    K}=(\mathcal K_t)_{t\geq 0}$ be solutions of the $(\mathbf P_0,
  \Omega^\uparrow,\Pi^1)$ and
  $(\delta_\smallk,\Omega^\downarrow,{\mathcal G}^{1,0})$-martingale
  problem, respectively. Then, if $\mathcal U$ and $\mathcal K$ are
  independent,
    \begin{equation}\label{eq:dualRel}
      \mathbf{E}\big[H(\mathcal U_t,\smallk)\big] =
      \mathbf{E}\big[H(\mathcal U_0,{\mathcal K}_t)\big],
  \end{equation} for all $t\geq 0$ and $H\in\mathcal H$.
\end{proposition}\sm

\begin{proof}
We shall establish that for $H^{n,\phi}\in\mathcal H$,
\begin{equation}\label{eq:treeL1i}
  \Omega^\uparrow H^{n,\phi}(\boldsymbol{\cdot},\smallk)(\smallu)
  =
  \Omega^\downarrow H^{n,\phi}\phi(\smallu,\boldsymbol{\cdot})(\smallk).
\end{equation}
Using the fact that $H^{n,\phi}$ is bounded the assertion then follows
from Theorem~4.4.11 (with $\alpha=\beta=0$) in \cite{EthierKurtz86}.

We verify (\ref{eq:treeL1i}) for the two components of the
dynamics separately. Observe first that by \eqref{eq:omega1} and
\eqref{eq:Gdown1},
\begin{equation}\label{eq:dual3}
  \begin{aligned}
    \Omega^{\uparrow,\mathrm{grow}}H^{n,\phi}(\boldsymbol{\cdot},
    (\smallp,\underline{\underline{r}}'))(\smallu) &=
    2\cdot\int\nu^{\smallu}(\mathrm{d}\underline{\underline{r}})
    \sum_{1\leq i<j} \frac{\partial}{\partial r_{i,j}}\phi
    \big((\underline{\underline{r}})^\smallp +
    \underline{\underline{r}}^\prime\big)
    \\
    &=
    \int\nu^{\smallu}(\mathrm{d}\underline{\underline{r}})\,\langle\nabla^{r^\prime}_{
      \smallp }\phi, \underline{\underline 2}\rangle
    \big((\underline{\underline{r}})^\smallp
    +\underline{\underline{r}}^\prime\big)
    \\
    &= \Omega^{\downarrow,\mathrm{grow}}H^{n,\phi}\big(\smallu,
    \boldsymbol{\cdot})(\smallp,\underline{\underline{r}}'),
\end{aligned}
\end{equation}
where we have used in the second equality that
$\frac{\partial}{\partial r_{i,j}}\phi
((\underline{\underline{r}})^\smallp+\underline{\underline{r}}^\prime)=0$,
whenever $i\sim_{\smallp} j$.

Similarly, using $\theta_{k,l}$ from \eqref{pp11b},
\begin{equation}\label{eq:dual4}
\begin{aligned}
  & \Omega^{\uparrow,\text{res}}H^{n,\phi}\big(\boldsymbol{\cdot},
  (\smallp,\underline{\underline{r}}')\big)(\smallu)
  \\
  &= \frac \gamma 2\int
  \nu^{\smallu}(\mathrm{d}\underline{\underline{r}})\, \sum_{1\le
    k,l}\big(\phi\big(\theta_{k,l}{(\underline{\underline{r}})^\smallp}
  + \underline{\underline{r}}'\big) -
  \phi\big((\underline{\underline{r}})^\smallp +
  \underline{\underline{r}}'\big)\big)
  \\
  &= \gamma\int\nu^{\smallu}(\mathrm{d}\underline{\underline{r}})\,
  \sum_{{\{\pi,\pi'\}\subseteq\smallp}\atop{\pi\not =\pi'}}\big(
  \phi\big((\underline{\underline{r}})^{\kappa_{\smallp}(\pi,\pi')} +
  \underline{\underline{r}}'\big) - \phi\big(
  (\underline{\underline{r}})^{\smallp} + \underline{\underline{r}}'\big) \big)\\
  &= \Omega^{\downarrow,\text{coal}}H^{n,\phi}\big(\smallu,
  \boldsymbol{\cdot})(\smallp,\underline{\underline{r}}').
\end{aligned}
\end{equation}

Combining (\ref{eq:dual3}) with (\ref{eq:dual4}) yields (\ref{eq:treeL1i}) and
thereby completes the proof.
\end{proof}\sm

\section[Martingale problems for Moran dynamics]{Martingale problems
  for tree-valued Moran dynamics}
\label{S:finite}
Fix $N\in\N$, and recall from Definition \ref{def:moran} the
tree-valued Moran dynamics ${\mathcal U}^N=({\mathcal U}^N_t)_{t\ge
  0}$ of population size $N$. In this section we characterize the
tree-valued Moran dynamics as unique solutions of a martingale problem
in Subsection \ref{Sub:mpN}. We then use an approximation argument to
establish the existence of the solution to the Fleming-Viot martingale
problem in Subsection~\ref{Sub:conv2}. Subsection~\ref{sec:coup}
establishes a coupling of tree-valued Moran models needed to establish
the Feller property of the tree-valued Fleming-Viot dynamics.

Notice that the states of the tree-valued Moran
dynamics with population size $N$ are restricted to
\begin{equation}\label{e:MoranN}
  \mathbb{U}_{N}
  :=
  \big\{\smallu = \overline{(U,r,\mu)}\in\mathbb{U}: \, N\mu\in\mathcal N(U)\big\}
  \subset
  \mathbb{U}_c,
\end{equation}
where $\mathcal N(U)$ is the set of integer-valued measures on $U$.
Moreover, if $\smallu\in\mathbb{U}_N$, then $\smallu$ can be
represented by the pseudo-metric measure space
\begin{equation}\label{moran2}
  \big(\{1,2,...,N\},r', N^{-1} \sum_{i=1}^N\delta_i\big),
\end{equation}
for some pseudo-metric $r'$ on $\{1,\ldots, N\}$. In the following we
refer to the elements $i\in\{1,2,...,N\}$ as the {\em individuals} of
the population of size $N$.

By construction, the tree-valued Moran dynamics are derived from the
following particle dynamics on the representative \eqref{moran2}:
\begin{itemize}
\item[{}] {\bf Resampling. } At rate $\tfrac \gamma 2>0$, a resampling
  event occurs between two individuals $k,l$ such that distances to
  $l$ are replaced by distances to $k$. This implies, in particular,
  that the genealogical distance between $k$ and $l$ is set to be
  zero. Equivalently, the measure changes from $\mu$ to
  $\mu+\frac{1}{N}\delta_k-\frac{1}{N}\delta_{l}$.
\item[{}] {\bf Distance growth. } The distance between any two
  different individuals $i,j$ grows at speed $2$.
\end{itemize}\sm

\subsection{The martingale problem for a fixed population size $N$}
\label{Sub:mpN}
In this subsection we characterize the resampling and distance growth
dynamics by a martingale problem.

Fix $N\in\N$. Similarly as in \eqref{mono1a}, for a metric space
$(U,r)$, define a map which sends a sequence of $N$ points to the
matrix of mutual distances
\begin{align}
  R^{N,(U,r)}: \begin{cases} U^N & \to \mathbb R^{\binom{N}{2}} \\
    (x_1,\ldots, x_N) & \mapsto (r(x_i, x_j))_{1\leq i<j\leq
      N}\end{cases}.
\end{align}

For a pseudo-metric measure space $(U,r,\mu)$ with $N\mu\in{\mathcal
  N}(U)$, let
\begin{equation}\label{e:without}
\begin{aligned}
   &\mu^{\otimes_{\downarrow} N}(\mathrm{d}(u_1,...,u_N))
     \\
      &:=
     \mu(\mathrm{d}u_1)\otimes \frac{\mu-\frac{1}{N}\delta_{u_1}}
        {1-\frac{1}{N}}(\mathrm{d}u_2)\otimes
  ...\otimes\frac{\mu-\frac{1}{N}
    \sum_{k=1}^{N-1}\delta_{u_{k}}}{1-\frac{(N-1)}{N}}(\mathrm{d}u_{N}),
\end{aligned}
\end{equation}
 the
{\em sampling (without replacement) measure} and define the  {\em $N$
 distance matrix
  distribution  (without replacement)}
${\nu}^{N,(U,r,\mu)}$ of $\smallu=\overline{(U,r,\mu)}\in\mathbb{U}_N$
 by
\begin{equation}\label{mono1ba}
  {\nu}^{N,\smallu}
  :=
  \big(R^{N,(U,r)}\big)_\ast\mu^{\otimes_{\downarrow}N}\in{\mathcal
 M}_1(\mathbb R_+^{N\choose 2}).
\end{equation}
Observe that $\smallu\in\mathbb{U}_N$ is uniquely
characterized by its $N$ distance matrix distribution. \sm

Once more, it is obvious that ${\nu}^{N,(U,r,\mu)}$ depends on
$(U,r,\mu)$ only through its equivalence class
$\overline{(U,r,\mu)}\in\mathbb{U}_N$ leading to the following
definition.

\begin{definition}[$N$-distance matrix distribution] For $N\in\mathbb
  N$, the $N$ distance matrix distribution $\nu^{N,\smallu}$ (without
  replacement) of $\smallu\in\mathbb{U}_N$ is defined as the $N$
  distance matrix distribution ${\nu}^{N,(U,r,\mu)}$ of an arbitrary
  representative $(U,r,\mu)$ of the equivalence class
  $\smallu=\overline{(U,r,\mu)}$.
\label{Def:18}
\end{definition}\sm

For a measurable, bounded $\phi: \R_+^{N\choose 2}\to\mathbb R$,
introduce the polynomial $\Phi^{\phi}_N$ by
  \begin{equation}\label{eq:phifinite}
    \Phi^{\phi}_N(\smallu)
    =
    \langle \nu^{N,\smallu},\phi\rangle
    :=
    \int_{\R_+^{N\choose 2}} \nu^{N,\smallu}(\mathrm d\underline{\underline{r}})\,\phi\big(\underline{\underline r}\big)
\end{equation}
and set
\begin{equation}\label{Pi1NN}
  \Pi_N
  :=
  \big\{\Phi_N^{\phi}:\,\phi\in \mathcal B(\R_+^{N\choose 2})\big\},
\end{equation}
and
\begin{equation}\label{Pi1N}
  \Pi^1_N
  :=  \big\{\Phi_N^{\phi}:\,\phi\in \mathcal C^1_b(\R_+^{N\choose 2})\big\}.
\end{equation}
In contrast to $\Pi^1$, the space $\Pi_N^1$ does not form an
  algebra. However, we only require that $\Pi_N^1$ is separating on
  $\mathbb U_N$, which can easily be shown.

We define an operator $\Omega^{\uparrow,N}:=
\Omega^{\uparrow,\mathrm{grow},N}+\Omega^{\uparrow,\text{res},N}$ on
$\Pi_N$ with domain $\Pi^1_N$ by independent superposition of {\em
  resampling} and {\em distance growth}.

We begin with the {\em distance growth} operator
$\Omega^{\uparrow,\textrm{grow},N}$. Since distances of any pair of
distinct points grow at speed 2 in periods without resampling, we put
\begin{equation}\label{darst1}
\begin{aligned}
  \Omega^{\uparrow,\mathrm{grow},N}\Phi_N^{\phi} %&:=
  &:= \big\langle\nu^{N,\smallu},\langle\nabla\phi,
  \underline{\underline 2}\rangle\big\rangle,
\end{aligned}
\end{equation}
with $\langle\nabla\phi, \underline{\underline 2}\rangle$ from
\eqref{e:divergenz}.

For the \emph{resampling operator} $\Omega^{\uparrow,\textrm{res},N}$,
consider first the action on a representative $(U,r,\mu)$ of the form
(\ref{moran2}). Any resampling event in which the individual $l$ is
replaced by a copy of the individual $k$ changes the measure from
$\mu$ to $\mu + \tfrac1N \delta_k - \tfrac 1N \delta_l$.

Therefore, since
\begin{align}
  \sum_{1\leq k,l\leq N} (R^{N,(U,r)})_\ast \big( \mu+ \tfrac 1N
  \delta_k - \tfrac 1N \delta_l\big)^{\otimes_\downarrow N} =
  \sum_{1\leq k,l\leq N} (\theta_{k,l})_\ast \nu^{N,\smallu}
\end{align}
we obtain for $\smallu=\overline{(U,r,\mu)}$ that
\begin{equation}
  \begin{aligned}
    & \Omega^{\uparrow,\text{res},N} \Phi_N^\phi(\smallu) \\& =
    \frac{\gamma}{2} \sum_{1\leq k,l\leq N} \Big( \langle
    (R^{N,(U,r)})_\ast \big( \mu+ \tfrac 1N \delta_k - \tfrac 1N
    \delta_l\big)^{\otimes_\downarrow N}, \phi\rangle - \langle
    (R^{N,(U,r)})_\ast
    \mu^{\otimes_\downarrow N}, \phi\rangle\Big) \\
    & = \frac{\gamma}{2} \sum_{1\leq k,l\leq N} \big(\langle
    \nu^{N,\smallu}, \phi\circ\theta_{k,l} \rangle - \langle
    \nu^{N,\smallu}, \phi\rangle\big).
  \end{aligned}
\end{equation}
It is easy to see that for given $N\in\N$, ${\Pi}_N^1$ is separating
in $\mathbb{U}_N$. We can therefore use the operator
$(\Omega^{\uparrow,N},\Pi^1_N)$ to characterize the tree-valued Moran
models analytically.

\begin{proposition}[Tree-valued Moran dynamics]
    For all $N\in\N$ and $\mathbf P_0^N\in \mathcal
  M_1(\mathbb{U}_N)$,
  \label{P:06} the $(\mathbf
  P_0^N,\Omega^{\uparrow,N},\Pi_N^1)$-martingale problem is
  well-posed.
\end{proposition}

\begin{proof}
  Let $(\mathcal I^N, \preceq^N)$ in Definition \ref{def:finite} be
  such that the law of $\mathcal U_0^N$ equals $\mathbf P_0^N$. Then
  the tree-valued Moran dynamics given by Definition \ref{def:moran}
  solve the $(\mathbf P_0^N,\Omega^{\uparrow,N},\Pi_N^1)$-martingale
  problem, by construction. This proves {\em existence}.

  For {\em uniqueness} -- following the same line of argument as given
  in Section~\ref{S:unique} -- one can check that the $(\mathbf
  P_0^N,\Omega^{\uparrow,N},\Pi^1_N)$-martingale problem is dual to
  the tree-valued Kingman coalescent where the duality functions
  $\Phi\in\Pi_N^1$ are smooth polynomials that involve sampling
  without replacement (see, for example, Corollary~3.7 in~\cite{GLW05}
  where a similar duality is proved on the level of the measure-valued
  processes).
\end{proof}\sm

\subsection{Convergence to the Fleming-Viot generator}
\label{Sub:conv2}
The goal of this subsection is to show that the operator for the
tree-valued Fleming-Viot martingale problem is the limit of the
operators for the tree-valued Moran martingale problems. This is one
ingredient for the proof of Theorem~\ref{T:02} given in
Section~\ref{S:proofs}.

\begin{proposition} Let $\Phi\in\Pi^1$. There exist
  $\Phi_1\in\Pi_1^1, \Phi_2\in\Pi_2^1,\ldots$ such that
  \label{P:01}
  \begin{equation}\label{Pi1c}
    \lim_{N\to\infty}\sup_{\smallu\in\mathbb{U}_N}\big| \Phi_N(\smallu)-\Phi(\smallu)\big| = 0,
  \end{equation}
  and
  \begin{equation}\label{4.5.2a}
    \lim_{N\to\infty}\sup_{\smallu\in\mathbb{U}_N}\big|\Omega^{\uparrow,N}\Phi_N(\smallu)-\Omega^{\uparrow}\Phi(\smallu)\big|
    = 0.
  \end{equation}
\end{proposition}\sm

\begin{proof}
  First, define the extension operator
  \begin{align}
    \iota_N: \begin{cases}\mathbb
      R^{\binom{N}{2}} & \mapsto \mathbb R^{\binom{\mathbb N}{2}} \\
      (r_{i,j})_{1\leq i<j\leq N} & \mapsto (r_{i_{\simeq N}\wedge
        j_{\simeq N}},r_{i_{\simeq N}\vee j_{\simeq N}})_{1\le
        i<j} ,\end{cases}
  \end{align}
  where $i_{\simeq N}:=1+((i-1)\mbox{ mod }N)$.  Fix
  $\Phi=\Phi^{n,\phi}\in\Pi^1$ for $n\in\mathbb N$, $\phi\in{\mathcal
    C}_b^1(\mathbb R_+^{\binom{\mathbb N}{2}})$. For $N\geq n$ set
  $\Phi_N := \Phi_N^{\phi\circ\iota_N} \in \Pi_N^1$. By the definition
  of the $N$-distance matrix distribution of a representative
  \eqref{mono1ba}, there is a $C>0$ such that
  \begin{equation}
    \label{eq:convphi}
    \begin{aligned}
      \sup_{\smallu\in\mathbb U_N}\big|\Phi_N(\smallu) -
      \Phi(\smallu)\big| & = \sup_{\smallu\in\mathbb U_N}\big| \langle
      \nu^{N,\smallu}, \phi\circ\iota_N
      \rangle - \langle \nu^\smallu, \phi\rangle\big| \\
      & = \sup_{\smallu\in\mathbb U_N}\big| \langle (\iota_N)_\ast
      \nu^{N,\smallu} - \nu^\smallu, \phi\rangle\big| \\ & \leq
      \frac{C}{N} ||\phi||
    \end{aligned}
  \end{equation}
  for all $N\geq n$. This shows \eqref{Pi1c}. For \eqref{4.5.2a}
  observe that $\Omega^{\uparrow} \Phi(\smallu) = \langle
  \nu^{\smallu}, \psi\rangle$ and $\Omega^{\uparrow,N}
  \Phi_N(\smallu)= \langle \nu^{N,\smallu}, \widetilde\psi\rangle$ for
  continuous, bounded functions $\psi$ and $\widetilde\psi$ satisfying
  $\widetilde\psi = \psi\circ\iota_N$. Hence, \eqref{4.5.2a} follows from
  \eqref{eq:convphi}.
\end{proof}\sm

\subsection{Coupling tree-valued Moran dynamics}
\label{sec:coup}
In this section we show how to couple two tree-valued Moran
dynamics. In particular, using a metric on ultra-metric measure spaces
introduced in \cite{GrePfaWin2006a}, we show that the coupled
processes become closer as time evolves (Proposition
\ref{P:tbc}). This will be an important ingredient in showing the
Feller property of the tree-valued Fleming-Viot dynamics stated in
Theorem~\ref{T:01}.

We fix $N\in\N$ and $\mathcal I^N := \{1,...,N\}$. Informally, we
couple two tree-valued Moran dynamics by using the same resampling
events. For this, recall the Poisson processes $\eta=\{\eta^{i,j};
i,j\in\mathcal I^N\}$ from Definition~\ref{def:moran} which determine
resampling events. Recall from Definition \ref{def:finite} the notion
of ancestors $A_s(i,t)$, $i\in\mathcal I^N$ and $s\leq t$.

In order to be in a position to compare coupled Moran models, we use
the following metric on $\mathbb U$ introduced in \cite[Section
10]{GrePfaWin2006a}.

\begin{definition}[Modified Eurandom metric]
  The modified Eurandom distance between $\smallu_1 =
  \overline{(U_1,r_1,\mu_1)}$ and
  $\smallu_2=\overline{(U_2,r_2,\mu_2)}\in\mathbb U$ is given by
  \label{def:eur}
  \begin{equation}
  \begin{aligned}
    \label{eq:pqe7}
    &d'_{\text{Eur}}(\smallu_1, \smallu_2)
   \\
  &:= \inf_{\widetilde\mu} \int_{U_1^2}\int_{U_2^2}
    \widetilde\mu(d(i_1, i_2)) \widetilde\mu(d(j_1, j_2))\,\big|r_1(i_1,
    j_1) - r_2(i_2, j_2)\big|\wedge 1
  \end{aligned}
  \end{equation}
  where the infimum is taken over all couplings of $\mu_1$ and $\mu_2$,
  i.e.,
  \begin{align}
    \widetilde\mu \in \big\{\widetilde{\mu}'\in\mathcal M_1(U_1\times U_2):\, (\pi_k)_\ast
    \widetilde\mu' = \mu_k, k=1,2\big\},
  \end{align}
  with $\pi_k: U_1\times U_2\to U_k$ denoting the projection on the
  $k^{\text{th}}$ coordinate, $k=1,2$.
\end{definition}

\begin{remark}[Connection to the Gromov-weak topology]\label{rem:eur}
    By Proposition~10.5 in~\cite{GrePfaWin2006a}, the distance
  $d_{\text{Eur}}'$ is indeed a metric and generates the Gromov-weak
  topology but is not complete. In particular, for $\mathbb U$-valued
  random variables $\mathcal U$, $\mathcal U_1$, $\mathcal U_2,...$
  which are all defined on the same probability space, we find that
  $\mathcal U_n \Rightarrow \mathcal U$, as $n\to\infty$, iff
  $\mathbf{E}[d_{\text{Eur}}'(\mathcal U_n, \mathcal U)]\to 0$, as
  $n\to\infty$.  \hfill$\qed$
  \label{Rem:10}
\end{remark}\sm

\begin{remark}[Modified Eurandom metric on $\mathbb
  U_N$]
Recall $\mathbb U_N$ from~\eqref{e:MoranN}, and
    let $\smallu_k=\overline{(\mathcal I^N, \widetilde r_k, \mu_k)}$,
    $k=1,2$, be two $\mathbb U_N$-valued random variables. Since
    $\mu_k$ has atoms of size $1/N$, $k=1,2$, the modified Eurandom
  metric is given by\label{rem:eur2}
  \begin{align}\label{eq:eur1}
    d'_{\text{Eur}}(\smallu_1, \smallu_2) =
    \inf_{\sigma\in\Sigma_{\mathcal I^N}}
    \frac{1}{N^2}\sum_{i,j\in\mathcal I^N} |\widetilde r_1(i,j) -
    \widetilde r_2(\sigma(i), \sigma(j))|\wedge 1,
  \end{align}
  where $\Sigma_{\mathcal I^N}$ is the set of permutations of $\mathcal
  I^N$. Moreover, there exist $(\mathcal I^N, r_k, \mu_k)\in\smallu_k,
  k=1,2$ such that
  \begin{align}\label{eq:eur2}
    d'_{\text{Eur}}(\smallu_1, \smallu_2) =
    \frac{1}{N^2}\sum_{i,j\in\mathcal I^N} |r_1(i,j) - r_2(i, j)|\wedge
    1.
  \end{align}
  \hfill$\qed$\label{rem:eur2}
\end{remark}\sm

\begin{definition}[Coupled tree-valued Moran dynamics]
  \label{def:cMd}
  For $\mathcal I = (\mathcal I_t)_{t\geq 0}$ and $\mathcal
  I_t := \mathcal I^N := \{1,...,N\}$, let $\preceq^1_0, \preceq^2_0$
  be two partial orders on $(-\infty, 0]\times\mathcal I^N$, both
  satisfying (iii) in Definition \ref{def:finite}. Moreover, let
  $\eta$ be a realization of the Poisson processes given in
  Definition~\ref{def:moran}, defining the processes
  $\preceq^1:=(\preceq_t^1~)_{t\geq 0}$ and
  $\preceq^2:=(\preceq^2_t)_{t\geq 0}$ as in
  Definition~\ref{def:moran}. Then, for $(\mathcal U_t^{N,k})_{t\geq
    0}$, read off from $(\mathcal I, \preceq^{k})$, $k=1,2$, the
  process $(\mathcal U_t^{N,1}, \mathcal U_t^{N,2})_{t\geq 0}$ is
  referred to as \emph{the coupled tree-valued Moran dynamics} started
  in $(\mathcal U_0^{N,1}, \mathcal U_0^{N,2})$.
\end{definition}\sm

\begin{proposition}[Contraction of coupled tree-valued
  Moran dynamics]
  Let $(\mathcal U_t^{N,1}, \mathcal U_t^{N,2})_{t\geq 0}$
  be
  \label{P:tbc}
  the coupled tree-valued Moran dynamics started in
  $(\mathcal U^{N,1}, \mathcal U^{N,2})$. Then for all $t>0$,
      \begin{align}
    \label{eq:pqe5}
    \mathbf E[d_{\text{Eur}}' \big(\mathcal U_t^{N,1}, \mathcal
    U_t^{N,2}\big)] = e^{-\gamma t} d_{\text{Eur}}' \big(\mathcal
    U_0^{N,1}, \mathcal U_0^{N,2}\big).
  \end{align}
\end{proposition}\sm

\begin{proof}
    Recall from Definition~\ref{def:finite} that $A_s(i,t)$ is the
  ancestor of $(i,t)$ by time $s$ and from~\eqref{eq:854} that $r_t^1$,
  $r_t^2$ are the metrics given by the coupled Moran dynamics by time
  $t\geq 0$.

  By the definition of the coupled tree-valued Moran dynamics, for
  $i,j\in\mathcal I_N$,
  \begin{align}
    \label{eq:pqe3}
    \big|r_t^{1}(i,j) - r_t^{2}(i,j)\big| =
    \big|r_0^{1}(A_0(i,t),A_0(j,t)) - r_0^{2}(A_0(i,t),A_0(j,t))\big|.
  \end{align}

   Let $I,J$ be independent, uniformly distributed on
    $\mathcal I_N$ and independent of all other random
    variables. Given $I\neq J$, we distinguish two cases: (i) $s\in
  \eta^{A_s(I,t), A_s(J,t)} \cup \eta^{A_s(J,t), A_s(I,t)}$ for some
  $0\leq s\leq t$. Here, the ancestral lines of $I$ and $J$ were
  affected by a joint resampling event, resulting in $A_0(I,t) =
  A_0(J,t)$. This event happens with probability $1-e^{-\gamma
    t}$. (ii)~In the other case, occurring with probability
  $e^{-\gamma t}$, we find that $A_0(I,t)$ and $A_0(J,t)$ are again
  distributed as $I$ and $J$. Hence, for all $t\ge 0$, by
  \eqref{eq:eur2},
  \begin{equation}
    \label{eq:pqe1}
    \begin{aligned}
      \mathbf E[d_{\text{Eur}}' (\mathcal U_t^{N,1},  \mathcal
      U_t^{N,2})] &= \mathbf E[|r_t^{1}(I,J) - r_t^{2}(I,J)|\wedge 1]
      \\ & = \mathbf E[|r_0^{1}(A_0(I,t),A_0(J,t)) -
      r_0^{2}(A_0(I,t),A_0(J,t))| \wedge 1] \\ & = e^{-\gamma t}
      \mathbf E[|r_0^{1}(I,J) - r_0^{2}(I,J)| \wedge 1] \\ & =
      e^{-\gamma t} d_{\text{Eur}}' ( \mathcal U^{N,1}, \mathcal
      U^{N,2}),
    \end{aligned}
  \end{equation}
  as claimed.
\end{proof}\sm

\section{Limit points are compact}
\label{sub:51}
Recall from Definition~\ref{def:moran} the tree-valued Moran dynamics
${\mathcal U}^{N}$ with population size $N\in\N$. In this section we
show that potential limit points of the sequence $\{{\mathcal
  U}^N;\,N\in\N\}$ have c\`adl\`ag sample paths in $\mathbb U$ and
take values in the space $\mathbb U_c$ of {\em compact} ultra-metric
measure spaces for $t>0$. In Subsection~\ref{S:compact} we state a
sufficient condition for relative compactness in $\mathbb M_c$ and
give in Subsection \ref{sub:cpopu} a criterion for a sequence of
population models to satisfy the compact containment condition. In
Subsection~\ref{sub:cpapp} we apply this criterion to show that the
sequence of tree-valued Moran dynamics $\mathcal U^N$ satisfies the
compact containment condition.

\subsection{Relative compactness in $\mathbb M_c$}
\label{S:compact}
We give a criterion for a set to be {\em relatively compact} in
$\mathbb M_c$.  In this subsection we are dealing with general (not
necessarily ultra-) metric measure spaces. We define for
$\smallx\in\mathbb M$ the \emph{distance distribution}
$w_{\smallx}\in{\mathcal M}_1(\R_+)$ by
\begin{equation}\label{eq:dist}
  w_{\smallx}(A)
  := \nu^{\smallx}\big\{\underline{\underline{r}}\in\R_+^{\N\choose 2}:\,r_{1,2}\in A\big\},
\end{equation}
for all $A\in{\mathcal B}(\mathbb R_+)$. Recall from \cite[Proposition
7.1]{GrePfaWin2006a} the following characterization of relative compactness.

\begin{proposition}[Characterization of relative compactness in
  $\mathbb M$]\mbox{}\\
  \label{noteP:05}
  A set $\Gamma\subseteq \mathbb M$ is relatively compact in the
  Gromov-weak topology iff the following two conditions hold:
  \begin{itemize}
  \item[$(i)$] $\{w_{\smallx}: \smallx\in\Gamma\}$ is tight in
    $\mathcal M_1(\mathbb R_+)$.
  \item[$(ii)$] For all $\varepsilon>0$ there exists $C_\varepsilon>0$
    such that $\sup_{\smallx\in \Gamma} \widetilde
    S_\varepsilon(\smallx) \leq C_\varepsilon,$
  \end{itemize}
  where for $(X,r,\mu)\in\smallx\in\mathbb M$
  \begin{align}
    \widetilde S_\varepsilon(\smallx) := \min\Big\{K: \exists
    x_1,...,x_K \in X: \mu \Big( \bigcup_{k=1}^K
    B_{\varepsilon}(x_i)\Big)>1-\varepsilon\Big\}.
  \end{align}
\end{proposition}

\noindent
The relative compactness criterion in $\mathbb M_c$ reads as follows:

\begin{proposition}[Criterion for relative compactness in
  $\mathbb{M}_c$]
  A set $\Gamma\subseteq\mathbb M_c$ is relatively compact in the
  \label{noteP:04}
  Gromov-weak topology on $\mathbb M_c$ if the following two
  conditions are satisfied.
  \begin{itemize}
  \item[$(i)$] $\{w_{\smallx}: \smallx\in\Gamma\}$ is tight in
    $\mathcal M_1(\mathbb R_+)$.
  \item[$(ii)$] For all $\varepsilon>0$ there exists
    $N_\varepsilon\in\mathbb N$ such that $\sup_{\smallx\in
      \Gamma}S_\varepsilon(\smallu)\leq N_\varepsilon,$ where
    $S_\varepsilon$ is the minimal number of open $\varepsilon$-balls
    needed to cover $(\text{supp}(\mu),r)$ for $(X,r,\mu)\in\smallx
    \in \mathbb M$.
  \end{itemize}
\end{proposition}\sm

\begin{remark}[Relative compactness criterion is only sufficient]
  By Proposition \ref{noteP:05}, (i) is a necessary condition
  \label{Rem:09b}
  for relative compactness in $\mathbb M$. Note that (ii) is not
  necessary for relative compactness in $\mathbb M_c$: Consider, for
  example,
\begin{align}\label{eq:exCp}
  \Gamma = \{\smallx_n = \overline{(\{0,1,...,n\}, r_{\mathrm{eucl}},
    \mathrm{Bin}(n, \tfrac{1}{n^2}))}: n\in\N\}\subset \mathbb M_c.
\end{align}
Since $\smallx_n \to \overline{(\N, r_{\mathrm{eucl}}, \delta_0)}$, as
$n\to\infty$, the set $\Gamma$ is relatively compact, but (ii) does
not hold. \hfill $\qed$
\end{remark}
\sm

The proof of Proposition \ref{noteP:04} is based on two Lemmata. Recall
that for a metric space $(X,r)$ an {\em $\varepsilon$-separated set}
is a subset $X'\subseteq X$ such that $r(x',y')>\varepsilon$, for all
$x',y'\in X'$ with $x'\not =y'$.

\begin{lemma}[Relation between $\varepsilon$-balls and
  $\varepsilon$-separated nets]
\label{L:05}
Fix $N\in\N$, a metric space $(X,r)$ with $\# X\ge N+1$
and $\varepsilon>0$. The following hold.
\begin{itemize}
\item[(i)] If $(X,r)$ can be covered by $N$ open balls of radius
  $\varepsilon$, then $(X,r)$ has no $2\varepsilon$-separated sets of
  cardinality $k\geq N+1$.
\item[(ii)] If $(X,r)$ has no $\varepsilon$-separated set of
  cardinality $N+1$, then $(X,r)$ can be covered by $N$ closed balls
  of radius $\varepsilon$.
\end{itemize}
\end{lemma}\sm

\begin{proof} (i) Assume that $x_1,...,x_N\in X$ are such that
  $X=\bigcup_{i=1}^N B_\varepsilon(x_i)$, where we denote by
  $B_\varepsilon(x)$ the open ball around $x\in X$ of radius
  $\varepsilon>0$.  Choose $(N+1)$ distinct points $y_1,...,y_{N+1}\in
  X$. By the pigeonhole principle, two of the points must fall into
  the same ball $B_\varepsilon(x_i)$, for some $i=1,...,N$, and are
  therefore in distance smaller than $2\varepsilon$. Hence
  $\{y_1,...,y_{N+1}\}$ is not $2\varepsilon$-separated. Since
  $y_1,...,y_{N+1}\in X$ were chosen arbitrarily, the claim follows.

  (ii) Again, we proceed by contradiction.  Let $K$ be the maximal
  possible cardinality of an $\varepsilon$-separated set in
  $(X,r)$. By assumption, $K\le N$. Assume that
  $S^K_\varepsilon:=\{x_1,...,x_K\}$ is an $\varepsilon$-separated set
  in $(X,r)$. We claim that $X=\bigcup_{i=1}^K
  \overline{B}_{\varepsilon}(x_i)$ with
  $\overline{B}_\varepsilon(x):=\{x'\in
  X:\,r(x,x')\le\varepsilon\}$. Indeed, assume, to the contrary, that
  $y\in X$ is such that $r(y,x_i)>\varepsilon$, for all $i=1,...,K$,
  then $S^K_\varepsilon\cup\{y\}$ is an $\varepsilon$-separated set of
  cardinality $K+1$, which gives the contradiction.
\end{proof}\sm

\begin{lemma}[Bounds on the number of balls to cover a limit point]
\label{l:app2}
Fix $\varepsilon>0$ and $N\in\N$. Let $\smallx=\overline{(X,r,\mu)}$,
$\smallx_1=\overline{(X_1,r_1,\mu_1)}$,
$\smallx_2=\overline{(X_2,r_2,\mu_2)}$, ... be elements of $\mathbb
M$ such that $\smallx_n\to\smallx$ in the Gromov-weak topology, as
$n\to\infty$. If $(\mathrm{supp}(\mu_1),r_1)$,
$(\mathrm{supp}(\mu_2),r_2)$, ... can be covered by $N$ open balls of
radius $\varepsilon$ then $(\mathrm{supp}(\mu),r)$ can be covered by
$N$ closed balls of radius $2\varepsilon$.
\end{lemma}\sm

\begin{proof}
    Define the restriction operator $\rho_N\big((r_{i,j})_{1\leq
    i<j}\big) := (r_{i,j})_{1\leq i<j\leq N}$. By Lemma~\ref{L:05}(i),
  there is no $n\in\N$ for which $(\mathrm{supp}(\mu_n),r_n)$ has a
  $2\varepsilon$-separated set of cardinality $N+1$. Set
  $B_{2\varepsilon}:=(2\varepsilon,\infty)^{{N+1\choose 2}}$.  Notice
  that $\rho_{N+1}^{-1}(B_{2\varepsilon})$ is open. Moreover,
  $(\mathrm{supp}(\mu),r)$ has a $2\varepsilon$-separated set of
  cardinality $N+1$ if and only if $(\rho_{N+1})_\ast\nu^\smallx
  (B_{2\varepsilon})>0$. However,
  \begin{equation}\label{eq:pre1}
    \begin{aligned}
      0 \leq (\rho_{N+1})_\ast\nu^{\smallx}(B_{2\varepsilon}) \leq
      \liminf_{n\to\infty}
      (\rho_{N+1})_\ast\nu^{\smallx_n}(B_{2\varepsilon}) = 0
    \end{aligned}
  \end{equation}
  by Theorem~5(b) in~\cite{GrePfaWin2006a} together with the
  Portmanteau theorem, therefore
  $(\rho_{N+1})_\ast\nu^\smallx(B_{2\varepsilon})=0$. By
  Lemma~\ref{L:05}(ii), $(\mathrm{supp}(\mu),r)$ can therefore be
  covered by $N$ closed balls of radius $2\varepsilon$.
\end{proof}\sm

\begin{proof}[Proof of Proposition~\ref{noteP:04}]
  Assume (i) and (ii) hold for a set $\Gamma\subseteq \mathbb
  M_c$. First note that by Theorem 2 in \cite{GrePfaWin2006a} the set
  $\Gamma$ is relatively compact in $\mathbb M$. It remains to show
  that every limit point of $\Gamma$ is compact.  To see this take
  $\smallx\in\mathbb M$ and $\smallx_1,\smallx_2,\ldots \in \Gamma$
  such that $\smallx_n\rightarrow\smallx$ in the Gromov-weak topology,
  as $n\to\infty$, and let $\varepsilon>0$.  By Assumption(ii)
  together with Lemma~\ref{l:app2}, $(\mathrm{supp}(\mu),r)$ can be
  covered by $N_{\varepsilon/2}$ closed balls of radius
  $\varepsilon$. Therefore, $\smallx$ is totally bounded which implies
  $\smallx\in\mathbb M_c$, and we are done.
\end{proof}\sm

\subsection{Compact containment (Proof of Proposition~\ref{PP:02})}
\label{sub:cpopu}
Recall that Proposition~\ref{PP:02} is based on
  the general notion of a finite population model; see
  Definition~\ref{def:finite}. In particular $\preceq =
  (\preceq_t)_{t\geq 0}$ is the process of partial orderings connected
  to genealogical relationships, $\mathcal I^N = (\mathcal
  I^N_t)_{t\geq 0}$ is the process of population sizes and $\tau$ is
  the lifetime of the population. Moreover, let for each $N\in\N$,
$(\mathcal U_t^N)_{t\in[0,\tau^N)}$ be the tree-valued population
dynamics read off from $({\mathcal I}^N,\preceq^N)$.

As a preparation we show two auxiliary lemmata which discuss the
consequences of the assumptions made in Proposition~\ref{PP:02}.
Recall the distance distribution $w_{\smallu}$ from (\ref{eq:dist}).

\begin{lemma}[Bounds on the distance distribution under
  Assumption~(i)]
  Fix $T>0$, and assume that $\{\mathcal U_0^N: N\in\N\}$ is tight in
  $\mathbb U$. If condition (i) of Proposition~\ref{PP:02} holds,
  then for all $\varepsilon>0$ there is a $C_\varepsilon>0$ such that
  \label{l1}
  \begin{equation}
    \label{eq:l11}
    \limsup_{N\to\infty}\mathbf P^N\big\{\sup_{t\in[0,T\wedge \tau^N)} w_{\mathcal
      U_t^N}([C_\varepsilon,\infty)) > \varepsilon\big\} \leq \varepsilon.
  \end{equation}
\end{lemma}\sm

\begin{proof} Let $\varepsilon>0$. Choose $\delta=\delta(\tfrac
  {\varepsilon}{4})>0$ such that (\ref{P:171}) holds for and
  $N\in\mathbb N$ and $\varepsilon/4$ (instead of $\varepsilon$) and
  any $\mathcal J^N$ such that $\mathcal J^N\subseteq \mathcal I_0^N$
  is ${\mathcal A}^N_{0}$-measurable with $\mu^N(\mathcal J^N)\leq
  \delta$. Since $\{\mathcal U_0^N:\,N\in\N\}$ is tight in $\mathbb
  U$, we can find such ${\mathcal A}_0^N$-measurable $\mathcal J^N
  \subseteq \mathcal I_0^N$, $N\in\N$, and a constant
  $\widetilde{C}_\varepsilon>0$ such that $\mu_0^N(\mathcal J^N)\leq
  \delta$, almost surely, and
  \begin{equation}
    \label{ww}
      \liminf_{N\to\infty}\mathbf{P}^N\big\{\mathcal I^N_0\setminus
        \mathcal J^N \mbox{ has diameter at most
      }\widetilde{C}_\varepsilon\big\}>1-\tfrac \varepsilon 2
  \end{equation}
  (see (i) in Proposition \ref{noteP:05}). Clearly, on the event that
  $\mathcal I_0^N\setminus \mathcal J^N$ has diameter at most
  $\widetilde{C}_\varepsilon$, the set $D_t(0,\mathcal I_0^n\setminus
  \mathcal J^N)$ of descendants of $\mathcal I_0^n\setminus \mathcal
  J^N$ at time $t$ has diameter at most
  $\widetilde{C}_\varepsilon+2t$. Hence
  \begin{equation}
    \label{ww1}
    \begin{aligned}
      \limsup_{N\to\infty} &
      \mathbf{P}^N\big\{\sup_{t\in[0,T\wedge\tau^N)} w_{{\mathcal
          U}^N_t}([\widetilde{C}_\varepsilon+2T,\infty))>\varepsilon\big\}
      \\ &\le \limsup_{N\to\infty} \mathbf P^N\big\{ \sup_{t\in [0,
        T\wedge\tau^N)} \mu_t^N(D_t(0,\mathcal J^N))>\tfrac
      \varepsilon{4}\big\} \\ & \qquad + \mathbf{P}^N\big\{\mathcal
      I_0^N\setminus \mathcal J^N \text{ has diameter at least
      }\widetilde C_\varepsilon\big\} \\ & \leq \varepsilon
    \end{aligned}
  \end{equation}
  by Assumption (i). The claim follows.
\end{proof}\sm

For the next lemma recall that for all $\varepsilon>0$ and $\mathcal
U_t = \smallu=\overline{(U,r,\mu)}\in\mathbb{U}$,
$S_{2\varepsilon}(\mathcal U_t)$ from \eqref{eq:861} and
$\widetilde{S}_{2\varepsilon}(\mathcal U_t)$ from \eqref{eq:861b}
denote the minimal numbers of $2\varepsilon$-balls needed to cover
$\mathrm{supp}(\mu)$ or to cover $\mathrm{supp}(\mu)\setminus V$ where
the exceptional set $V\subseteq U$ satisfies $\mu(V)\le
2\varepsilon$. In particular, these definitions coincide with the same
notions introduced in Propositions~\ref{noteP:05} and~\ref{noteP:04}.

\begin{lemma}[Uniform bounds on $S_{2\varepsilon}$ and $\widetilde
  S_{2\varepsilon}$]
  Fix $T>0$.
  \label{l2}
  \begin{itemize}
  \item [(a)] Assume Condition~(ii.i) from
    Proposition~\ref{PP:02}. Then for all $\varepsilon>0$ we can find
    $C_\varepsilon>0$ such that
    \begin{equation}\begin{aligned}\label{eq:l21}
        \limsup_{N\to\infty} \mathbf
        P^N\big\{\sup_{t\in[\varepsilon,T)} S_{2\varepsilon}(\mathcal
        U_t^N)>C_\varepsilon\big\} \leq 2\varepsilon.
      \end{aligned}
    \end{equation}
  \item [(b)] Assume that the family $\{\mathcal U_0^N;\,N\in\N\}$ is
    tight in $\mathbb U$ and Conditions~(i) and~(ii.ii) from
    Proposition~\ref{PP:02}.  Then for all $\varepsilon>0$ we can find
    $C_\varepsilon>0$ such that
    \begin{equation}\begin{aligned}\label{eq:l21b}
        \limsup_{N\to\infty} \mathbf P^N\big\{\sup_{t\in[0,T)}
        \widetilde{S}_{2\varepsilon}(\mathcal
        U_t^N)>C_\varepsilon\big\} \leq 2\varepsilon.
      \end{aligned}
    \end{equation}
  \end{itemize}
\end{lemma}\sm

\begin{proof} (a) The proof relies heavily on the fact that for all
  $t$, $t'$, $\varepsilon$, and $\varepsilon'$ such that
  $[t-\varepsilon,t]\subseteq [t'- \varepsilon', t']$
  \begin{equation}\label{eq:15}
    S_{2\varepsilon}(\mathcal U_t^N) \geq S_{2\varepsilon'}(\mathcal U_{t'}^N).
  \end{equation}

  Fix $\varepsilon>0$. Without loss of generality, we assume that
  $T=k\varepsilon$ for some $k\in\mathbb N$. Since for all
  $t\in[\varepsilon,T)$, the family $\{S_{2\varepsilon}(\mathcal
  U_t^N);\,N\in\N\}$ is tight by assumption, there exists a
  $C_\varepsilon>0$ such that for all $N\in\mathbb N$,
  \begin{equation}
    \label{eq:l21d}
  \sum_{i=2}^{2k-1}\mathbf P^N\big\{ S_{\varepsilon}(\mathcal U_{i\tfrac{\varepsilon}{2}}^N)>
  {C}_\varepsilon\big\}\leq 2\varepsilon.
  \end{equation}

  Applying \eqref{eq:15} therefore yields that for all $N\in\mathbb
  N$,
  \begin{equation}
    \label{eq:P133}
    \begin{aligned}
      \mathbf P^N\big\{\sup_{t\in[\varepsilon,T)}
      S_{2\varepsilon}(\mathcal U_t^N)> {C}_\varepsilon \big\} &\leq
      \sum_{i=2}^{2k-1} \mathbf
      P^N\big\{\sup_{t\in[i\tfrac{\varepsilon}{2},
        (i+1)\tfrac{\varepsilon}{2})}
      S_{2\varepsilon}(\mathcal U_t^N)> {C}_\varepsilon\big\}  \\
      &\leq \sum_{i=2}^{2k-1}\mathbf P^N\big\{S_{\varepsilon}(\mathcal
      U_{i\tfrac{\varepsilon}{2}}^N)> {C}_\varepsilon \big\} \leq
      2\varepsilon
    \end{aligned}
  \end{equation}
  and the assertion follows. \sm

  (b) We extend the notion introduced in \eqref{eq:861b} by setting
  for $\varepsilon>0$ and $0<\zeta<1$,
  \begin{align}
    \label{eq:856}
    \widetilde S_{2\varepsilon, \zeta} (\mathcal U_t^N) & :=
    \inf_{\mathcal J\subseteq \mathcal I_t: \;\; \mu_t^N(\mathcal
      J)\leq \zeta} \#\{A_{t-\varepsilon}(i,t): i\in \mathcal I\setminus
    \mathcal J\}.
  \end{align}
  In particular, $\widetilde S_{2\varepsilon}(\mathcal U_t^N) =
  \widetilde S_{2\varepsilon, 2\varepsilon}(\mathcal U_t^N)$, and thus
  for all $0<\zeta<1$ and $t\in[\varepsilon,T)$, the family
  $\{\widetilde S_{2\varepsilon, \zeta}(\mathcal U_t^N):\, N\in\N\}$
  is tight by Assumption (ii.ii).

  Let $t$, $t'$, $\delta$ and $\delta'$ be such that
  $[t-\delta,t]\subseteq [t'-\delta', t']$. By definition of
  $\widetilde S_{2\delta,\zeta}(\mathcal U_t^N)$, for all $0<\zeta<1$,
  $t<\tau^N$ and $N\in\mathbb N$ there is a ${\mathcal
    A}^N_{t}$-measurable subset $\mathcal J^{N,\zeta,t}\subseteq
  \mathcal I_t^N$ such that $\mu^N_t(\mathcal J^{N,\zeta,t})\leq
  \zeta$ and $\mathcal I^N_t\setminus \mathcal J^{N,\zeta,t}$ can be
  covered by $\widetilde{S}_{2\delta,\zeta}(\mathcal U^N_t)$ balls of
  radius $2\delta$. Moreover, for all $\zeta,\zeta'\in(0,1)$,
  \begin{equation}
    \label{e:ancespart}
    \begin{aligned}
      \big\{\widetilde{S}_{2\delta, \zeta}(\mathcal
      U_t^N)<\widetilde{S}_{2\delta',\zeta'}(\mathcal U_{t'}^N)\big\}
      &\subseteq \big\{\mu^N_{t'}\big(D_{t'}(\mathcal
      J^{N,\zeta,t},t)\big) > \zeta'\big\},
    \end{aligned}
  \end{equation}
  and hence
  \begin{equation}
    \label{e:ancespart1}
    \begin{aligned}%
      &\mathbf{P}^N\big\{\widetilde{S}_{2\delta,\zeta}(\mathcal
      U_t^N)<\sup_{t'\in[t,(t-\delta)+\delta')}\widetilde{S}_{2\delta',\zeta'}(\mathcal
      U_{t'}^N)\big\}
      \\
      &\le \mathbf{P}^N\big\{\sup_{t'\in[t,(t-\delta)+\delta')}
      \mu^N_{t'}\big(D_{t'}(\mathcal J^{N,\zeta,t},t)\big)>
      \zeta'\big\}.
    \end{aligned}
  \end{equation}

  Fix $T>0$ and $\varepsilon>0$. Without loss of generality, we assume
  that $T=k\varepsilon$ for some $k\in\mathbb N$ as well as
  $\tau^N\geq T$. By Condition~(i) of Proposition~\ref{PP:02} applied
  ($2k$ times) with $s:=\tfrac{1}{2}i\varepsilon$, $i=0,...,2k-1$, we
  can choose a $\zeta=\zeta(\varepsilon,T)$ suitably small such that
  for each $i=0,...,2k-1$ and for all $\mathcal A_{i\tfrac\varepsilon
    2}^N$-measurable sets $\mathcal J^{N, \zeta, \tfrac 12
    i\varepsilon }\subseteq \mathcal I_{\tfrac 12 i\varepsilon }^N$
  with $\mu^N_s(\mathcal J^{N, \zeta, \tfrac 1 2 i \varepsilon
  })\leq\zeta$,
  \begin{equation}
    \label{e:001}
    \limsup_{N\to\infty}\mathbf{P}^N
    \big\{\sup_{t\in[\tfrac{1}{2}i\varepsilon,\tfrac{1}{2}(i+1)\varepsilon)}
    \mu^N_{t}(D_{t} (\mathcal J^{N,\zeta,\tfrac{1}{2}i\varepsilon},
    \tfrac{1}{2}i\varepsilon))>\varepsilon\big\}
    \le \tfrac{\varepsilon}{2k}.
  \end{equation}

  Thus, inserting (\ref{e:001}) into (\ref{e:ancespart1}) applied with
  $t=\tfrac 12 i\varepsilon$, $\delta = \tfrac \varepsilon 2$,
  $\delta':=\zeta':=\varepsilon$, and $\zeta$ from \eqref{e:001},
  \begin{equation}
    \label{eq:1112}
    \limsup_{N\to\infty}
    \mathbf P^N\big\{\widetilde{S}_{\varepsilon,\zeta}(\mathcal U_{\tfrac{1}{2}i\varepsilon}^N)<
    \sup_{t\in[\tfrac{1}{2}i\varepsilon,\tfrac{1}{2}(i+1)\varepsilon)}\widetilde{S}_{2\varepsilon,2\varepsilon}
    (\mathcal U_{t}^N)\big\}
    \le\tfrac{\varepsilon}{2k}.
  \end{equation}

  Since for all $\zeta\in(0,1)$, $t\in[\varepsilon,T)$, the family
  $\{\widetilde S_{2\varepsilon,\zeta}(\mathcal U_t^N);\,N\in\N\}$ is
  tight by assumption (ii.ii), and $\{\mathcal U_0^N: N\in\mathbb N\}$
  is assumed to be tight as well, there exists a $C_\varepsilon>0$
  such that for all $N\in\mathbb N$,
  \begin{equation}
    \label{eql21:c}
    \sum_{i=0}^{2k-1}\mathbf P^N\big\{ \widetilde S_{\varepsilon,\zeta}(\mathcal U_{\tfrac 12 i \varepsilon}^N)>
    {C}_\varepsilon\big\}\leq\varepsilon.
  \end{equation}

  Therefore
  \begin{equation}
    \label{eq:P221}
    \begin{aligned}
      \limsup_{N\to\infty} & \mathbf P^N\big\{\sup_{t\in[0,T)}
      \widetilde{S}_{2\varepsilon}(\mathcal U_t^N)>C_\varepsilon\big\}
      \\
      &\leq\limsup_{N\to\infty} \sum_{i=0}^{2k-1}\mathbf
      P^N\big\{\sup_{t\in[\tfrac{1}{2}i\varepsilon,\tfrac{1}{2}(i+1)\varepsilon)}\widetilde{S}_{2\varepsilon}(\mathcal
      U_{t}^N)>C_\varepsilon\big\}
      \\
      &\leq \limsup_{N\to\infty}\sum_{i=0}^{2k-1} \mathbf
      P^N\big\{\widetilde{S}_{\varepsilon,\zeta}(\mathcal
      U_{\tfrac{1}{2} i\varepsilon}^N)<
      \sup_{t\in[\tfrac{1}{2}i\varepsilon,\tfrac{1}{2}(i+1)\varepsilon)}\widetilde{S}_{2\varepsilon}(\mathcal
      U_{t}^N)\big\} \\ & \qquad \qquad \qquad \qquad \qquad \qquad
      +\sum_{i=0}^{2k-1}\mathbf
      P^N\big\{\widetilde{S}_{\varepsilon,\zeta}(\mathcal
      U_{\tfrac{1}{2}i\varepsilon}^N)>C_\varepsilon\big\}
      \\
      &\le 2\varepsilon,
    \end{aligned}
  \end{equation}
  which finally shows the assertion.
\end{proof}\sm

\begin{proof}[Proof of Proposition~\ref{PP:02}]
  Fix $T>0$ and $\delta>0$.

  (a) Since Conditions~(i) and~(ii.i) from Proposition~\ref{PP:02}
  hold, we find for all $n\in\N$ a $C_{\delta 2^{-n}}>0$ such that
  (\ref{eq:l11}) and (\ref{eq:l21}) hold with $\varepsilon=\delta
  2^{-n}$. Put
  \begin{equation}
    \label{eq:T021}
    \Gamma_{1,\delta}
    :=
    \big\{\smallu \in \mathbb U:\,w_\smallu(
    [C_{\delta 2^{-n}},\infty))\leq \delta2^{-n},\; \text{ for all
    }n\in\mathbb N\},
  \end{equation}
  and
  \begin{equation}
    \label{eq:T021b}
    \Gamma_{2,\delta}
    :=
    \big\{\smallu \in \mathbb U_c:\,S_{2\delta 2^{-n}}(\smallu)
    \leq C_{\delta 2^{-n}},\; \text{ for all }  n\in\mathbb N\big\},
  \end{equation}
  where we denote by $S_{2\delta 2^{-n}}(\smallu)$ the number of balls
  of radius $\delta 2^{-n}$ needed to cover $\smallu$. Then
  $\Gamma_{1,\delta}\cap\Gamma_{2,\delta}$ is relatively compact in
  $\mathbb U_c$ by Proposition \ref{noteP:04}.  Moreover, by Lemma
  \ref{l1},
  \begin{equation}
    \label{eq:T022}
    \begin{aligned}
      &\inf_{N\in\mathbb N} \mathbf P^N\big\{\mathcal U_t^N
      \in\Gamma_{1,\delta},\; \text{ for all
      }t\in[0,T\wedge\tau^N)\big\}
      \\
      &\ge 1 -\sum_{n=1}^\infty\sup_{N\in\mathbb N}\mathbf
      P^N\big\{\sup_{t\in[0,T\wedge\tau^N)} w_{\mathcal
        U_t^N}([C_{2^{-n}\delta},\infty))>2^{-n}\delta\big\}
      \\
      &\geq 1 -\sum_{n=1}^\infty 2^{-n}\delta = 1 - \delta.
    \end{aligned}
  \end{equation}

  Similar calculations based on Lemma~\ref{l2} show that
  \begin{equation}
    \label{eq:T023}
    \begin{aligned}
      \inf_{N\in\mathbb N} \mathbf P^N\big\{\mathcal U_t^N \in
      &\Gamma_{2,\delta},\; \text{for all } t\in[\delta, T)\big\}
      \geq 1-2\delta.
    \end{aligned}
  \end{equation}

  Hence
  \begin{equation}
    \label{eq:T024}
    \inf_{N\in\mathbb N} \mathbf P^N(\mathcal U_t^N \in
    \Gamma_{1,\delta} \cap \Gamma_{2,\delta},\;\text{for all }
    \delta \in [t, T\wedge\tau^N)) \geq 1-3\delta,
  \end{equation}
  and \eqref{eq:P02b} follows. \sm

  (b) Assume the conditions (i) and (ii.ii) from
  Proposition~\ref{PP:02}. Then for all $n\in\N$ there is a $C_{\delta
    2^{-n}}>0$ such that (\ref{ww1}) and (\ref{eq:l21b}) hold with
  $\varepsilon=\delta 2^{-n}$.  Put
  \begin{equation}
    \label{Gamma3}
    \Gamma_{3,\delta}
    :=
    \big\{\smallu \in \mathbb U_c:\, \widetilde{S}_{2\delta2^{-n}}(\smallu)\leq C_{\delta 2^{-n}},\; \text{ for all }
    n\in\mathbb N\big\},
  \end{equation}
  where $\widetilde{S}_{2\delta 2^{-n}}(\smallu)$ denotes the number
  of $2\delta 2^{-n}$-balls needed to cover a frequency of $(1-2\delta
  2^{-n})$ of $\smallu$. By
  Proposition~\ref{noteP:05}, $\Gamma_{1,\delta}\cap\Gamma_{3,\delta}$
  is compact in $\mathbb U$.  Moreover, by a similar argument as above
  we find that
  \begin{equation}
    \label{Gamma3est}
    \inf_{N\in\mathbb N}\mathbf P^N\big\{\mathcal U_t^N \in
    \Gamma_{3,\delta},\; \text{ for all } t\in[0,T\wedge\tau^N)\big\} \geq
    1-2\delta,
  \end{equation}
  which gives \eqref{eq:P02c}.
\end{proof}\sm

\subsection{The compact containment condition for Moran models}
\label{sub:cpapp}
The following result is an important step in the proof of tightness of
the family of tree-valued Moran dynamics. Recall the distance
distribution $w_\smallx$ from (\ref{eq:dist}). The next result states
that the family $\{{\mathcal U}^N;\,N\in\N\}$ satisfies all
assumptions from Proposition \ref{PP:02}.

\begin{proposition}[Compact containment]
  Let for each $N\in\N$, $\mathcal U^N=(\mathcal U^N_t)_{t\geq 0}$ be
  the tree-valued Moran dynamics of population size $N$.  Assume that
  the family $(\mathcal U_0^{N})_{N=1,2,...}$ is tight in $\mathcal
  M_1(\mathbb U)$. Then,\label{P:02} the family $\{\mathcal U^N: N
  \in\mathbb N\}$ satisfies the Conditions~(i),~(ii.i) and~(ii.ii)
  from Proposition~\ref{PP:02}.
\end{proposition}

\begin{proof} Fix $T>0$, $\varepsilon\in(0,T)$ and $N\in\N$, and note
  that $\tau^N=\infty$. As for {\em Condition~(i)}, let $s\in[0,T)$
  and consider a ${\mathcal A}^N_{s}$-measurable sequence $(\mathcal
  J^N)_{N\in\N}$ with $\mathcal J^N\subseteq \mathcal I$. Then the
  process $Y^N:=(Y^N_t)_{t\in[s,T)}$, defined for $t\in[s,T)$ as
  $Y^N_t:=\frac{\#D_t(s,\mathcal J^N)}{\#\mathcal I}$, is a
  $\{0,\tfrac{1}{N},...,1\}$-valued birth-death process with
  transitions $y\mapsto y\pm\tfrac{1}{N}$ (each) with rate
  $\tfrac{1}{2}N^2\gamma y(1-y)$. In particular, $Y^N$ is a
  martingale, and therefore the claim follows by Doob's maximum
  inequality. \sm

  To verify {\em Condition~(ii.i), notice that} the family
  $\{S^N_{2\varepsilon}(t);\,N\in\N\}$ is stochastically uniformly
  bounded by $K_\varepsilon$, where $K = (K_t)_{t\geq 0}$ denotes the
  process for the number of lines in a rate $\gamma$ Kingman
  coalescent. In particular, the family
  $\{S^N_\varepsilon(t);\,N\in\N\}$ is tight.

  {\em Condition~(ii.ii)} directly follows from Condition~(ii.i).
\end{proof}

\section{Limit points have continuous paths}
\label{sub:52}
It is well-known that the measure-valued Fleming-Viot process has
continuous paths (e.g., \cite{Dawson1993}). In this section we show
that the same is true for the tree-valued Fleming-Viot dynamics by
controlling the jump sizes in the approximating sequence of Moran
models.

Recall from Definition \ref{def:moran} the tree-valued Moran model
${\mathcal U}^{N}$ of population size $N\in\N$.

\begin{proposition}[Limit points have continuous paths]
  If $\mathcal U^N \TNo\mathcal U$ for some process $\mathcal U$ with
  sample paths in the Skorohod space, ${\mathcal
    D}_{\mathbb{U}}([0,\infty))$, of c\`adl\`ag functions from
  $[0,\infty)$ to $\mathbb{U}$, then $\mathcal U\in{\mathcal
    C}_{\mathbb{U}}([0,\infty))$, almost surely. \label{P:03}
\end{proposition}

\begin{proof}
  Recall from Section~\ref{Sub:conv} the construction of the
  tree-valued Moran dynamics $\mathcal U^N = (\mathcal U_t^N)_{t\geq
    0}$, $\mathcal U_t^N = \overline{(\mathcal I, r_t^N, \frac 1N\sum
    \delta_i)}$ with $\mathcal I = \{1,...,N\}$ based on Poisson point
  processes $\{\eta^{i,j};\,1\le i,j\le N\}$. (Compare also with
  Figure~\ref{fig:MM}). In addition, recall the modified Eurandom
  metric from Definition \ref{def:eur}. Note that the tree-valued
  Moran dynamics has paths in ${\mathcal D}_{\mathbb U_c}(\R_+)$,
  almost surely.

  If $\eta^{k,l}\{t\}=0$ for all $k,l\in\mathcal I$, then ${\mathcal
    U}_{t-}^N={\mathcal U}_{t}^N$.  Otherwise, if $\eta^{k,l}\{t\}=1$,
  for some $k,l\in\mathcal I$, then
  \begin{equation}\label{qq3}
    \begin{aligned}
      d_{\text{Eur}}'\big({\mathcal U}_{t-}^N,{\mathcal U}_t^N\big) &
      \leq \frac{1}{N^2} \sum_{i,j}|r_{t-}^N(i,j) - r_t^N(i,j)|\wedge
      1 \\ & = \frac{1}{N^2} \sum_{i=l\text{ or }j=l} |r_{t-}^N(i,j) -
      r_t^N(i,j)|\wedge 1 \\ & \leq \frac 2 N
    \end{aligned}
  \end{equation}
  and therefore
  \begin{equation}\label{qqq6}
    \int_0^\infty\mathrm{d}T\,e^{-T}
    \sup_{t\in[0,T]}d_{\text{Eur}}'\big({\mathcal U}^N_{t-},{\mathcal U}^N_{t}
    \big)
    \leq
    \frac 2N,
  \end{equation}
  for all $T>0$ and almost all sample paths $\mathcal U^N$.  Hence the
  assertion follows by Theorem~3.10.2 in \cite{EthierKurtz86}.
\end{proof}\sm

\section{Proofs of the main results
  (Theorems~\ref{T:01},~\ref{T:02},~\ref{T:03})}
\label{S:proofs}
In this section we give the proof of the main results stated in
Section~\ref{S:results}. Theorems~\ref{T:01} and~\ref{T:02} are proved
simultaneously.

\begin{proof}[Proof of Theorems~\ref{T:01} and~\ref{T:02}]
  Recall, for each $N\in\N$, the state-space $\mathbb{U}_N$, and the
  $\mathbb{U}_N$-valued Moran dynamics, ${\mathcal U}^N=({\mathcal
    U}^N_t)_{t\ge 0}$, from (\ref{e:MoranN}) and Definition
  \ref{def:moran}, respectively. Let $\mathbf P_0\in\mathcal
  M_1(\mathbb U)$ be the distribution of $\mathcal U_0$ and $\mathbf
  P_0^N\in\mathcal M_1(\mathbb U_N)$ be the distribution of $\mathcal
  U_0^N$ such that $\mathcal U_0^N\Longrightarrow \mathcal U_0$ as
  $N\to\infty$.

  By Proposition~\ref{P:06}, the
  $(\mathbf{P}_0^N,\Omega^{\uparrow,N},\Pi_N^1)$-martingale problem is
  well-posed, and is solved by ${\mathcal U}^N$.
  Proposition~\ref{P:01} implies with a standard argument (see, for
  example, Lemma~4.5.1 in~\cite{EthierKurtz86}) that if ${\mathcal
    U}^N\Rightarrow{\mathcal U}$, for some ${\mathcal U}\in{\mathcal
    D}_{\mathbb{U}}([0,\infty))$, as $N\to\infty$, then ${\mathcal U}$
  solves the $(\mathbf{P}_0,\Omega^{\uparrow},\Pi^1)$-martingale
  problem. Hence for {\em existence} we need to show that the sequence
  $\{{\mathcal U}^N;\,N\in\N\}$ is tight, or equivalently by
  Remark~\ref{rem:05} combined with Remark~4.5.2 in
  \cite{EthierKurtz86} that the compact containment condition in
  $\mathbb U$ holds. However, the latter follows directly from
  Propositions \ref{P:02} and \ref{PP:02}.
  \sm

    By standard theory (see, for example, Theorem~4.4.2 in
  \cite{EthierKurtz86}), {\em uniqueness} of the
  $(\mathbf{P}_0,\Omega^\uparrow,\Pi^1)$-martingale problem follows
  from uniqueness of the one-dimensional distributions of solutions of
  the $(\mathbf{P}_0, \Omega^\uparrow, \Pi^1)$-martingale problem. The
  latter can be verified using the duality of the tree-valued
  Fleming-Viot dynamics to the tree-valued Kingman coalescent,
  ${\mathcal K}:=({\mathcal K}_t)_{t\ge 0}$, as defined in
  (\ref{grevy}).  That is, if ${\mathcal U}=({\mathcal U}_t)_{t\ge 0}$
  is a solution of the $(\mathbf
  P_0,\Omega^\uparrow,\Pi^1)$-martingale problem, then
  \eqref{eq:dualRel} holds for all $\kappa\in\mathbb{K}$, $t\ge 0$ and
  $H\in\mathcal H$. Since $\mathcal H$ is separating in $\mathcal
  M_1(\mathbb U)$ by Proposition~\ref{P:dual}(i), uniqueness of the
  one-dimensional distributions follows.
\sm

\sloppy So far we have shown that the
$(\mathbf{P}_0,\Omega^\uparrow,\Pi_1)$-martingale problem is
well-posed and its solution arises as the weak limit of the solutions
of the $(\mathbf P^N_0, \Omega^{\uparrow,N},\Pi_N^1)$-martingale
problems. In particular, the tree-valued Moran dynamics converge to
the tree-valued Fleming-Viot dynamics. Hence we have shown
%the first assertion of
Theorem~\ref{T:01} and Theorem~\ref{T:02}.
\end{proof}\sm

\begin{proof}[Proof of Proposition~\ref{P:07}]
  (i), (ii) The tree-valued Fleming-Viot dynamics is the weak limit of
  tree-valued Moran dynamics. Hence, Propositions~\ref{P:02} and
  \ref{PP:02} imply that the tree-valued Fleming-Viot dynamics have
  values in the space of compact ultra-metric measure spaces for each
  $t>0$, almost surely. In addition, the tree-valued Fleming-Viot
  dynamics has continuous paths by Proposition~\ref{P:03}, almost
  surely.
\end{proof}\sm

\begin{proof}[Proof of Proposition~\ref{P:05}] Note that the strong
  Markov property follows from the Feller property, \cite[Theorem
  4.2.7]{EthierKurtz86}.  (By completeness, we can assume the
  filtration generated by the tree-valued Fleming-Viot dynamics is
  right-continuous, as needed in this Theorem.) Let $\mathcal
  U^{\smallu}=(\mathcal U_t^\smallu)_{t\geq 0}$ be the solution of the
  $(\delta_\smallu, \Omega^\uparrow, \Pi^1)$-martingale problem,
  i.e.\, the tree-valued Fleming-Viot dynamics, started in $\mathcal
  U_0 = \smallu$. For the Feller property, it suffices to show that
  $\smallu'\to\smallu$ implies that $\mathcal U^{\smallu'}_t
  \Longrightarrow \mathcal U^\smallu_t$ for all $\smallu\in\mathbb U$
  and $t>0$.  Recall the coupled tree-valued Moran dynamics from
  Section \ref{sec:coup}. For $\smallu,\smallu'\in\mathbb U$, take
  $\smallu_N, \smallu_N'\in\mathbb U_N$ with $\smallu_N\to\smallu$ and
  $\smallu_N'\to\smallu'$ in the Gromov-weak topology. Let $(\mathcal
  U_t^{N,1},\mathcal U_t^{N,2})_{t\geq 0}$ be the coupled tree-valued
  Moran dynamics, started in $(\smallu_N, \smallu_N')$. Since
  $\{(\mathcal U_t^{N,k})_{t\geq 0}, N\in\mathbb N\}$ is tight in
  $\mathbb{U}$ by Theorem \ref{T:02}, $k=1,2$, $\{(\mathcal U_t^{N,1},
  \mathcal U_t^{N,2})_{t\geq 0}: N\in\mathbb N\}$ is tight in $\mathbb
  U\times\mathbb{U}$. Let $(\mathcal U_t^{\smallu}, \mathcal
  U_t^{\smallu'})_{t\geq 0}$ be a weak limit point which must be a
  coupling of tree-valued Fleming-Viot dynamics by
  construction. Moreover, since the modified Eurandom metric (see
  Definition \ref{def:eur}) is continuous in the Gromov-weak topology
  and bounded
  \begin{equation}
    \label{eq:njio}
    \begin{aligned}
      \mathbf E[d_{\text{Eur}}'(\mathcal U_t^\smallu, \mathcal
      U_t^{\smallu'})] & = \lim_{N\to\infty} \mathbf
      E[d_{\text{Eur}}'(\mathcal U_t^{N,1}, \mathcal U_t^{N,2})] \\ &
      \leq \lim_{N\to\infty} d_{\text{Eur}}'(\smallu_N, \smallu_N')\\ &=
      d_{\text{Eur}}'(\smallu, \smallu')
    \end{aligned}
  \end{equation}
  by Proposition \ref{P:tbc}. In particular, $\smallu'_n\to\smallu$, as $n\to\infty$,
  implies that $$\mathbf E[d_{\text{Eur}}'(\mathcal U_t^\smallu,
  \mathcal U_t^{\smallu'_n})]\tno 0,$$ which in turn implies $\mathcal
  U_t^{\smallu'_n}\Longrightarrow \mathcal U_t^\smallu$, as $n\to\infty$, by Remark~\ref{rem:eur}.
\end{proof}\sm

\begin{proof}[Proof of Corollary \ref{cor:qv2}]
  For $\Phi=\Phi^{n,\phi}$ as in the Corollary, observe that $\langle
  \nu^\smallu, \phi\rangle^2 = \langle \nu^\smallu,
  (\phi,\phi)_n\rangle$ with $(\phi,\phi)_n$ from \eqref{eq:qvprep1}.
  Therefore, given $\mathcal U_t=\smallu$, we compute (compare with
  \cite[Proof of Theorem~1.1]{FukushimaStroock1986})
  \begin{equation}
    \label{eq:qv3}
    \begin{aligned}
      \tfrac{\mathrm d\langle \Phi(\mathcal U)\rangle_t}{\mathrm{d}t}
      &=
      \Omega^\uparrow\Phi^2(\smallu)-2\Phi(\smallu)\Omega^\uparrow\Phi(\smallu)
      \\
      &= \langle\nu^\smallu,\langle\nabla(\phi,\phi)_n,
      \underline{\underline 2}\rangle-2 (\phi, \langle\nabla\phi,
      \underline{\underline 2}\rangle)_n\rangle \\ & \qquad +
      \frac\gamma 2 \sum_{k,l=1}^n \langle \nu^\smallu, (\phi\circ
      \theta_{k,l}, \phi)_n + (\phi, \phi\circ \theta_{k,l})_n -
      2(\phi,\phi\circ \theta_{k,l})_n\rangle \\ & \qquad +
      \gamma\sum_{k,l=1}^n \big( \langle \nu^\smallu,
      (\phi,\phi)_n\circ \theta_{k,n+l}\rangle - \langle\nu^\smallu,
      (\phi,\phi)_n\rangle\big)
    \end{aligned}
  \end{equation}
  and the result follows from the first two terms vanishing and
  \begin{equation}
    \label{eq:qv4}
    \begin{aligned}
      \sum_{k,l=1}^n \langle \nu^\smallu,
      (\phi,\phi)_n\circ\theta_{k,n+l}\rangle & = \sum_{k,l=1}^n
      \langle \nu^\smallu,
      (\bar\phi,\bar\phi)_n\circ\theta_{k,n+l}\rangle
      \\
      &= n^2 \langle \nu^\smallu,
      (\bar\phi,\bar\phi)_n\circ\theta_{1,n+1}\rangle
    \end{aligned}
  \end{equation}
  with the symmetrization $\bar\phi$ introduced in Remark
  \ref{Rem:13}(iii) and $\Phi^{n,\phi} = \Phi^{n,\bar
    \phi}$.
\end{proof}\sm

\begin{proof}[Proof of Theorem \ref{T:03}]
  In order to prove Theorem \ref{T:03} we need two ingredients:
  \begin{itemize}
  \item The family $\{\mathcal U_t;\, t>1\}$ is tight.
  \item $\mathbf{E}^{\delta_\smallu}[\Phi(\mathcal U_{t})] \to
    \mathbf E[\Phi(\mathcal U_\infty)]$, as $t\to\infty$, for all
    $\Phi\in\Pi^1$ and $\smallu\in\mathbb U$.
  \end{itemize}
  Then, Theorem \ref{T:03} follows from Lemma~3.4.3 together with Theorem~3.4.5 of
  \cite{EthierKurtz86}.

    We show tightness of $\{\mathcal U_t;\, t>1\}$ in $\mathbb U$
  using Theorem~3 and (3.3) of \cite{GrePfaWin2006a}. First, recalling
  \eqref{eq:dist} and when $\mathbf E[w_{\mathcal U_t}]$ is the first
  moment measure of $w_{\mathcal U_t}\in\mathcal M_1(\mathcal
  M_1(\mathbb R_+))$, for $t$ and $C>0$,
  \begin{align}
    \mathbf E[w_{\mathcal U_t}] ([C,\infty)) = \begin{cases}
      e^{-\gamma t} \mathbf E[w_{\mathcal U_0}]([C-t,\infty)), & C\geq
      t \\ e^{-\gamma C}, & C < t. \end{cases}
  \end{align}
  Indeed, by exchangeability $\mathbf{E}[w_{{\mathcal
      U}_t}]([C,\infty))$ equals the probability that a ``typical''
  pair of individuals drawn from the population at time $t$ has
  distance at least $C$, if $t\le C$ then this event equals the event
  that their ancestral lines do not coalesce in the time window
  $[0,t]$ and that the distance of their ancestors at time $0$ is at
  least $C$. This event has probability $e^{-\gamma t}$ (no
  coalescence for at least time $t$) times $\mathbf{E}[w_{{\mathcal
      U}_0}]([C-2t,\infty))$. If $t>C$ then the distance between a
  ``typical'' pair of individuals to be at least $C$ is equivalent to
  that their ancestral lines do not coalesce in the time window
  $[0,C]$ which has probability $e^{-\gamma C}$.

  So, for given $\varepsilon>0$, choose $C>0$ large enough such that
  $\mathbf E[w_{\mathcal U_0}]([C,\infty))<\varepsilon$ and
  $e^{-\gamma C}<\varepsilon$. Then, $ \mathbf E[w_{\mathcal U_t}]
  ([2C,\infty)) < \varepsilon$ for all $t>0$ and so, $\{\mathbf
  E[w_{\mathcal U_t}], t>1\}$ is tight.

  Secondly, for $\mathcal U_t = \overline{(U_t,r_t,\mu_t)}$, we have to
  show that for $0<\varepsilon<1$ there is $\delta>0$ with
  \begin{align}
    \sup_{t>1} \mathbf E[\mu_t\{u: \mu_t(B_\varepsilon(u))\leq
    \delta\}]<\varepsilon.
  \end{align}

 Note that the expectation on the left hand side does not depend on $t$.
 Using that $\mathcal U_\infty$ is determined by $\Lambda =
  \gamma\cdot \delta_0$ in (4.7) of \cite{GrePfaWin2006a} we find
  \begin{align}
    \lim_{\delta\to 0}\sup_{t>1} \mathbf E[\mu_t\{u: \mu_t(B_\varepsilon(x))\leq
    \delta\}] & = \lim_{\delta\to 0} \mathbf E[\mu_\infty\{u:
    \mu_\infty(B_\varepsilon(x))\leq \delta\}] = 0
  \end{align}
  by (4.9) and (4.11) of \cite{GrePfaWin2006a}. So, tightness
  follows.

  The fact that the Kingman tree is a unique equilibrium distribution
  is an application of the duality relation from Proposition
  \ref{P:dual}. Fix $\phi\in{\mathcal C}_b^1(\R_+^{\N\choose 2})$. We
  apply the duality relation (\ref{eq:dualRel}) between the
  tree-valued Fleming-Viot dynamics and the tree-valued Kingman
  coalescent which starts in
  $\smallk_0=(\smallp_0,\underline{\underline{r}}'_0)$ with
  $\smallp_0:=\{\{n\},\,n\in\N\}$ and
  $\underline{\underline{r_0}}'\equiv 0$. By construction of the dual
  process $\mathcal K$, $\mathbf
  E^{\delta_{\smallk_0}}[\phi(\underline{\underline{r}}_t)] \to
  \mathbf E[\langle \nu^{\mathcal U_\infty},\phi\rangle ]$ {and
    ${\mathcal P}_t\to\{\N\}$,} as $t\to\infty$ where $\mathcal
  U_\infty$ is the (rate $\gamma$) Kingman measure tree {from
    (\ref{ag2b})}. Hence, {by (\ref{eq:dualRel}),}
  \begin{equation}\label{e:momcon}
    \begin{aligned}
      \lim_{t\to\infty}\mathbf{E}^{\delta_\smallu}\big[\langle\nu^{\mathcal
        U_{t}}, \phi\rangle \big] &={
        \lim_{t\to\infty}\mathbf{E}^{\delta_{\smallk_0}}\big[\int_{\R_+^{\N\choose
            2}}\nu^{\smallu}(\mathrm{d}\underline{\underline{r}})\,
        \phi\big((\underline{\underline{r}})^{{\mathcal
            P}_t}+\underline{\underline{r_t}}'\big) \big]}
      \\
      &= \lim_{t\to\infty}\mathbf{E}^{\delta_{\smallk_0}} \big[
      \phi\big(\underline{\underline{r}}'_t\big)\big] \\&= \mathbf{E}
      \big[\langle \nu^{\mathcal U_{\infty}}, \phi\rangle \big].
    \end{aligned}
  \end{equation}
  Since $\phi\in{\mathcal C}_b^1(\R_+^{\N\choose 2})$ was chosen
  arbitrarily, (ii) follows and we are done.
\end{proof}\sm

\section[Proof of Theorems~\ref{T:04} and~\ref{T:05}]{Proof of the
  applications (Proof of Theorems~\ref{T:04} and~\ref{T:05})}
\label{S.proofapp}
In this section we prove the results stated in Section~\ref{S:kurtz}.

\begin{proof}[Proof of Lemma~\ref{L:02}] Consider the \emph{traveling
    salesperson problem} for a salesperson who must visit all
  $x_1,...,x_n$ and who starts at one $x_i$ to which she comes back at
  the end of the trip. It is easy to see that such a path must pass
  all edges of the subtree spanned by $x_1,...,x_n$ in both
  directions, so the length of the path is at least twice the tree
  length. It is also easy to see that taking an optimal path and
  leaving out $x_i$ gives an optimal path for the remaining leaves
  $x_1,...,x_{i-1}, x_{i+1},...,x_n$.

  We claim that there is one path connecting the set of leaves such
  that each edge in the tree is passed exactly twice, which is
  equivalent to the assertion of the Lemma. Assume to the contrary
  that such an order does not exist. We take a path of minimal
  length. There must be one edge which is visited at least four
  times. W.l.o.g.\ we assume that this edge is internal, i.e.\ not
  adjacent to any $x_i$. So there are four points $x_i,x_j,x_k,x_l\in
  X$, visited in the order $x_i, x_j, x_k, x_l, x_i$, such that
  $[x_i,x_j]\cap [x_k, x_l]$ is visited at least four times, where
  $[x, y]$ is the path from $x$ to $y$ in $X$. Since leaving out
  leaves gives again an optimal path, leaving out all leaves except
  $x_i, x_j, x_k, x_l$ must lead to an optimal path connecting these
  four points. However, this optimal path must be $x_i, x_j, x_l, x_k,
  x_i$ (or its reverse), since this path passes all edges only
  twice. Hence, we have a contradiction and the assertion is proved.
\end{proof}\sm

\begin{proof}[Proof of Theorem~\ref{T:04}]
  We first show {\em injectivity} of $\xi$. Assume we are given a
  compact ultra-metric measure space $(U_0,r_0,\mu_0)$ and its
  equivalence class $\smallu_0 = \overline{(U_0,r_0,\mu_0)}$. We show
  that if $\lambda := \xi(\smallu_0)$, then
  $\xi^{-1}(\{\lambda\})=\{\smallu_0\}$. We do this by explicitly
  reconstructing $\smallu_0$ from $\lambda$.

  We proceed in three steps.  In the first two steps we consider the
  case where $\mu_0$ is supported by finitely many atoms.  In Step~1
  we follow an argument provided to us by Steve Evans which explains
  how to recover the isometry class of $(\mathrm{supp}(\mu_0),r_0)$
  from $\lambda$. In Step~2 we then recover the measure
  $\mu_0$. Finally, the case of a general element in $\mathbb{U}$ is
  obtained by approximation via finite ultra-metric measure spaces in
  Step~3.

  \subsubsection*{Step~1 (Evans's reconstruction procedure for finite
    trees)} Assume that $\smallu\in\xi^{-1}(\{\lambda\})$ and that
  $\smallu=\overline{(U,r,\mu)}$ with $\#\mathrm{supp}(\mu)<\infty$.
  Put
  \begin{equation}\label{eq:qqq12a}
    A_N
    :=
    \big\{({{l}}_1:=0,{{l}}_2,...):\,{{l}}_k>{{l}}_{k-1} \text{ for exactly }N-1\text{ different }k\big\}.
  \end{equation}
  First observe that $\#\mathrm{supp}(\mu)=N$ if and only if $\lambda$
  is supported on $A_N$. That is, we can recover
  $\#\mathrm{supp}(\mu)$ from $\lambda$.  So, assume that $\mu$ has
  $N$ atoms and w.l.o.g.\ $U:=\{1,...,N\}$. We now recover
  $\underline{\underline{r}}=(r_{i,j})_{1\le i<j\le N}$ from
  $\lambda$.

  For that purpose, introduce on $\mathbb R_+^\N$ the lexicographic
  ordering $\prec$, i.e., $\underline {{l}} \prec \underline {{l}}'$
  iff for $k^\ast := \min\{k: {{l}}_k\neq {{l}}'_k\}$ we have
  ${{l}}_{k^\ast} < {{l}}'_{k^\ast}$. Let
  \begin{align}\label{B}
    B
 :=
    \big\{\underline{{l}}\in\text{supp}(\lambda):\,{{l}}_1<...<{{l}}_N\big\}
  \end{align}
be the space of all vectors $\underline{{{l}}}$ which are
accessible by sequentially sampling the $N$ different points of $U$
and evaluating subsequently the lengths of the sub-trees spanned by
them.  Moreover, let
\begin{equation}
\label{e:last}
   \underline{{{l}}}^\ast:=\mathrm{min}_{\prec}B,
\end{equation}
i.e., $\underline{{{l}}}^\ast:=({{l}}^\ast_k)_{k\in\N}$ is the minimal
element in $B$ with respect to the order relation $\prec$.

W.l.o.g.\ we assume that $U=\{1,...,N\}$ and that for all
$n\in\{1,...,N\}$,
  \begin{equation}\label{enum}
    {{l}}_n^\ast
    :=
    L_n^{(U,r)}(\{1,...,n\}).
  \end{equation}
  Notice that if $d_n^\ast$ denotes the depth of the sub-tree spanned
  by $\{1,...,n\}$, i.e., $d_n^\ast := \tfrac 12 \max\{r(i,j); 1\leq
  i,j\leq n\}$, for $n\in\N$, then $d_1^\ast=0$ and the recursion
  \begin{equation}
    \label{eq:qqq10}
    \begin{aligned}
      d_n^\ast = \frac
      12\big(d_{n-1}^\ast+({{l}}_n^\ast-{{l}}_{n-1}^\ast)\vee
      d_{n-1}^\ast\big).
    \end{aligned}
  \end{equation}
  holds for $n\geq 2$.

  We claim that we can even recover
  $(r_{i,j})_{1\leq i<j\leq N}$ from $({{l}}_n^\ast)_{n=1,...,N}$.
  In fact, for all $n\in\N$,
  \begin{equation}\label{weave} r_{n-1,n} = \min_{1\le k\leq n-1}
    r_{k,n},
  \end{equation}
  To see this, assume
  to the contrary that there is a minimal $n\in\N$ for which we find a $k<n-1$
  such that $r_{k,n}$ is minimal and $r_{k,n} < r_{n-1,n}$. Choose the
  minimal $i$ with $k<i\leq n-1$ and $r_{k,n} < r_{i,n}$. Then,
  sampling the $i$ points $1,2,...,k,...,i-1,n$ (in that order) leads to the sequence of tree lengths
  ${{l}}_1^\ast, {{l}}_2^\ast, ..., {{l}}_{i-1}^\ast, {{l}}_{i-1}^\ast +
  \tfrac 12 r_{k,n}$. However, by the minimality of $i$ we have that
  $r_{k,n}\geq r_{i-1,n}$ and by the ultra-metric property $r_{k,n} <
  r_{i,n}\vee r_{i-1,n} = r_{i-1,i}$. Hence, the above tree lengths
  are smaller (with respect to $\prec$) than ${{l}}_1^\ast,
  {{l}}_2^\ast, ..., {{l}}_{i-1}^\ast, {{l}}_{i}^\ast$ since
  ${{l}}_{i}^\ast \geq {{l}}_{i-1}^\ast + \tfrac 12 r_{i-1,i}$. So,
  assuming that \eqref{weave} does not hold contradicts the assumption
  that ${{l}}^\ast$ is minimal.

However, from (\ref{weave}) we conclude the following recursion: for
all $n\in\{2,...,N\}$ and $1\le k\le n-1$,
  \begin{equation}
    \label{eq:qqq12}
    \begin{aligned}
      r_{k,n} &= r_{k,n-1}\vee
      2\big({{l}}_n^\ast-{{l}}_{n-1}^\ast-(d_n^\ast-d_{n-1}^\ast)\big).
    \end{aligned}
  \end{equation}
  The latter together with the necessary requirements that
  $r_{n,n}:=0$ and $r_{1,2}:=\tfrac{1}{2}{{l}}^\ast_2$ determines the
  metric on $U$ uniquely.

  \subsubsection*{Step~2 (Reconstruction of weights in finite trees)}   In this step we reconstruct
  weights $(p_1,...,p_N)$ on $(\{1,...,N\},r)$ from the given
  $\lambda$. Denote by $\Gamma\subseteq\Sigma_N$ the set of permutations
  of $\{1,...,N\}$ for which the metric $r$ given in Step~1 satisfies
  $r_{i,j}=r_{\sigma(i), \sigma(j)}$, for all $1\leq i,j\leq N$.
  Since we are interested in measure-preserving isometry classes only,
  we need to show that $(p_1,...,p_N)$ are uniquely determined up to
  permutations $\sigma\in\Gamma$.

  For all $\underline
  k=(k_1,...,k_{N-1},k_N)\in\{0,1,...\}^{N-1}\times\{\infty\}$, define
  \begin{align}\label{eq:qqq14}
    \underline {{l}}^\ast_{\underline k} := \big({{l}}_1^\ast=0,
    \underbrace{{{l}}^\ast_1,...,{{l}}^\ast_1}_{k_1-\text{times}},{{l}}_2^\ast,
    \underbrace{{{l}}^\ast_2,...,{{l}}^\ast_2}_{k_2-\text{times}},{{l}}_3^\ast,
    \underbrace{{{l}}^\ast_3,...,{{l}}^\ast_3}_{k_3-\text{times}},...\big)
  \end{align}
  where $\underline {{l}}^\ast$ is the minimal subtree length vector
  in the support of $\lambda$ from Step~1. Observe that sampling from
  the subtree length distribution first the point~$1$ a number of
  $k_1+1$ times, then the point~$2$, then one of the points in
  $\{1,2\}$ a number of $k_2$ times, and so on, results exactly in the
  vector $\underline {{l}}^\ast_{\underline k}$. Hence, taking all
  possible permutations $\sigma\in\Gamma$ into account, and since
  $\lambda(\{\underline{{l}}^\ast\}) = |\Gamma| \cdot \prod_{i=1}^N
  p_i$,
  \begin{equation}\label{eq:qqq15}
    \begin{aligned}
      \lambda(\{\underline{{l}}^\ast_{\underline k}\}) &=
      \big(\prod\nolimits_{i=1}^Np_i\big)\cdot\sum_{\sigma\in\Gamma}\prod_{i=1}^{N-1}
      \Big( \sum_{1\le j\leq i} p_{\sigma(j)}\Big)^{k_i}
      \\
      &=\frac{1}{|\Gamma|}
      \lambda(\{\underline{{l}}^\ast\})\cdot\sum_{\sigma\in\Gamma}\prod_{i=1}^{N-1}
      \Big( \sum_{1\le j\leq i} p_{\sigma(j)}\Big)^{k_i}.
    \end{aligned}
  \end{equation}
  We claim that (\ref{eq:qqq15}) determines $(p_1,...,p_N)$ uniquely
  up to permutations $\sigma\in\Gamma$.

  To see this, observe that the algebra of functions on the
  $N-1$-dimensional simplex $S_N$, generated by the functions
  \begin{align}\label{eq:qqq16}
    \Big\{f((p_1,...,p_{N})):= \prod_{i=1}^{N-1}\big(\sum_{1\le
      j\leq i}p_{j}\big)^{k_i};\;k_1,...,k_{N-1}\in\N_0\Big\}
  \end{align}
  separates points. Hence, $f\in\mathcal C_b(S_N)$ can be
  approximated uniformly by functions in this algebra by the
  Stone-Weierstrass Theorem. Hence, by knowing
  $\lambda(\{\underline{{l}}_{\underline k}^\ast\})/\lambda(\{\underline
  {{l}}^\ast\})$ for all $\underline k$, using \eqref{eq:qqq15}, we
  also know the values of
  \begin{align}
    \frac{1}{|\Gamma|}\sum_{\sigma\in\Gamma} f((p_{\sigma(1)}, ...,
    p_{\sigma(n)})
  \end{align}
  by an approximation argument.

    In particular, we can find
  the set $A:=\{(p_{\sigma(1)}, ..., p_{\sigma(N)}):
  \sigma\in\Gamma\}$.  By setting $\mu\{i\}=p_i$ for an arbitrary
  $(p_1,...,p_N)\in A$ we have recovered $\mu$ uniquely up to
  isometries such that $\xi^{-1}(\{\lambda\})=\{\smallu\}$ by
  construction.

  \subsubsection*{Step~3 (General ultra-metric measure spaces)}
  Let $\smallu=\overline{(U,r,\mu)}\in\xi^{-1}(\{\smallu_0\})$ not necessarily finite anymore.
  We shall approximate $\smallu$ by finite ultra-metric measure spaces
  which we then treat as described in the first two steps.

 For that purpose, let for all $\varepsilon>0$,
 the ($\varepsilon$-shrunken) pseudo-metric $r_\varepsilon$ on $U$ by
 putting
 \begin{equation}\label{e:epss}
   r_\varepsilon
 :=
   0\vee (r - \varepsilon).
 \end{equation}
 Notice that since $\overline{(\mathrm{supp}(\mu),r)}$ is
 ultra-pseudo-metric, $\overline{(\mathrm{supp}(\mu),r_\varepsilon)}$ is ultra-pseudo-metric as well, for all $\varepsilon$.

 Moreover, for all $\varepsilon>0$ there is a covering of $U$
 of disjunct balls $B_1, B_2,...\subseteq U$ of radius
 $\varepsilon$ with $\mu(B_1)\geq \mu(B_2)\geq \ldots$. Take
 $N_\varepsilon$ large enough such that for
 $B_\varepsilon=\bigcup_{i=1}^{N_\varepsilon} B_i$ we have
 $\mu(B_\varepsilon)>1-\varepsilon$. Set
 $\mu_\varepsilon(\cdot):=\mu(\cdot|B_\varepsilon)$ and
  \begin{align}\label{eq:gpwlt1}
    \smallu^\varepsilon := \overline{(U, r_\varepsilon,
      \mu_\varepsilon)}.
  \end{align}
  Then $\smallu^\varepsilon$ is a finite
  metric measure space and $\smallu^\varepsilon \to \smallu$ in
  the Gromov-weak topology, as $\varepsilon\to 0$.

   Given $u_1,u_2,...\in U$, set
    ${{l}}_n:=L_n^{(U,r)}(\{u_1,...,u_n\})$ leading to the subtree
    length vector $({{l}}_1, {{l}}_2,...)  \in\R_+^\N$. We define the
    map $\underline{\ell}^\varepsilon:\,\R_+^\N\to\R_+^\N$ given by
    \begin{equation}\label{vareps}
      \underline{\ell}^\varepsilon:\, ({{l}}_1, {{l}}_2,...)  \mapsto
      ({{l}}_1^\varepsilon, {{l}}_2^\varepsilon, ...)
    \end{equation}
    with ${{l}}^{\varepsilon}_1=0$ and ${{l}}^{\varepsilon}_2=0\vee
    ({{l}}_2-\varepsilon)$ and for $n\ge 3$, recursively,
    \begin{equation}\label{nextrec}
      {{l}}^\varepsilon_n :=
      {{l}}^\varepsilon_{n-1}+\big({{l}}_n-{{l}}_{n-1}-\tfrac{1}{2}\varepsilon\big)^+.
    \end{equation}

Moreover, set
\begin{align}\label{Aeps}
  A_{\varepsilon,n}:=\{({{l}}_1=0, {{l}}_2,...): {{l}}_i>{{l}}_{i-1}
  \text{ for exactly $N_\varepsilon-1$ different
    $i\in\{1,2,...,n\}$\}},
\end{align}
and we observe that
\begin{equation}\label{eq:qqq20a}
\begin{aligned}
  \xi(\smallu^\varepsilon) (\cdot) &=
  (\underline{\ell})_\ast\nu^{\smallu^{\varepsilon}}
  = \lim_{n\to\infty}(\underline{\ell}^\varepsilon_\ast
  \lambda)\big(\cdot|A_{\varepsilon,n}).
\end{aligned}
\end{equation}
Now, take $\smallu, \widetilde\smallu\in\xi^{-1}(\{\lambda\})$.
Observe that $\widetilde\smallu^\varepsilon\to\widetilde\smallu$ and
$\smallu^\varepsilon\to\smallu$ in the Gromov-weak topology, as
$\varepsilon>0$. Hence we are in a position to apply Steps~1 and~2 to
find that $\widetilde\smallu^\varepsilon \in \xi^{-1}
(\lim_{n\to\infty}\underline
\ell_\ast^\varepsilon\lambda(\cdot|A_{\varepsilon,n})) =
\{\smallu^\varepsilon\}$, for all $\varepsilon>0$.  This shows that
$\smallu = \lim_{\varepsilon\to 0}\smallu^\varepsilon =
\lim_{\varepsilon\to
  0}\widetilde\smallu^\varepsilon=\widetilde\smallu$. \sm

As for {\em continuity of $\xi$},  assume that
$(\smallu_k)_{k\in\N}$ is a sequence in $\mathbb{U}$ such that
$\smallu_k\to\smallu$, for some $\smallu\in\mathbb{U}$, in the
Gromov-weak topology, as $k\to\infty$.  Then by definition,
$\Phi(\smallu_k)\to\Phi(\smallu)$, for all $\Phi\in\Pi^0$, as
$k\to\infty$. In particular, since the map $\underline{\underline r}\mapsto
\ell_n(\underline{\underline r})$ is continuous as it is the minimum of finitely many
continuous functions, for all $n\in\N$,
$\langle\xi(\smallu_k),\psi\rangle\to\langle\xi(\smallu),\psi\rangle$,
for all $\psi\in{{\mathcal C}_b}(\R_+^\N)$, or equivalently,
$\xi(\smallu_k)\Rightarrow\xi(\smallu)$ in the weak topology on
${\mathcal M}_1(\R_+^\N)$, as $k\to\infty$.

In order to show \emph{continuity of $\xi^{-1}$}, we take $\lambda$,
$\lambda_1$, $\lambda_2$,... in $\xi(\mathbb{U})$
such that $\lambda_m\Rightarrow\lambda$, as $m\to\infty$.
We have to show that $\smallu_m:= \xi^{-1}(\lambda_m) \to
\xi^{-1}(\lambda) =: \smallu$ in the Gromov-weak topology, as
$m\to\infty$. For this, we need to show that
the three steps in the
proof of injectivity of $\xi$ hold under weak limits.

For Step~1 and~2, assume that $\smallu$ is finite with
$\#\mathrm{supp}(\mu)=N$. Then the same holds for all large $m\in\N$.
Define for all $m\in\N$ (based on $\smallu^m$) sets
$B^m\subseteq\R_+^\N$, minimal elements $\underline{l}^{\ast,m}\in
B^m$, $(d_n^{\ast,m})_{n\ge 1}$ and $\underline{\underline{r}}^{m}$ as
in (\ref{B}), (\ref{e:last}), (\ref{eq:qqq10}) and (\ref{eq:qqq12}),
respectively.  Then we can clearly recover that the mutual distances
$\underline{\underline{r}}$ in $\smallu$ as the limit of
$\underline{\underline{r}}^{m}$, as $m\to\infty$.  Moreover, note that
the set of functions \eqref{eq:qqq16} is not only separating, but also
convergence determining. Hence since all metric measure spaces are
finite, we find that $\smallu_m\to\smallu$.

For the general case considered in Step~3, recall the notions
$\smallu^\varepsilon$, $\underline\ell^\varepsilon$ and
$A_{\varepsilon, n}$ from \eqref{eq:gpwlt1}, (\ref{vareps}) and
(\ref{Aeps}). Note then that $\smallu_m\to\smallu$ as $m\to\infty$ if
and only if $\smallu_m^\varepsilon \rightarrow \smallu^\varepsilon$ as
$m\to\infty$ for all $\varepsilon>0$.  Moreover, for all
$\varepsilon>0$,
\begin{align}
  \smallu_m^\varepsilon & = \underline \ell _\ast
  \nu^{\smallu_m^\varepsilon}(\cdot) = \lim_{n\to\infty}
  \underline\ell^\varepsilon_\ast \lambda_m(\cdot|A_{\varepsilon, n})
  \to \lim_{n\to\infty} \underline\ell^\varepsilon_\ast
  \lambda(\cdot|A_{\varepsilon, n}) = \smallu^\varepsilon.
\end{align}
The interchange of limits is justified, because $
\underline\ell^\varepsilon_\ast \lambda_m(\cdot|A_{\varepsilon,
  n}) \Rightarrow   \underline\ell^\varepsilon_\ast \lambda(\cdot|A_{\varepsilon,
  n})$ as $m\to\infty$, if $n$ is large enough, and we
have shown continuity of $\xi^{-1}$.
\end{proof}\sm

\begin{proof}[Proof of Theorem~\ref{T:05}]
    (i) Since $\xi$ is bijective on $\xi(\mathbb U)$, it is a
  consequence of Theorem 3.2 in \cite{MR1637085} that the martingale
  problem for $(\xi_\ast\mathbf P(\mathbb U), \Omega^{\uparrow,\Xi},
  \Pi^{1,\Xi})$ is well-posed. Moreover, by construction,
  $(\xi(\mathcal U_t))_{t\geq 0}$ solves the martingale problem. In
  addition, since $\mathcal U$ has the Feller property and $\xi$ and
  also $\xi^{-1}$ (see Theorem \ref{T:04}) are continuous, $\Xi$ is
  Feller, too. The last assertion follows from the continuity of the
  sample paths of the tree-valued Fleming-Viot dynamics and the
  continuity of $\xi$.

  (ii) With $\underline\ell$ from \eqref{e:lunder},
  \begin{equation}\label{eq:genY}
    \begin{aligned}
      &\Omega^{\uparrow}(\Psi\circ\xi)\big({\smallu}\big)
      \\
      &= \langle \nu^\smallu, \langle\nabla \psi\circ\underline\ell,
      \underline{\underline 2}\rangle\rangle + \gamma\sum_{1\le k<l}
      \langle \nu^{\smallu}, \psi\circ\underline \ell
      \circ\theta_{k,l} - \psi\circ\underline\ell \Big\rangle \\ &=
      \sum_{n\ge 2}n \langle \nu^\smallu,
      \frac{\partial}{\partial\ell_n} (\psi\circ\underline\ell)\rangle
      + \gamma \sum_{n\ge 2} (n-1) \langle \nu^\smallu,
      \psi\circ\beta_{n-1}\circ\underline\ell
      -\psi\circ\underline\ell\rangle \\ & = \sum_{n\ge 2} n\langle
      \xi(\smallu), \frac{\partial}{\partial\ell_n}\psi\rangle +
      \gamma \sum_{n\ge 1} n\langle \xi(\smallu),
      \psi\circ\beta_{n}-\psi\big\rangle
    \end{aligned}
  \end{equation}
  and we are done.
\end{proof}

To prepare the proof of Corollary~\ref{Cor:04} we investigate for each
time $t\ge 0$ the {\em mean sample Laplace transform},
\begin{equation}\label{goben}
   g(t;\underline\sigma)
 :=
   \mathbf{E} \big[\Psi^{\underline\sigma}(\Xi_t)\big],
\end{equation}
of the subtree lengths distribution $\Xi_t$,
where for $\underline{\sigma}\in\R_+^{\N}$,
\begin{equation}\label{grev35}
   \Psi^{\underline\sigma}(\Xi)
 :=
   \int_{\R_+^\N}\Xi(\mathrm d\underline{{l}})
   \psi^{\underline\sigma}(\underline{{l}})
\end{equation}
with the test function
\begin{equation}\label{grev34}
   \psi^{\underline\sigma}(\underline{{l}})
 :=
   \exp(-\langle\underline\sigma,\underline {{l}}\rangle).
\end{equation}
As usual, $\langle\boldsymbol{\cdot},\boldsymbol{\cdot}\rangle$ denotes the inner product.

\begin{lemma}[ODE system for the mean sample Laplace transforms]
  For $\underline \sigma\in\mathbb R_+^{\mathbb N}$ having only
  finitely many non-zero entries, the functions $g(.;\underline
  \sigma)$ satisfy the following system of differential equations:
  \label{LCor:03}
  \begin{equation}\label{eq:treeL1}
    \frac{\mathrm d}{\mathrm dt}g(t;{\underline\sigma})
    =
    -\big(\sum_{k=2}^\infty k\sigma_k\big) g(t;{\underline
      \sigma})+\gamma\sum_{k=1}^\infty k\big(g(t;{\tau_k\underline
      \sigma})-g(t;{\underline\sigma})\big)
  \end{equation}
  with the {\em merging operator}
  \begin{equation}\label{eq:treeL1a}
    \tau_k:\,(\sigma_1,...,\sigma_{k-1},\sigma_k,\sigma_{k+1},\sigma_{k+2},...)
    \mapsto
    (\sigma_1,...,\sigma_{k-1},\sigma_k+\sigma_{k+1},\sigma_{k+2},...).
  \end{equation}
\end{lemma}\sm

\begin{proof} By standard arguments, $\Psi^{\underline \sigma}\in\Pi^\Xi$ and
\begin{equation}\label{ag5}
   \frac{\mathrm{d}}{\mathrm{d}t}g(t;\underline\sigma)
 =
   \mathbb{E}\big[\Omega^{\uparrow,\Xi}
   \Psi^{\underline\sigma}(\Xi_t)\big].
\end{equation}

Hence, inserting (\ref{eq:genY}), and using $\psi^{\underline
  \sigma}(\beta_k \underline {{l}}) = \psi^{\tau_k \underline
  \sigma}(\underline {{l}})$ for all $k=1,2,\ldots$, with $\beta_k$
from \eqref{eq:genY2} and $\tau_k$ from \eqref{eq:treeL1a}, we find
\begin{equation}\label{eq:step2a}
\begin{aligned}
   &\frac{\mathrm{d}}{\mathrm{d}t}g(t,\underline{\sigma})
  \\
 &=
   \mathbf{E} \Big[ -\int\Xi_t (\mathrm{d}\underline{{l}})\,
   \sum_{k=2}^\infty k
   \sigma_k\psi^{\underline{\sigma}}(\underline{{l}})+\gamma
   \int\Xi_t(\mathrm{d}\underline{{l}})\,
   \sum_{k=1}^\infty k\big(\psi^{\underline{\sigma}}(\beta_k\underline{{l}})-
   \psi^{\underline{\sigma}}(\underline{{l}})\big)\Big]
  \\
 &=
   -\big(\sum_{k=2}^\infty k
   \sigma_k\big)g(t,\underline{\sigma})+\gamma\sum_{k=1}^\infty k
   \big(g(t,\tau_k\underline{\sigma})-g(t,\underline{\sigma})\big),
\end{aligned}
\end{equation}
as claimed.
\end{proof}\sm

\begin{remark}\label{Rem:04} Recall, for each $n\in\N$,
the function $g^n$ from (\ref{grev40}).
For each $n\ge 2$ and $\sigma\ge 0$, applying \eqref{eq:treeL1} to
$\underline{\sigma}=(\sigma\delta_{k,n})_{k\ge 2}$
yields, setting $g^1(t;\sigma):=1$,
\begin{equation}\label{w:001b}
\begin{aligned}
  \frac{\mathrm d}{\mathrm dt}g^n(t;\sigma) &= -n\sigma g^n(t;\sigma)+
  \gamma \binom{n}{2}\big(g^{n-1}(t;\sigma)-g^n(t;\sigma)\big)
  \\
  &= \frac{\gamma}{2}n(n-1)g^{n-1}(t;\sigma)-\frac{\gamma}{2} n
  \big(\tfrac{2}{\gamma}\sigma+n-1\big)g^n(t;\sigma),
\end{aligned}
\end{equation}
i.e.,
\begin{equation}\label{w:001c}
\begin{aligned}
\frac{{\rm d}}{{\rm d}t} \big( g^2(t;\sigma), g^3(t;\sigma), \ldots\big)
 &=
   \frac{\gamma}{2}\Big[ A\big(\tfrac{2}{\gamma}\sigma\big)\big(g^2(t;\sigma),g^3(t;\sigma),...\big)^\top+b^\top\Big],
\end{aligned}
\end{equation}
where
\begin{equation}\label{e:bbb}
   b^\top
 :=
   (2,0,...)^\top
\end{equation}
and for $\widetilde{\sigma}\ge 0$ the matrix $A:=A(\widetilde{\sigma})$ is defined by
\begin{equation}\label{e:systemA}
   A_{k,l}
 :=
  \begin{cases}k(k-1), & \mbox{ if }k=l+1,
    \\
    -k(\widetilde\sigma+k-1), & \mbox{ if }k=l,
    \\
    0, & \mbox{else},
  \end{cases}
\end{equation}
for all $k,l\ge 2$.
\hfill$\qed$
\end{remark}\sm

The proof of Corollary~\ref{Cor:04} uses the following preparatory
lemma.
\begin{lemma}
Fix $\widetilde{\sigma}\ge 0$. \label{L:04}
Let $B=(B_{k,l})_{k,l\geq 2}$ and $B^{-1}=(B^{-1}_{k,l})_{k,l\geq 2}$ be matrices defined by
\begin{equation}\label{mpeq:B}
  B_{k,l}
  :=
  \frac{\tfrac{k!}{l!}{k-1\choose l-1}\Gamma(\widetilde\sigma+2l)}
  {\Gamma(\widetilde\sigma+k+l)},
  \quad\mbox{ and }\quad
  B^{-1}_{k,l}
  =
  \frac{(-1)^{k+l}\tfrac{k!}{l!}{k-1\choose l-1}\Gamma(\widetilde\sigma+k+l-1)}
  {\Gamma(\widetilde\sigma+2k-1)}.
\end{equation}
\begin{itemize}
\item[(i)]
The matrices $B$ and $B^{-1}$ are inverse to each other.
\item[(ii)] The matrix $A=A({\widetilde{\sigma}})
=(A_{k,l})_{k,l\ge 2}$ has eigenvalues
\begin{equation}\label{e:systemD}
   \lambda_k
 :=
   -k(\widetilde\sigma+k-1),\qquad k\ge 2.
\end{equation}
\item[(iii)] If $D=(\lambda_k\delta_{k,l})_{k,l\ge 2}$ then
\begin{equation}\label{mpe:Df}
   f(A)=B f(D) B^{-1}
\end{equation}
for all analytical functions $f:\R^{\N^2}\to\R^{\N^2}$. Specifically,
$A^{-1}=B D^{-1} B^{-1}$ and
$e^{At}=B e^{Dt}B^{-1}$ for all $t\ge 0$.
\item[(iv)] For $\widetilde{\sigma}>0$, let
  $A^{-1}(\widetilde{\sigma})=(A^{-1}_{k,l})_{k,l\ge 2}$ be given by
  $A^{-1}_{k,l} = 0$ for $k<l$ and
\begin{equation}\label{mpAinverse}
  A^{-1}_{k,l}
  :=
  -
  \frac{(k-1)!\Gamma(\widetilde{\sigma}+l-1)}{l!
    \Gamma(\widetilde{\sigma}+k)},\qquad  k\geq l.
\end{equation}
Then $A^{-1}$ and $A$ are inverse to each other.
\end{itemize}
\end{lemma}\sm

\begin{proof} First, we note that $A, A^{-1}, B, B^{-1}$ are lower
  triangular infinite matrices. This implies that the domain of the
  maps induced by these matrices is $\mathbb R^{\mathbb N}$. In
  particular, we do not have to distinguish between left- and
  right inverse matrices of $A$ and $B$.

  (i) We need to show that
  \begin{align}
    (B\cdot B^{-1})_{k,l} = \delta_{k,l}
  \end{align}
  for $k\geq l\ge 2$. This is clear in the case where $k\leq l$.  For
  $k> l\ge 2$, with constants $C$ changing from line to line, and
  using the abbreviations $\widehat{k}:=k-l$ and
  $\widehat{\sigma}:=\widetilde{\sigma}+2l-1$,
\begin{equation}\label{mpinverse00}
\begin{aligned}
   (B\cdot B^{-1})_{k,l}
 &=
   \sum_{m=l}^{k}B_{k,m}B^{-1}_{m,l}
  \\
 &=
   \sum_{m=l}^{k}\frac{\frac{k!}{m!}{k-1\choose m-1}
   \Gamma(\widetilde{\sigma}+2m)}{\Gamma(\widetilde{\sigma}+k+m)}\cdot
   \frac{(-1)^{m+l}\frac{m!}{l!}{m-1\choose l-1}
   \Gamma(\widetilde{\sigma}+m+l-1)}{\Gamma(\widetilde{\sigma}+2m-1)}
  \\
  &=
   C\sum_{m=l}^{k}(-1)^{m+l}\frac{(\widetilde{\sigma}+2m-1)\Gamma(\widetilde{\sigma}+m+l-1)}{(k-m)!(m-l)!\Gamma(\widetilde{\sigma}+k+m)}
  \\
 &=
   C\sum_{m=0}^{\widehat{k}}(-1)^{m}
   \frac{(\widehat{\sigma}+2m)\Gamma(\widehat{\sigma}+m)}
   {\Gamma(\widehat{k}-m+1)\Gamma(m+1)\Gamma(\widehat{\sigma}+\widehat{k}+m+1)}
  \\
 &=
   C\sum_{m=0}^{\widehat{k}}(-1)^{m}
   \frac{(\widehat{\sigma}+2m)\Gamma(\widehat{\sigma}+m)}{\Gamma(m+1)}\cdot\frac{\Gamma(\widehat{\sigma}+2\widehat{k}+1)}
   {\Gamma(\widehat{\sigma}+\widehat{k}+m+1)\Gamma(\widehat{k}-m+1)}
  \\
 &=
   0,
\end{aligned}
\end{equation}
where we have used that
\begin{equation}\label{mpalso}
  C\cdot\frac{(\widehat{\sigma}+2m)\Gamma(\widehat{\sigma}+m)}{\Gamma(m+1)}
  =
  \frac{\Gamma(\widehat{\sigma}+m+1)}{\Gamma(m+1)\Gamma(\widehat{\sigma}+1)}+
  \frac{\Gamma(\widehat{\sigma}+m)}{\Gamma(m)\Gamma(\widehat{\sigma}+1)}
\end{equation}
and then applied Formula~(5d) on page~10 in~\cite{Riordan1968}.\sm

(ii) Since $A$ is lower triangular, this is obvious. \sm

(iii) Note that
\begin{equation}\label{mpeigen2a}
   (A\cdot B)_{2,l}-\lambda_l B_{2,l}=0
\end{equation}
and
\begin{equation}\label{mpeigen2}
\begin{aligned}
   \lambda_l-\lambda_k
 &=
   \widetilde\sigma(k-l)+(k^2-k-l^2+l)
  \\
 &=
   (k-l)\big(\widetilde\sigma+k+l-1\big).
\end{aligned}
\end{equation}

Thus for all $k\ge 3$ and $l\ge 2$,
\begin{equation}\label{mpeq:eigen3}
\begin{aligned}
  B_{k,l}
 =
  \frac{k(k-1)}{(k-l)(\widetilde\sigma+k+l-1)}B_{k-1,l},
\end{aligned}
\end{equation}
and since $A_{k,k}=\lambda_k$,
\begin{equation}\label{mpeq:eigen1}
\begin{aligned}
  \big(A\cdot B\big)_{k,l}-\lambda_l B_{k,l} &=
  A_{k,k-1}B_{k-1,l}+(\lambda_k-\lambda_l)B_{k,l}
  \\
  &= \big(k(k-1)-k(k-1)\big)B_{k-1,l}
  \\
  &= 0,
\end{aligned}
\end{equation}
which proves that $B$ contains all
eigenvectors of $A$. Hence the claim follows by standard linear
algebra. \sm

(iv) It is clear that $(A\cdot A^{-1})_{k,k}=1$, while for $k\not = l$,
\begin{equation}\label{mpproofinv}
\begin{aligned}
  (A\cdot A^{-1})_{k,l} &= A_{k,k-1}\cdot A^{-1}_{k-1,l}+A_{k,k}\cdot
  A^{-1}_{k,l}
  \\
  &=
  k(k-1)\frac{(k-2)!\Gamma(\widetilde{\sigma}+l-1)}{l!\Gamma(\widetilde{\sigma}+k-1)}-k(\widetilde{\sigma}+k-1)
  \frac{(k-1)!\Gamma(\widetilde{\sigma}+l-1)}{l!\Gamma(\widetilde{\sigma}+k)}
  \\
  &= 0.
\end{aligned}
\end{equation}
\end{proof}\sm

\begin{proof}[Proof of Corollary~\ref{Cor:04}]
Fix $n\in\N$ and $\sigma\ge 0$. Put
\begin{equation}\label{qqq9}
   h^{\sigma,n}(t)
 :=
   g^{n}\big(\tfrac{2t}{\gamma};\sigma\big).
\end{equation}

By (\ref{w:001b}), the vector
%\begin{equation}\label{qqq10}
  $\underline h
 :=
   (h^{\sigma,2},h^{\sigma,3},...)^\top$
%\end{equation}
satisfies the linear system of ordinary differential
equations
\begin{equation}\label{eq:system}
  \frac{\mathrm{d}}{\mathrm{d}t}\underline
  h=A\underline h+b,
\end{equation}
or equivalently,
\begin{equation}\label{eq:system1}
   h(t)
 =
   e^{At}h(0)+e^{At}A^{-1}b-A^{-1}b,
\end{equation}
with $b=(2,0,0,...)^\top$ and
$A=(A_{k,l})_{k,l\geq 2}$ as defined in (\ref{e:systemA}).
Consequently, if $B$, $B^{-1}$ and $D$ are as in Lemma~\ref{L:04}, then
\begin{equation}\label{e:solu2}
  h(t)
  = - A^{-1} b + Be^{Dt}\big( B^{-1}h(0) + D^{-1}B^{-1}b\big).
\end{equation}
Combining (\ref{qqq9}) with (\ref{e:solu2}) yield the explicit
expressions given in (\ref{e:solu45}).

Finally, by \eqref{e:solu2},
\begin{equation}\label{???}
\begin{aligned}
  g^n\big(t;\sigma\big) &\tto
  -2\Big(A\big(\tfrac{2}{\gamma}\sigma\big)^{-1}\Big)_{n,2}
  \\
  &=
  \frac{\Gamma(n)\Gamma(\widetilde{\sigma}+1)}{\Gamma(n+\widetilde{\sigma})}
  \\
  &= \mathbf{E}\big[e^{-\sigma\sum_{k=2}^n{\mathcal E}^k}\big].
\end{aligned}
\end{equation}
\end{proof}\sm

{\sc Acknowledgement. } We thank David Aldous, Steve Evans, Patric
Gl\"ode, Pleuni Pennings and Sven Piotrowiak for helpful discussions.
We are particularly grateful to Steve Evans who provided the key
argument in the proof of Theorem~\ref{T:04}.

\newcommand{\etalchar}[1]{$^{#1}$}

%\bibliography{mp}
%\bibliographystyle{alpha}

\end{document}